\documentclass[11pt]{amsart}
\usepackage[T1]{fontenc}
\usepackage{lmodern,amsmath,amssymb,amsthm,microtype,mathtools,booktabs,longtable,array,enumitem,xurl}
\usepackage[margin=1in]{geometry}
\usepackage[colorlinks=true,linkcolor=blue,citecolor=blue,urlcolor=blue]{hyperref}
\newtheorem{theorem}{Theorem}[section]
\newtheorem{proposition}[theorem]{Proposition}
\newtheorem{lemma}[theorem]{Lemma}
\newtheorem{corollary}[theorem]{Corollary}
\theoremstyle{definition}

\theoremstyle{remark}
\newtheorem{remark}[theorem]{Remark}
\newcommand{\C}{\mathbb C}
\newcommand{\Q}{\mathbb Q}
\newcommand{\Z}{\mathbb Z}
\newcommand{\N}{\mathbb N}
\newcommand{\R}{\mathbb R}
\newcommand{\dd}{\mathrm d}

\newcommand{\lcm}{\operatorname{lcm}}

\newcommand{\lean}[1]{{\small\texttt{\nolinkurl{#1}}}}
\newcommand{\leanref}[2]{\par\smallskip\noindent\textit{Lean reference.}
  \lean{QuadraticPi.#1.}\par\nobreak\noindent\hspace*{1em}\lean{#2}.\par\smallskip}
\newcommand{\status}[1]{\par\smallskip\noindent\textbf{Formalization note.} #1\par\smallskip}
\numberwithin{equation}{section}

\newcommand{\ii}{\mathrm i}
\newcommand{\ceil}[1]{\left\lceil #1\right\rceil}
\newcommand{\floor}[1]{\left\lfloor #1\right\rfloor}

\newcommand{\del}{\Delta_H}
\newcommand{\Wbar}{\overline W_H}
\newcommand{\Ear}{E_{\mathrm{ar}}}
\newcommand{\Ean}{E_{\mathrm{an}}}
\newcommand{\BB}{\mathcal B_H}
\newcommand{\RR}{\mathcal R_H}
\newcommand{\CC}{\mathcal C_H}

\newtheorem{mainresult}{Theorem}

\hypersetup{pdftitle={Irrationality Exponents of Real CM Elliptic Periods and Powers of the Lemniscate Period},pdfauthor={Gregory V. Chudnovsky and David V. Chudnovsky}}
\title[CM periods and lemniscate powers]{Irrationality Exponents of Real CM Elliptic Periods\newline and Powers of the Lemniscate Period}
\author{Gregory V. Chudnovsky}
\author{David V. Chudnovsky}
\thanks{Prepared with OpenAI ChatGPT assistance in mathematical exploration, exposition, formalization development, and internal review.}
\date{October 20, 2026 -- consolidated research edition}
\begin{document}
\begin{abstract}
We give a determinant proof that every nonzero real period of a
nonsingular elliptic curve over $\mathbb Q$ with complex multiplication,
measured with a specified nonzero rational invariant differential, has
irrationality exponent two. The proof uses an integral origin chart,
matched local lengths, an actual compactification and weighted packet
map, and arithmetic and analytic estimates for the same selected minor
on a rectangular sublattice of the full CM period lattice.
For the integral-defined lemniscate period
$\Omega=2\int_0^1(1-t^4)^{-1/2}\,dt$, we also give the uniform bound
$\mu(\Omega^n)=O(n\log(2n))$ for integers $n\geq1$.
Its new arithmetic ingredient clears integration denominators prime by
prime while retaining the coefficient order; a one-root norm argument
and an analytic estimate at algebraic centers complete the power proof.
The two arguments are presented with their different charts and exact
dependencies. Lean verifies the new CM determinant argument and rational
normalization relative to explicit classical uniformization, CM-lifting
and theta inputs; the lemniscate-power endpoint is also verified. The
formalization record states the exact classical boundary and theorem
references.
\end{abstract}
\maketitle
\noindent\textit{Formalization record updated October 21, 2026; checkpoint 170.}
\par\noindent\textit{Submission preparation: October 21, 2026.}
\tableofcontents

\section{Introduction and main results}\label{sec:overview}
This paper brings together the real-CM-period argument of the
October 12 integrated manuscript and the degree-aware power argument
of October 18. The first part proves a sharp exponent for a class of
periods; the second proves a bound uniform in the positive integral
power of one precisely normalized period. Both arguments use a
nonzero interpolation minor and compare arithmetic and analytic
estimates for that identical minor.

For an irrational real number $\xi$, we use the ordinary irrationality
exponent
\begin{equation}\label{eq:overview-mu}
 \mu(\xi)=\sup\left\{\nu>0:
 \begin{array}{l}
 0<|\xi-p/q|<q^{-\nu}\text{ for infinitely many}\\
 \text{distinct reduced fractions }p/q,\ p\in\mathbb Z,\ q\in\mathbb N,\ q\geq2
 \end{array}\right\}.
\end{equation}
Dirichlet's theorem gives $\mu(\xi)\geq2$. The eventual lower bounds
below are stated for all integers $p$ and all sufficiently large
positive integer denominators $q$, whether or not $p/q$ is reduced.

\begin{mainresult}[Real CM periods]\label{main:cm}
Let $E/\mathbb Q$ be a nonsingular elliptic curve with complex
multiplication over $\mathbb C$, and let $\eta$ be a specified nonzero
invariant differential defined over $\mathbb Q$. Every nonzero real
period $\omega$ of $(E,\eta)$ satisfies $\mu(\omega)=2$.
\end{mainresult}
\leanref{EllipticCMPaperNormalization}{paper_period_main_of_noninteger_CM}
The explicit classical-input scope of this reference is stated in
Section~\ref{sec:classical-boundary}.
The proof is supplied in Part~\ref{part:cm}, concluding with
Theorem~\ref{thm:cm-main}. The specified differential is part of the
statement: rational changes of model and rational rescaling are checked
in Section~\ref{subsec:p1-normalization}. The ordinary exponent in
\eqref{eq:overview-mu} is applied only to real periods. Constants and
approximation thresholds may depend on the curve, differential, and
chosen period.

\begin{mainresult}[Powers of the lemniscate period]\label{main:powers}
Put
\[
 \Omega=2\int_0^1\frac{dt}{\sqrt{1-t^4}}.
\]
There is a constant $C>0$, depending only on this fixed normalization,
such that for every integer $n\geq1$ there is $Q(n)\geq2$ with
\[
 \left|\Omega^n-\frac pq\right|>q^{-C n\lceil\log(2n)\rceil}
 \qquad(p\in\mathbb Z,\ q\in\mathbb N,\ q\geq Q(n)).
\]
In particular, every $\Omega^n$ is irrational and
$\mu(\Omega^n)\leq C n\lceil\log(2n)\rceil=O(n\log(2n))$.
\end{mainresult}
\leanref{LemniscatePowerExponentConclusion}{lemniscate_power_main}
Part~\ref{part:powers} proves this statement, including an admissible
constant $C=6\cdot10^{10}G^2$ specified in
\eqref{pow:eq:mainC}. The constant is deliberately conservative;
the point is its independence of $n$. The sharp value at $n=1$ is
already supplied by the specialized lemniscate theorem
\cite{lemniscate}, and also follows from Theorem~\ref{main:cm}
after the actual period normalization has been identified.
Rational scaling gives the same exponent bounds for $(r\Omega)^n$
for each $r\in\mathbb Q^\times$, because $r^n$ is rational.
The power theorem here concerns the lemniscate period; the extension
to other CM lattices is a separate problem.

\subsection{Two charts for two different constructions}
Throughout, logarithms are natural, $\mathbb N$ includes zero, and
$[t^s u^\beta]F$ means the coefficient of the monomial
$t^s\prod_i u_i^{\beta_i}$ in $F$.
The symbols for a local variable, source weight, or cutoff are local
to the part in which they are introduced.

After rational normalization, Part~\ref{part:cm} uses
$y^2=x^3+Ax+B$ and $\eta_0=dx/(2y)$. Choosing an integer $a$ with
$a^3+Aa+B\ne0$, it uses the chart at the identity
\[
 u=(x-a)^{-1},\qquad t=y(x-a)^{-2},\qquad
 t^2=F(u),\qquad F'(0)=1,
\]
and the primitive $L'(t)=-1/F'(U(t))$. Its reduced source monomials
$u^k t^b Z^\alpha$ have elliptic weight $2w_0(k+2b)$.

Part~\ref{part:powers} uses the square lattice
$\Omega(\mathbb Z+i\mathbb Z)$, the equation $y^2=x^3-x$, and
the diagonal half-period $z_T=(1+i)\Omega/2$, at which $x=y=0$.
Here $t=y$, $x=x_0(t)$, and the primitive satisfies
$\ell'(t)=(3x_0(t)^2-1)^{-1}$.
The source monomials $x^a y^b X^\alpha$ have elliptic weight
$w_0(2a+3b)$. Both primitives have derivative $-1$ at their
respective base points, but they are different germs at different
points. The power proof applies the half-period EPA theorem from
\cite[Theorem~3.1]{lemniscate}; it does not substitute the
origin-polar packet map from Part~\ref{part:cm} for that theorem.

\subsection{The arithmetic improvement for powers}
An integration coefficient has the form $a_j/j$ with $a_j\in\mathbb Z$.
In a contribution of $r$ such coefficients to a coefficient of order
$k$, their indices satisfy $j_1+\cdots+j_r\leq k$.
For every prime $p$ this gives
\[
 \sum_{h=1}^r v_p(j_h)
 \leq\min\left\{r\lfloor\log_p U\rfloor,\ k/p\right\}.
\]
Retaining this constraint reduces the clearing cost to a quantity
proportional to $(1+\log(1/\theta))/v_0$, uniformly in the cutoff
once the additive weight is large enough. Lemmas
\ref{pow:lem:clearing} and \ref{pow:lem:Dsize} give the exact
integer denominator and its proof.
For an approximation $p/q$ to $\Omega^n$, one then uses the
single positive root $\omega=(p/q)^{1/n}$ and the extension
$\mathbb Q(i,\omega)/\mathbb Q(i)$ of degree at most $n$.
The extra radical used to clear denominators cancels from the
relative-norm inequality. An exact transport to entire theta-cleared
columns supplies the analytic bound at these algebraic centers.

\subsection{Proof dependencies and formal verification}
The general-CM argument uses classical complex uniformization,
qualitative period transcendence, local intersection multiplicity,
the ample-plus-nef criterion, and Serre vanishing in explicitly
stated forms. Its curve-specific local, geometric, and determinant
specializations are written out in Part~\ref{part:cm}.
The power proof inherits the lemniscate period normalization and
separated-weight interpolation theorem from \cite{lemniscate};
its clearing, norms, conjugate estimates, algebraic-center transport,
parameters, and contradiction are proved in Part~\ref{part:powers}.
Liu, Jiang and Zhang \cite{liujjiangzhang} obtain $\mu(\log r)=2$
for positive rational $r\ne1$ using related moving-center
interpolation, curve inequalities, and blowup arguments.

The written results and the formalization record have distinct scopes.
The new CM argument---including the actual origin EPA, the same-minor
arithmetic and analytic estimates, parameter selection and norm
contradiction---has now been checked in Lean. The rational model,
differential factor and full scaled-lattice adapters are checked as well.
The public theorem is
\lean{QuadraticPi.EllipticCMPaperNormalization.paper_period_main_of_noninteger_CM}.
It takes the actual normalized period lattice, the CM lattice action,
and the classical theta simple-zero input explicitly; it does not
construct complex uniformization and integration from an abstract curve.
The power endpoint is
\lean{QuadraticPi.LemniscatePowerExponentConclusion.lemniscate_power_main}.
Appendices~\ref{sec:ledger}--\ref{sec:new-formal-plan} give the exact
correspondence, retained evidence and classical-input boundary. The
present update changes exposition and status references, not the
mathematical theorem and proof environments. Quasi-periods remain
outside the results proved here.

\clearpage
\part{Real periods of CM elliptic curves}\label{part:cm}
\section{Statement, scope, and proof structure}
\label{sec:main}
We use the exponent defined in \eqref{eq:overview-mu}. Requiring
reduced fractions gives the same exponent for an irrational number.
The upper bound follows by excluding infinitely many exponent-$\nu$
approximants for every $\nu>2$; Dirichlet's theorem then gives equality.

\begin{theorem}[Real CM periods]
\label{thm:cm-main}
Let $E/\mathbb Q$ be a nonsingular elliptic curve with complex
multiplication over $\mathbb C$, and let $\eta$ be a nonzero invariant
differential defined over $\mathbb Q$. Every nonzero real period
$\omega\in\mathbb R$ of $(E,\eta)$ satisfies
\[
 \mu(\omega)=2.
\]
\end{theorem}

An elliptic curve ``over $\mathbb Z$'' here means an integral
Weierstrass model whose generic fiber is nonsingular over
$\mathbb Q$; good reduction at every prime is not assumed.
The assertion concerns real periods and ordinary rational
approximation. Arbitrary complex periods and arbitrary algebraic
rescalings of a chosen real period are not covered by the statement.
Constants and approximation thresholds may depend on the curve,
differential, and period.

The impetus was the separated-weight method of OpenAI's
September 24, 2026 preprint on $\mu(\pi)=2$ \cite{openai}.
Its exponent conclusion is not an assumption in this proof.
The ordinary polynomial argument needed here is proved in
Section~\ref{sec:acceptance-ordinary-fiber}. The earlier lemniscate
manuscript \cite{lemniscate} supplies the comparison construction;
the present proof uses the actual full elliptic lattice and a
two-point polar divisor, and checks those replacements explicitly.

Siegel's 1932 theorem already supplies the qualitative
transcendence needed in the CM case \cite{siegel}; it does not
supply the sharp upper exponent. In particular, his theorem that
at least one fundamental period is transcendental suffices here:
if a nonzero period $\omega$ were algebraic, then $\omega$ and
$\alpha\omega$, for a nonreal CM multiplier $\alpha$, would
generate a finite-index sublattice consisting of algebraic numbers.
Every period would then be algebraic, a contradiction. The CM
multiplier is algebraic because its lattice action is an integral
two-by-two matrix. We use this qualitative theorem only for
irrationality and the Dirichlet lower bound.

\subsection{How the proof is assembled}
The reduction to an integral short equation preserves the exponent
by rational scaling. Its origin chart has an integral coordinate
series and a logarithm whose degree-$n$ coefficient has denominator
dividing $n$. The polynomial frame in that chart controls both
weighted degree and contact.

A dimension count provides a nonzero auxiliary polynomial.
Excess contact on a curve would force enough derivatives to vanish
identically. A persistent component then gives one prime and one
localized length to which both Taylor and B\'ezout estimates apply.
Separated weights force constant elliptic projection; a direct
ordinary-fiber argument excludes the remaining excess. This proves
a uniform curve inequality.

The compactification, fixed center ideals, and blowup turn that
inequality into eventual surjectivity of the actual weighted packets.
CM provides centers $\omega(j_1+i\sqrt d\,j_2)$ on a rectangular
sublattice. Surjectivity gives a nonzero determinant using every
target row. Integral clearing bounds it from below; a full-lattice
sigma multiplier and a collision estimate bound that same minor
from above. The final parameter choice makes the two bounds
incompatible with infinitely many exponent-$\nu$ approximants,
for each $\nu>2$.

All proof ingredients used below are mathematical statements,
independent of the number of files that compile. Formal-status
paragraphs report what the saved Lean development actually covers.

\subsection{Rational scaling and the normalized equation}
For $r\in\mathbb Q^\times$ one has $\mu(r\xi)=\mu(\xi)$.
Writing $r=a/b$ with nonzero integers $a,b$, an approximation
$p/q$ to $r\xi$ gives $bp/(aq)$ as an approximation to $\xi$.
The denominator changes by at most the fixed factor $|a|$ and
the error by $|b/a|$. Such fixed factors are absorbed by any
strict decrease in an approximation exponent. Applying this
argument to $r^{-1}$ gives equality.

We shall work with
\begin{equation}\label{eq:cubic}
 E:y^2=x^3+Ax+B,\qquad A,B\in\mathbb Z,\qquad
 4A^3+27B^2\ne0,\qquad \eta_0=\frac{dx}{2y}.
\end{equation}
The next reduction checks the differential as well as the equation.

\subsection{A complete classical reduction to the normalized period problem}
\label{subsec:p1-normalization}

\begin{proposition}[Rational model and differential reduction]
\label{prop:p1-rational-model}
Let $E/\Q$ be a nonsingular elliptic curve and let $\eta$ be a
nonzero invariant differential defined over $\Q$. There is a rational
isomorphism to a short integral model
$E_1:Y^2=X^3+AX+B$, with $A,B\in\Z$, and a number
$r\in\Q^\times$ such that
$\eta=r\,\dd X/(2Y)$ under pullback. Its period set is consequently
$r$ times the period set of $(E_1,\dd X/(2Y))$.
Complex multiplication is preserved by this change of model.
\end{proposition}

\begin{proof}
Start with a rational Weierstrass equation
$y^2+a_1xy+a_3y=x^3+a_2x^2+a_4x+a_6$.
Set $y_1=y+(a_1x+a_3)/2$ and
$x_1=x+(a_2+a_1^2/4)/3$. Completing the square and removing the
quadratic term gives $y_1^2=x_1^3+A_0x_1+B_0$, with rational
$A_0,B_0$, and
\[
 \frac{\dd x}{2y+a_1x+a_3}=\frac{\dd x_1}{2y_1}.
\]
Choose a positive integer $d$ divisible by the denominators of
$A_0$ and $B_0$, and put $X=d^2x_1$, $Y=d^3y_1$.
Then $A=d^4A_0$, $B=d^6B_0$ are integers and
$\dd X/(2Y)=d^{-1}\dd x_1/(2y_1)$.
The rational invariant differentials form a one-dimensional
$\Q$-space, so $\eta=c\,\dd x/(2y+a_1x+a_3)$ for some
$c\in\Q^\times$; take $r=cd$.
Integration around corresponding closed cycles proves the assertion
about periods. Conjugation by the isomorphism identifies endomorphism
rings, which proves the last assertion.
\end{proof}

\begin{proposition}[The specified differential determines the lattice]
\label{prop:p1-period-lattice}
For the nonsingular complex curve $E:y^2=x^3+Ax+B$ and
$\eta_0=\dd x/(2y)$, its period group $\Lambda$ is a lattice and
there is an analytic group isomorphism
\[
 \C/\Lambda\longrightarrow E(\C),\qquad
 z\longmapsto(\wp_\Lambda(z),\wp_\Lambda'(z)/2)
\]
away from the origin, with the origin sent to $O$ and pullback
$\eta_0=\dd z$. Thus $g_2(\Lambda)=-4A$ and
$g_3(\Lambda)=-4B$. If $E$ has complex multiplication, some
$\alpha\in\C\setminus\R$ satisfies $\alpha\Lambda\subseteq\Lambda$.
\end{proposition}

\begin{proof}
We invoke classical complex uniformization with the specified
differential: integration modulo its periods identifies $E(\C)$
with $\C/\Lambda$. This is the standard theorem recorded in
\cite[Chapter III, Theorem 3.10 and Remark 3.11]{milne}.
Its inverse has pullback differential $\dd z$.
The usual Weierstrass parametrization with this differential is
$x=\wp_\Lambda$, $2y=\wp_\Lambda'$; substitution in the cubic
and comparison with
$\wp'^2=4\wp^3-g_2\wp-g_3$ gives the two invariant identities.
An endomorphism fixing the identity lifts to $z\mapsto\alpha z$
with $\alpha\Lambda\subseteq\Lambda$, by
\cite[Chapter III, Proposition 3.3]{milne}.
If such an $\alpha$ is real, multiplication by $\alpha$ has matrix
$\alpha I_2$ in any real lattice basis. Since this matrix has
integer entries, $\alpha\in\Z$. A CM endomorphism outside $\Z$
therefore gives the claimed nonreal multiplier.
\end{proof}

\status{Rational short-normal-form reduction reuses Mathlib's existing
variable-change theorem. Actual integral-model construction, discriminant
scaling and the invariant-differential coefficient factor are checked in
\lean{QuadraticPi.EllipticRationalIntegralModel.exists_integral_short_model}
and
\lean{QuadraticPi.EllipticRationalIntegralModel.invariant_differential_factor}.
The full lattice, its normalized invariant identities and the lifting of
CM endomorphisms are explicit classical inputs. Their transport and the
paper-facing period conclusion are checked in
\lean{QuadraticPi.EllipticCMPaperNormalization.paper_period_main_of_noninteger_CM}.
The classical existence and integration assertions themselves are not
claimed as internally proved Lean theorems.}

\begin{lemma}[Reduction to one positive real period]
\label{lem:p1-real-period}
If a lattice $\Lambda\subset\C$ contains a nonzero real number,
then $\Lambda\cap\R=\omega_+\Z$ for a unique $\omega_+>0$.
Consequently an exponent-two theorem for $\omega_+$ implies the
same theorem for every nonzero real member of $\Lambda$.
\end{lemma}

\begin{proof}
The intersection is a nonzero discrete subgroup of $\R$.
There is a smallest positive member: intersect with the compact
interval from zero to any fixed positive member and use discreteness
of a lattice to obtain a finite nonempty set of positive members.
Division by its minimum, with integer quotient and remainder in
$[0,\omega_+)$, proves the cyclic assertion. Every other nonzero
real period is a nonzero integer multiple of $\omega_+$, so rational
scaling invariance proves the consequence.
\end{proof}

For the eventual-lower-bound formulation used by this project,
rational scaling is already available as
\lean{QuadraticPi.EventualLowerBound.rat_mul} and
\lean{QuadraticPi.irrationalityExponent_rat_mul_eq_two} in
\lean{RationalScaling.lean}. The more general finite-spectrum
invariance theorem is
\lean{QuadraticPi.irrationalityExponent_rat_mobius_eq_of_bddAbove};
it explicitly assumes irrationality and a bounded approximation
spectrum. They are reused with the checked general-CM all-denominator
lower bound; no invariance under arbitrary algebraic scaling is used.

\section{An integral origin chart and its analytic logarithm}\label{sec:origin}

Fix integers $A,B,a$ and set $h(X)=X^3+AX+B$. The local statements in
this section do not require $h$ to be squarefree. Their application to
an elliptic curve does require the nonsingularity in \eqref{eq:cubic}.

\subsection{The quartic chart}

On the part of the cubic where $x-a\ne0$, put
\[
 u=\frac1{x-a},\qquad t=\frac{y}{(x-a)^2}.
\]
Multiplying the cubic equation by $u^4$ gives
\begin{equation}\label{eq:quartic}
 t^2=F(u),\qquad
 F(U)=U+3aU^2+(3a^2+A)U^3+h(a)U^4.
\end{equation}
The point at infinity corresponds locally to $u=t=0$. The decisive
facts are $F(0)=0$ and $F'(0)=1$.

\begin{proposition}[Integral local solution]\label{prop:integral}
There is a unique $U(T)\in\Z[[T]]$ with $U(0)=0$ and
$F(U(T))=T^2$. If
\[
 J(T)=F'(U(T)),\qquad V(T)=-J(T)^{-1},
\]
then $V\in\Z[[T]]$, $V(0)=-1$, and
\[
 J U'=2T,\qquad U'+2TV=0.
\]
The series $L\in\Q[[T]]$ defined by $L(0)=0$ and $L'=V$ satisfies
\begin{equation}\label{eq:denom}
 n[T^n]L\in\Z\quad(n\ge1).
\end{equation}
Consequently $\lcm(1,\ldots,N)$ clears every coefficient of $L$ through
degree $N$.
\end{proposition}

\begin{proof}
Since its linear coefficient is a unit, $F$ has a compositional inverse
$G\in\Z[[T]]$ with zero constant coefficient. Put $U(T)=G(T^2)$.
Both compositional inverse identities show existence and uniqueness.
The constant coefficient of $J$ is one, so its multiplicative inverse
is integral. Differentiating the identity $F(U)=T^2$ yields $JU'=2T$.
Multiplication by $-J^{-1}$ gives the second displayed identity without
division by the nonunit $T$.

Write $V=\sum_{j\ge0}v_jT^j$. Define
$[T^n]L=v_{n-1}/n$ for $n\ge1$. This proves \eqref{eq:denom}.
If $n\le N$, then $n$ divides $\lcm(1,\ldots,N)$, proving the common
denominator claim. The statement is a sufficient denominator bound;
it does not assert that each coefficient has reduced denominator $n$.
\end{proof}

\leanref{EllipticOriginLocal}{localU_unique}
The same module supplies
\lean{localU_equation}, \lean{localJacobian_mul_velocity},
\lean{localU_derivative_velocity}, \lean{ellipticLog_derivative},
\lean{ellipticLog_coeff_degree_cleared}, and
\lean{ellipticLog_trunc_lcm_cleared}.

The negative sign is dictated by the chosen differential:
\[
 x=a+u^{-1},\quad y=tu^{-2},\quad
 \frac{\dd x}{2y}=-\frac{\dd u}{2t}
                 =-\frac{\dd t}{F'(u)}.
\]
Thus $L'(0)=-1$, not $+1$. For a differential $c\eta_0$ with
$c\in\Q^\times$, the primitive is $cL$ and its fixed rational
denominator must also be included. The unqualified integrality in
\eqref{eq:denom} is specific to $\eta_0$.

\subsection{Convergence is proved, not inferred from coefficient integrality}

\begin{proposition}[Analytic realization]\label{prop:analytic}
There is an analytic germ $U(t)$ at zero satisfying
$U(0)=0$ and $F(U(t))=t^2$. Its Taylor series is precisely the complex
image of the integral series in Proposition~\ref{prop:integral}.
There is an analytic germ $L(t)$ with
\[
 L(0)=0,\qquad L'(t)=-1/F'(U(t)),
\]
whose Taylor series is the complex image of the rational series $L$.
In particular the degree and least-common-multiple denominator bounds
hold for the coefficients of this actual analytic germ.
\end{proposition}

\begin{proof}
The analytic inverse-function theorem at $F(0)=0$, $F'(0)=1$ gives a
local analytic inverse $G$; take $U(t)=G(t^2)$. The Taylor homomorphism
preserves the polynomial identity. Centered formal solutions are unique:
the difference of two such solutions factors from $F(W)-F(Z)$, and the
remaining factor has constant coefficient one. This identifies the Taylor
series with the integral formal solution.

The function $-1/F'(U(t))$ is analytic near zero. Integrating it on a
small disk gives a normalized analytic primitive $L$. Its Taylor series
has derivative $V$ and constant coefficient zero, which uniquely
determine its coefficients in characteristic zero. It is therefore the
previously constructed rational series.
\end{proof}

\leanref{EllipticOriginConvergence}{analyticU_hasIntegralPowerSeriesAt}
The analytic inverse is constructed in
\lean{EllipticOriginAnalytic}; the primitive and coefficient comparison
are \lean{EllipticOriginLogarithm.analyticLog_taylorSeries},
\lean{analyticLog_hasRationalPowerSeriesAt}, and
\lean{analyticLog_trunc_lcm_cleared}. These statements produce some
positive local radius. They do not claim an explicit numerical radius
or a global period identification.

\begin{proposition}[The actual normalized uniformization germ]
\label{prop:actual-origin-germ}
For the normalized uniformization of Proposition~\ref{prop:p1-period-lattice},
the meromorphic functions $u(z),t(z)$ defined by the chart extend
holomorphically at $z=0$ with $u(0)=t(0)=0$. They satisfy
$u(z)=U(t(z))$ and $L(t(z))=z$ near zero; hence
$t(z)=\tau(z)$, $u(z)=U(\tau(z))$, and $\tau'(0)=-1$.
\end{proposition}
\begin{proof}
The standard Laurent expansions $x(z)=z^{-2}+O(z^2)$ and
$y(z)=-z^{-3}+O(z)$ give $u(z)=z^2+O(z^4)$ and
$t(z)=-z+O(z^3)$. The chart equation and uniqueness of the centered
analytic solution give $u(z)=U(t(z))$. The specified differential
gives $dz=-dt/F'(u)$, so the derivative of $L(t(z))$ is one.
Both functions vanish at zero, proving the identity. The inverse
function theorem and $L'(0)=-1$ finish the proof.
\end{proof}

\section{The polar source, its frame, and its weighted packets}
\label{sec:polar-filtration-comparison}

The following statements close several local algebraic and formal
compatibility questions. They do not assert interpolation surjectivity.
Let $K$ be a field of characteristic zero and let $E/K$ be the smooth
projective curve with affine equation
\[
 y^2=h(x)=x^3+Ax+B,\qquad 4A^3+27B^2\ne0.
\]
Write $O$ for its point at infinity, choose $a\in K$ with $h(a)\ne0$,
and let $\mathcal P$ be the effective $K$-rational divisor of zeros
of $x-a$. Over an algebraic closure it is $Q_++Q_-$, where
$Q_\pm=(a,\pm\sqrt{h(a)})$. Define
\[
 u=(x-a)^{-1},\quad t=y(x-a)^{-2},\quad
 F(U)=U+3aU^2+(3a^2+A)U^3+h(a)U^4.
\]
Thus $t^2=F(u)$. The letters $t,u$ denote rational functions;
formal local variables below will be denoted $T,V_i$.

\begin{proposition}[Divisor identity and exact change of section space]
\label{prop:polar-space-isomorphism}
One has
\[
 \operatorname{div}(u)=2O-\mathcal P.
\]
For every integer $k\ge0$, multiplication by $u^k$ is a
$K$-linear isomorphism
\[
 L(2kO)\longrightarrow L(k\mathcal P),\qquad f\longmapsto u^kf.
\]
For $k\ge1$ an explicit basis of the latter space is
\[
 1,u,\ldots,u^k;\qquad t,tu,\ldots,tu^{k-2},
\]
with the second list empty if $k=1$. Its dimension is $2k$;
$L(0)$ has basis $1$.
\end{proposition}

\begin{proof}
The function $x$ has its unique pole at $O$, of order two.
At either $Q_\pm$, $y\ne0$, so the equation shows that $x-a$
is a local parameter: its zero is simple. Consequently
$\operatorname{div}(x-a)=\mathcal P-2O$. For any rational
function $f$,
\[
 \operatorname{div}(u^kf)+k\mathcal P
   =\operatorname{div}(f)+2kO.
\]
This proves the claimed isomorphism, including its inverse.

For completeness the basis can be obtained without a
Riemann--Roch dimension calculation. Every function with poles
only at $O$ is regular on the affine curve $E\setminus\{O\}$,
whose coordinate ring is $K[x,y]/(y^2-h(x))$. Division by the
monic polynomial in $y$ gives a unique expression
$p(x)+yq(x)$. At $O$, the monomials $x^i$ have pole orders
$2i$, while $yx^j$ have pole orders $2j+3$. These orders are
distinct, so no highest-pole cancellation is possible. Thus a
basis of $L(2kO)$ is
\[
 1,x,\ldots,x^k;\qquad y,xy,\ldots,x^{k-2}y.
\]
Its images under multiplication by $u^k$ are
\[
 u^{k-i}(1+au)^i\quad(0\le i\le k),\qquad
 tu^{k-j-2}(1+au)^j\quad(0\le j\le k-2).
\]
In increasing order of their smallest powers of $u$, the two
change-of-basis matrices are triangular with diagonal one.
They therefore span exactly the displayed polar basis and are
independent. The argument is over $K$ and does not require
either $Q_+$ or $Q_-$ to be individually $K$-rational.
\end{proof}

The divisor calculation also gives
$\operatorname{ord}_O(u)=2$ and $\operatorname{ord}_O(t)=1$.
At either point of $\mathcal P$, $u$ has a simple pole and $t$
has a double pole. These are exact orders, because $y(Q_\pm)\ne0$.
Hence the origin-regular filtration is the actual section
filtration $L(k\mathcal P)$, not just a convenient set of
polynomials with asserted pole bounds.

\subsection{The origin packet intertwines the invariant frame}
\label{sec:origin-packet-frame}

The derivation dual to $\eta_0=dx/(2y)$ is
\[
 \delta x=2y,\qquad\delta y=3x^2+A.
\]
Direct differentiation in the function field gives
\begin{equation}\label{eq:quartic-invariant-frame}
 \delta u=-2t,\qquad\delta t=-F'(u).
\end{equation}
Indeed the first identity is the derivative of $(x-a)^{-1}$.
For the second differentiate $y(x-a)^{-2}$, use $y^2=h(x)$,
and substitute $x=a+u^{-1}$. Equivalently differentiating
$t^2=F(u)$ gives the result after cancellation in the function
field. The resulting identities are regular identities on the
quartic affine chart. The derivation
$-2t\partial_u-F'(u)\partial_t$ kills $t^2-F(u)$, so it
descends to its coordinate ring.

Let $U(T)\in K[[T]]$ be the unique series with $U(0)=0$ and
$F(U(T))=T^2$. Its existence follows by composition inversion
from $F(0)=0$, $F'(0)=1$. Set
\[
 J(T)=F'(U(T)),\qquad
 L(0)=0,\quad L'(T)=-J(T)^{-1}.
\]
Characteristic zero permits termwise integration. In particular
$J(0)=1$, $L'(0)=-1$, and differentiation of the defining
equation gives $J(T)U'(T)=2T$.

Introduce additive coordinates $Z_1,\ldots,Z_m$ and derivations
\[
 D_0=\delta+\sum_{i=1}^m\partial_{Z_i},\qquad
 D_i=\partial_{Z_i}\quad(1\le i\le m).
\]
For $c=(c_1,\ldots,c_m)\in K^m$, define the actual formal
packet homomorphism
\[
 \Phi_c(P)=P\bigl(U(T),T,c_i+V_i+L(T)\bigr)
 \quad\hbox{in }K[[T,V_1,\ldots,V_m]].
\]
It is well defined on $K[u,t,Z]/(t^2-F(u))$.

\begin{proposition}[Exact packet and frame identities]
\label{prop:origin-packet-intertwines}
For every element $P$ of this coordinate ring,
\[
 \Phi_c(D_0P)=-J(T)\partial_T\Phi_c(P),\qquad
 \Phi_c(D_iP)=\partial_{V_i}\Phi_c(P).
\]
Let $T=\tau(z)$ be the compositional inverse of $z=L(T)$.
Then the normalized packet
\[
 \widehat\Phi_c(P)
   =P\bigl(U(\tau(z)),\tau(z),c_i+V_i+z\bigr)
\]
satisfies
\[
 \widehat\Phi_c(D_0P)=\partial_z\widehat\Phi_c(P),\qquad
 \widehat\Phi_c(D_iP)=\partial_{V_i}\widehat\Phi_c(P).
\]
\end{proposition}

\begin{proof}
Both sides of the first pair are derivations relative to $\Phi_c$,
so it suffices to check the algebra generators. For $u$ the
equality is $-JU'=-2T$; for $t$ it is $-J=-F'(U)$;
for every $Z_i$ it is $-JL'=1$. The additive identities follow
by direct differentiation. Since $L$ has zero constant term and
nonzero linear term it has a unique composition inverse, with
$\tau(0)=0$, $\tau'(0)=-1$. The formal chain rule gives
$\tau'(z)=-J(\tau(z))$. Substitute this identity into the
first pair to obtain the normalized pair. No global period
identification or CM hypothesis is used.
\end{proof}

\begin{proposition}[Strict weighted packets are unchanged by this
reparametrization]\label{prop:strict-jet-reparametrization}
Let $v_0,r_1,\ldots,r_m>0$ and $H>0$. In the formal series
ring let $\mathfrak j_H$ consist of the series all of whose
monomials have weight
$v_0s+\sum_i r_i\beta_i\ge H$.
The substitution $z=L(T)$, leaving every $V_i$ fixed, induces an
isomorphism of the finite quotients by $\mathfrak j_H$.
Consequently the ranks of any collection of packets
$\widehat\Phi_c(P)$ and $\Phi_c(P)$ agree after taking all
coefficients with the strict cutoff
$v_0s+\sum_i r_i\beta_i<H$.
\end{proposition}

\begin{proof}
The weight condition defines a monomial ideal, because all
variable weights are positive. Since $L(T)$ has order one,
every monomial in $L(T)^sV^\beta$ has $T$-degree at least $s$
and has exactly the same additive multi-index $\beta$.
Thus substitution preserves $\mathfrak j_H$. The inverse
substitution $T=\tau(z)$ has the same property, so the induced
maps are inverse. In the monomial basis of the finite quotient,
ordered by the $T$-degree within each fixed $\beta$, the matrix
is triangular with diagonal $(-1)^s$. It is invertible over
$K$. For several centers it acts separately on each center block.

The additive shift is already part of the packet:
$c_i+V_i+z$ becomes exactly $c_i+V_i+L(T)$ under this
substitution. We do not additionally replace $V_i$ by
$V_i+L(T)$ in an arbitrary target series. Such a separate
shear need not preserve an anisotropic weighted ideal when
$r_i>v_0$. This distinction is essential.
\end{proof}

More precisely, form a square matrix using \emph{all} rows in
the strict packet (at all chosen centers), and any one fixed
selection of the same number of source columns. Write $M_T$ and
$M_z$ for its matrices in the two parameters. The triangular
calculation proves
\[
 \det M_T=(-1)^{\sum_{\text{rows }(c,s,\beta)}s}\det M_z.
\]
Indeed the same block-triangular row operation, independent of
the selected columns, takes $M_z$ to $M_T$. Over $\mathbb C$
their determinant norms are therefore equal. Since the
additive multi-index $\beta$ is unchanged, determinant grade
sums are unchanged as well. Thus an already established
determinant-level clearing identity for these full row packets
transfers exactly, although individual entries of the row
operation may have denominators. This assertion need not hold
for an arbitrary subcollection of packet rows.

This establishes rank and same-column determinant invariance
between two parameters on the \emph{same} origin-regular
packet. It does not establish full row rank for either packet,
nor does it supply an initial arithmetic clearing estimate for
one of them.

\subsection{The exact derivative cost of the polar filtration}
\label{sec:polar-derivative-cost}

Give $u,t,Z_i$ source weights $2w_0,4w_0,w_i$, respectively,
where $w_0>0$ and $w_i\ge0$. Reduce representatives to degree
at most one in $t$ by $t^2=F(u)$. This reduction never increases
weight, because every term of $F(u)$ has weight at most $8w_0$.
The resulting source space at cutoff $H$ is exactly the span of
\[
 u^jt^bZ^\alpha,\qquad b\in\{0,1\},\qquad
 2w_0(j+2b)+\sum_iw_i\alpha_i\le H.
\]
Equivalently its coefficient of $Z^\alpha$ lies in
$L(k_\alpha\mathcal P)$, where
$k_\alpha=\lfloor(H-\sum_iw_i\alpha_i)/(2w_0)\rfloor$;
negative $k_\alpha$ contribute the zero space.

\begin{proposition}[Polar signed word bound]
\label{prop:polar-signed-word}
If $P$ has source cutoff $H$, a word containing $n_i$
occurrences of $D_i$ has source cutoff at most
\[
 H+2n_0w_0-\sum_{i=1}^m n_iw_i.
\]
In particular $\delta L(k\mathcal P)\subseteq
L((k+1)\mathcal P)$ for $k\ge0$.
\end{proposition}

\begin{proof}
The term $-2t\partial_u$ raises weight by at most
$4w_0-2w_0=2w_0$. Every term of $F'(u)$ has weight at most
$6w_0$, so $-F'(u)\partial_t$ raises weight by at most
$6w_0-4w_0=2w_0$. The additive summands in $D_0$ lower weight
by $w_i\ge0$, and $D_i$ lowers it by $w_i$. Reduction by the
relation does not increase weight. Induction on the ordered word
proves the bound, including arbitrary real cutoffs and the zero
polynomial. The section-space statement follows from the exact
basis and its weights. It also follows directly by differentiating
a local Laurent expansion: the vector field dual to the nowhere
vanishing invariant differential is regular at $Q_\pm$ and can
increase a pole order by at most one.
\end{proof}

With jet weights $(v_0,w_i/\theta)$ and $0<\theta<1$, the
signed source costs $(2w_0,-w_i)$ are bounded termwise by
$(2w_0/v_0)$ times the jet weights. Thus a word of total
nonnegative jet cost at most $R$ has source cutoff at most
$H+(2w_0/v_0)R$. These signed source costs must not replace
the nonnegative contact costs in persistence statements.

\paragraph{Why this chart requires its own interpolation proof.}
Multiplication by $u^k$ has a zero of order $2k$ at the origin;
it is not a unit operation on regular jets. Thus the section-space
isomorphism alone cannot transfer the old lemniscate EPA. The curve
inequality and globalization below use precisely the polar source
and derivative cost proved here. The local coefficient identities
have checked Lean counterparts in \lean{EllipticOriginLocal.lean}
and \lean{EllipticOriginParameter.lean}. The actual quartic frame,
compactification, weighted sections and packet map are now assembled
in the origin-quartic development, culminating in
\lean{QuadraticPi.EllipticOriginQuarticEPA.epa}.

\section{The local algebra behind the two length bounds}
\label{sec:p1-generic-local-lengths}

We give the generic argument before specializing it to the origin quartic.
The lower bound measures the Taylor coefficients forced to vanish by
derivatives. The upper bound measures the same local quotient by a
coordinate slice and a weighted B\'ezout inequality. Keeping the prime,
ideal and residue field fixed is essential to their later comparison.

\subsection{Local rings, normal directions and classical inputs}

Let $k$ be a field of characteristic zero, let
$R=k[X_1,\ldots,X_n]$, and let $Q$ be a prime ideal. Write
\[
 S=R_Q,\qquad \mathfrak m=QS,\qquad
 K=S/\mathfrak m=\kappa(Q),\qquad \rho:R\longrightarrow K.
\]
The field $K$ is canonically the fraction field of $R/Q$; all maps below
use this canonical identification. For an ideal $I\subseteq Q$, define
\[
 \ell_Q(I)=\operatorname{length}_S(S/IS)\in\mathbb N\cup\{\infty\}.
\]
Let $\overline X_i=\rho(X_i)$ and put
\[
 V_Q=\left\{v\in K^n:
       \sum_i v_i\rho(\partial_i p)=0\text{ for every }p\in Q\right\},
 \qquad N_Q^{\rm amb}=K^n/V_Q.
\]
Thus a set $A\subseteq\{1,\ldots,n\}$ is a \emph{coordinate normal
basis} if the classes of the coordinate vectors $e_i$, $i\in A$,
form a basis of $N_Q^{\rm amb}$.

We use the following classical commutative algebra and intersection
facts, in these precise forms.
\begin{enumerate}
\item Polynomial local rings over a field are regular and
Cohen--Macaulay. Quotienting a Cohen--Macaulay local ring by a
nonzerodivisor lowers its dimension by one and preserves the
Cohen--Macaulay property. A full system of parameters is a regular
sequence. We also use the finiteness of associated primes and the fact
that a finite-dimensional vector space over an infinite field is not a
finite union of proper linear subspaces. For the local-ring assertions,
see the Stacks Project, Section 10.104, especially Lemmas 10.104.2
and 10.104.7, \url{https://stacks.math.columbia.edu/tag/00N7}.
\item If $(A,\mathfrak m_A)\to(B,\mathfrak m_B)$ is a flat local
homomorphism, then
\begin{equation}\label{eq:p1-flat-local-length}
 \operatorname{length}_B(M\otimes_A B)
 =\operatorname{length}_A(M)
       \operatorname{length}_B(B/\mathfrak m_A B).
\end{equation}
For finite-length $M$ this follows by tensoring a composition series;
flatness preserves its exact sequences. Faithful flatness also preserves
arbitrarily long strict chains, which deals with infinite length.
This is Stacks, Lemma 10.52.13, tag 02M1,
\url{https://stacks.math.columbia.edu/tag/02M1}.
\item A finite separable field extension is finite \emph{\'etale}.
In particular, if $K/F$ is finite separable, $K\otimes_F K$ is a
product of fields and its factor selected by multiplication
$K\otimes_F K\to K$ is $K$ itself. See Stacks, Lemma 10.143.4,
\url{https://stacks.math.columbia.edu/tag/00U3}.
\item We use the isolated-point affine B\'ezout inequality:
if $g_1,\ldots,g_r\in E[Y_1,\ldots,Y_r]$ have an isolated common
zero $a\in E^r$, then
\begin{equation}\label{eq:p1-isolated-affine-bezout}
 \operatorname{length}_{E[Y]_{\mathfrak m_a}}
    E[Y]_{\mathfrak m_a}/(g_1,\ldots,g_r)
 \le\prod_{i=1}^r\deg g_i.
\end{equation}
Here the quotient is a local complete intersection: its $r$ equations
form a system of parameters in a regular local ring of dimension $r$.
The precise projective result supplying the inequality is that the sum
of the multiplicities at isolated intersections of globally generated
effective Cartier divisors on a complete variety is at most their
intersection number. Apply it to the homogenized hypersurfaces in
$\mathbb P^r$, whose intersection number is $\prod_i\deg g_i$.
See W.~Fulton, \emph{Introduction to Toric Varieties}, Annals of
Mathematics Studies 131, Princeton University Press, 1993,
Section 5.5, p.~123. The identification of the local
complete-intersection multiplicity with the displayed length is also
explained in Stacks, Section 43.16,
\url{https://stacks.math.columbia.edu/tag/0B01}.
This version permits positive-dimensional components elsewhere,
including at infinity. We do not assume the homogenized equations
intersect properly everywhere. For a non-algebraically-closed field $E$,
extension to an algebraic closure preserves the length at the rational
point $a$ by \eqref{eq:p1-flat-local-length}.
\end{enumerate}

These are the classical inputs. The reductions to their hypotheses and
all constants needed here are proved below.

\subsection{Ordered Taylor maps and the lower bound}

Let $D_i:R\to R$ be $k$-linear derivations, indexed by a finite set.
Assign nonnegative costs $\kappa_i$ and define the cost of a word by
\[
 c(i_1,\ldots,i_l)=\sum_{h=1}^l\kappa_{i_h},\qquad
 c(\varnothing)=0.
\]
The composition $D_\nu$ is taken in the written order, with
$D_\varnothing=\mathrm{id}$. The word cost is additive; it is not an
independent arbitrary function on words.

\begin{lemma}[Rectangular Taylor lower bound]
\label{lem:p1-generic-taylor-lower}
Let $I\subseteq Q$. Select derivations $D_{j_1},\ldots,D_{j_r}$
whose residue coefficient vectors
\[
 v_h=(\rho(D_{j_h}X_1),\ldots,\rho(D_{j_h}X_n))\in K^n
\]
have linearly independent classes in $N_Q^{\rm amb}$. Let $b_h\ge2$
be integers, and suppose
\[
 \sum_{h=1}^r(b_h-1)\kappa_{j_h}\le\delta,
 \qquad
 D_\nu p\in Q\quad(p\in I,\ c(\nu)\le\delta).
\]
Then
\[
 \prod_{h=1}^r b_h\le\ell_Q(I).
\]
\end{lemma}
\begin{proof}
For a single derivation $D$, the formal exponential
\[
 p\longmapsto\sum_{a\ge0}\frac{D^a p}{a!}T^a
\]
is a ring homomorphism $R\to R[[T]]$. Indeed, iterated Leibniz
gives $D^a(pq)=\sum_{b=0}^a\binom abD^b(p)D^{a-b}(q)$;
division by $a!$ is exactly the power-series multiplication rule.
Compose these homomorphisms coefficientwise in a fixed order and
then apply $\rho$. This gives a ring homomorphism
\[
 \tau:R\longrightarrow K[[T_1,\ldots,T_r]],\qquad
 [T^\alpha]\tau(p)=
 \frac{\rho(D_{j_r}^{\alpha_r}\cdots D_{j_1}^{\alpha_1}p)}
      {\alpha_1!\cdots\alpha_r!}.
\]
The displayed order is one fixed choice; no interchange of derivations
is used, so commutativity of the $D_i$ is unnecessary.
The constant coefficient of $\tau(p)$ is $\rho(p)$. Every $p\notin Q$
therefore maps to a unit, and $\tau$ extends to $S=R_Q$.

Let
\[
 B=K[[T_1,\ldots,T_r]]/(T_1^{b_1},\ldots,T_r^{b_r}),
 \qquad \mathfrak n=(T_1,\ldots,T_r)B.
\]
If $\alpha_h<b_h$, the derivative in the coefficient formula has
cost at most $\sum_h(b_h-1)\kappa_{j_h}$. Hence $\tau(I)$ is zero
in $B$, and the same holds after localization. We obtain a ring map
$\varphi:S/IS\to B$ inducing the identity of residue fields $K$.

We verify its first-order surjectivity. For $v\in K^n$, the rule
\[
 p\longmapsto\sum_i v_i\rho(\partial_i p)
\]
is a derivation from $R$ to its residue module $K$, and extends to
$S$ by the quotient rule. It vanishes on $\mathfrak m^2$, so its
restriction to $\mathfrak m$ defines a $K$-linear functional on
$\mathfrak m/\mathfrak m^2$. If a linear combination of these
functionals for $v_1,\ldots,v_r$ is zero, evaluate it on every
$p\in Q$. The corresponding combination of vectors belongs to
$V_Q$, and independence of their normal classes forces every
coefficient to be zero. Thus these $r$ cotangent functionals are
independent. By finite-dimensional linear algebra the map from
$\mathfrak m/\mathfrak m^2$ to $K^r$ given by their evaluations is
surjective. It is exactly the first-coefficient map of $\tau$.
Because $b_h\ge2$, all $r$ first-order variables survive in $B$;
therefore the map to $\mathfrak n/\mathfrak n^2$ is surjective.

Here is why residue and first-order surjectivity suffice. Lift any
desired residue in $K$, and lift elements of $\mathfrak m$ whose
images span $\mathfrak n/\mathfrak n^2$. Their products span each
$\mathfrak n^s/\mathfrak n^{s+1}$ over $K$. Coefficients in $K$
can be lifted using residue surjectivity; any higher-order terms of
such coefficient lifts do not affect this quotient. We can consequently
correct an error modulo $\mathfrak n^s$ to one modulo
$\mathfrak n^{s+1}$. The ideal $\mathfrak n$ is nilpotent, so
finitely many corrections give an exact preimage. This proves
$\varphi$ is surjective without assuming that a coefficient field
$K$ has been chosen inside $S$.

The monomials $T^\alpha$ with $0\le\alpha_h<b_h$ form a basis of
$B$ over $K$. Its composition factors are $K$, so its length as a
$B$-module, and as an $S$-module through the surjection, is
$\prod_h b_h$. The quotient $S/IS$ has at least this length.
\end{proof}

For positive costs, $r\ge1$ and $\delta>0$, the choice
\[
 b_h=\left\lceil\frac{\delta}{r\kappa_{j_h}}\right\rceil
\]
is within budget, since $\lceil x\rceil-1<x$ for $x>0$.
Whenever all these integers are at least two, the lemma gives
\begin{equation}\label{eq:p1-generic-taylor-real}
 \frac{\delta^r}{r^r\prod_h\kappa_{j_h}}\le\ell_Q(I).
\end{equation}
For every $1\le r\le m+1$ and every selected index set, the single
condition $\delta>(m+1)\max_i\kappa_i$ ensures these cutoffs are
at least two. The condition is required; point vanishing alone
cannot force an arbitrarily large local length.

\subsection{An exact coordinate slice at a prime}

Let $A$ be a coordinate normal basis and let $B_0$ be its complement.
To avoid confusing a variable with its value, denote the retained
slice variables by $Y_i$, $i\in A$. Define
\[
 \operatorname{sl}_A:R\longrightarrow K[Y_i:i\in A],\qquad
 X_i\longmapsto
 \begin{cases}Y_i,&i\in A,\\ \overline X_i,&i\in B_0,\end{cases}
\]
and let $\mathfrak n_A=(Y_i-\overline X_i:i\in A)$.

\begin{lemma}[Exact slice length]
\label{lem:p1-exact-coordinate-slice}
For every ideal $I\subseteq Q$,
\[
 \ell_Q(I)=
 \operatorname{length}_{K[Y]_{\mathfrak n_A}}
 \left(K[Y]_{\mathfrak n_A}/
                  \operatorname{sl}_A(I)K[Y]_{\mathfrak n_A}\right).
\]
If $Q$ is minimal over $I$, then $\mathfrak n_A$ is minimal over
the sliced ideal.
\end{lemma}
\begin{proof}
First we identify the complementary coordinates. A vector in $V_Q$
defines a derivation of $R/Q$ into $K$ and hence, by the quotient rule,
a $k$-derivation of $K$ into itself. Conversely every such derivation
gives its vector of values on $\overline X_i$. Thus
$V_Q=\operatorname{Der}_k(K,K)$ under this coordinate identification.
The normal-basis condition is the direct-sum statement
\[
 K^n=V_Q\oplus\operatorname{span}_K\{e_i:i\in A\}.
\]
Projection onto the $B_0$ coordinates therefore identifies
$V_Q$ with $K^{B_0}$. Dually, the differentials
$d\overline X_i$, $i\in B_0$, form a basis of $\Omega_{K/k}$.
In characteristic zero this says that the $\overline X_i$, $i\in B_0$,
are a transcendence basis of $K/k$. To recall the criterion, a
polynomial relation of least total degree would, on differentiation,
give a nontrivial relation among these differentials: some nonzero
partial derivative has smaller degree and cannot itself vanish on
the tuple. Conversely an algebraically independent tuple extends to
a transcendence basis; the remaining finite algebraic extension is
separable. Differentiating its minimal polynomials expresses all
remaining differentials in those of the transcendence basis. Conversely,
each partial derivation on the purely transcendental field extends
uniquely across a finite separable extension, by the same minimal-polynomial
formula with nonzero derivative. Thus those differentials are independent
and their number is $\operatorname{trdeg}_k K$. Applying this to the
already independent complementary coordinates proves the assertion.

Put $F=k(\overline X_i:i\in B_0)$. Then $K/F$ is finite separable.
Localizing the complementary polynomial variables identifies
\[
 S=R_Q\simeq F[X_i:i\in A]_{\mathfrak q},
\]
where $\mathfrak q$ is the kernel of evaluation $X_i\mapsto
\overline X_i$. Its residue field is $K$: the retained residue
coordinates generate $K$ over $F$, and finitely many algebraic
generators generate a field already as an algebra.

Extend the coefficient field from $F$ to $K$ and localize at the
point $Y_i=\overline X_i$. This gives a flat local map
\[
 S\longrightarrow T:=K[Y_i:i\in A]_{\mathfrak n_A}.
\]
Its closed fiber is the localization of $K\otimes_F K$ at the
kernel of multiplication. By separability this localization is
exactly $K$, so its length is one. Formula
\eqref{eq:p1-flat-local-length} proves the asserted equality for
the actual extended ideal $IT=\operatorname{sl}_A(I)T$.

If $Q$ is minimal over $I$, then $IS$ has radical $\mathfrak m$;
its quotient is a zero-dimensional Noetherian local ring, hence
Artinian. The same is true after this flat local extension, by the
length equality. No prime strictly below $\mathfrak n_A$ can then
contain the sliced ideal. This proves the minimality assertion.
\end{proof}

\subsection{A weighted point bound retaining one fixed equation}

For positive rational weights $\rho_1,\ldots,\rho_r$, write
$\deg_\rho p\le M$ to mean that every exponent $\alpha$ in the
support of $p$ satisfies $\sum_i\rho_i\alpha_i\le M$; the zero
polynomial satisfies every such bound.

\begin{lemma}[Weighted fixed-equation point bound]
\label{lem:p1-weighted-point-fixed}
Let $E$ be a field of characteristic zero, $a\in E^r$, and $r\ge1$.
Let $J=(f_j:j\in\mathcal I)\subset E[Y_1,\ldots,Y_r]$ and
assume the point ideal $\mathfrak m_a$ is minimal over $J$.
Suppose $0\ne h\in J$, $\deg_\rho h\le d_0$,
$\deg_\rho f_j\le M$ for every $j$, and $d_0,M\ge0$. Then
\[
 \operatorname{length}_{E[Y]_{\mathfrak m_a}}
        (E[Y]_{\mathfrak m_a}/J E[Y]_{\mathfrak m_a})
 \le\frac{d_0M^{r-1}}{\prod_{i=1}^r\rho_i}.
\]
The index set $\mathcal I$ need not be finite.
\end{lemma}
\begin{proof}
Choose a positive integer $q$ such that $a_i':=q\rho_i$ is a
positive integer for every $i$. Use the translated power substitution
\[
 Y_i=a_i+Z_i^{a_i'}.
\]
The polynomial ring $E[Z]$ is free over $E[Y]$ for this map, with
basis $Z^\beta$, $0\le\beta_i<a_i'$. Indeed translation reduces
the assertion to grouping each exponent modulo $a_i'$.
Its rank is $\prod_i a_i'$. Over the point $a$ the only prime in
the fiber is $(Z_1,\ldots,Z_r)$, since the fiber equations are
$Z_i^{a_i'}=0$. Consequently the induced local map
\[
 E[Y]_{\mathfrak m_a}\longrightarrow E[Z]_{(Z)}
\]
is finite free of this same rank. Its closed fiber has basis the
same monomials and length $\prod_i a_i'$. Applying
\eqref{eq:p1-flat-local-length} shows that substitution multiplies
the local quotient length by exactly $\prod_i a_i'$.

Denote substituted polynomials and ideals by primes. The origin is
an isolated point of $J'$, and $h'\ne0$, because the substitution
is injective. Expanding a substituted monomial only lowers its
exponents from the largest powers, so
\[
 \deg h'\le qd_0,\qquad \deg f_j'\le qM.
\]
We now retain $h'$ as the first parameter. The local ring
$E[Z]_{(Z)}$ is a regular local domain of dimension $r$; hence the
nonzero element $h'$ in its maximal ideal is a nonzerodivisor.
Inductively suppose a regular sequence starting with $h'$ has been
chosen and the remaining local quotient has positive dimension.
It is Cohen--Macaulay. None of its associated primes contains
the image of $J'$: that ideal has maximal radical, whereas these
associated primes are minimal and not maximal in positive dimension.
For each of the finitely many associated primes choose a generator
$f_j'$ outside it. A constant linear combination of this finite
collection can avoid every associated prime simultaneously, since
$E$ is infinite. It is a nonzerodivisor and still has degree at
most $qM$.

After $r-1$ steps we have polynomials $g_2,\ldots,g_r\in J'$
such that $h',g_2,\ldots,g_r$ form a full regular sequence in the
local ring, with $\deg g_i\le qM$. Thus the origin is isolated
for these $r$ equations. Their ideal is contained in $J'$, so the
length for $J'$ is at most the length for these equations.
The isolated-point inequality \eqref{eq:p1-isolated-affine-bezout}
therefore bounds it by
\[
 (qd_0)(qM)^{r-1}.
\]
Divide by $\prod_i a_i'=q^r\prod_i\rho_i$ to obtain the claim.
All selections use only finitely many of the original generators.
\end{proof}

\begin{theorem}[Generic prime-local fixed-hypersurface bound]
\label{thm:p1-generic-prime-upper}
Return to $R=k[X_1,\ldots,X_n]$ and a prime $Q$. Let $A$ be a
coordinate normal basis of cardinality $r\ge1$. Suppose
$Q$ is minimal over $I=(f_j:j\in\mathcal I)$ and $h\in I$ has
nonzero slice $\operatorname{sl}_A(h)$. For positive rational
weights $\rho_i$ and $d_0,M\ge0$, assume
\[
 \deg_\rho h\le d_0,\qquad \deg_\rho f_j\le M
 \quad\text{for all }j.
\]
Then $\ell_Q(I)$ is finite and
\begin{equation}\label{eq:p1-generic-prime-upper}
 \ell_Q(I)\le\frac{d_0M^{r-1}}{\prod_{i\in A}\rho_i}.
\end{equation}
\end{theorem}
\begin{proof}
Apply Lemma~\ref{lem:p1-exact-coordinate-slice}. Replacing a
complementary variable by a constant removes its positive contribution
to a monomial's weight, so slicing cannot increase support degree
in the retained variables. The sliced fixed equation is nonzero
by hypothesis and lies in the sliced ideal. Minimality of the
sliced point was proved in the same lemma. Lemma
\ref{lem:p1-weighted-point-fixed}, over the exact coefficient field
$K$, now gives \eqref{eq:p1-generic-prime-upper} for the same length.
\end{proof}

\begin{corollary}[Ordinary weighted prime-local bound]
\label{cor:p1-ordinary-prime-upper}
Let $Q$ be minimal over $I=(f_j:j\in\mathcal I)$ and let $A$ be
a coordinate normal basis of cardinality $r\ge1$. For positive rational
weights $\rho_i$, suppose $M\ge0$ and $\deg_\rho f_j\le M$ for
every $j$. Then
\[
 \ell_Q(I)\le\frac{M^r}{\prod_{i\in A}\rho_i}.
\]
\end{corollary}
\begin{proof}
The exact slice lemma makes the point $\mathfrak n_A$ an isolated
point of the sliced ideal in $r$ variables. At least one generator
has nonzero slice: otherwise the sliced ideal is zero, whose sole
minimal prime in the polynomial domain is $(0)$, whereas
$\mathfrak n_A\ne(0)$ because $r\ge1$. Choose such a generator
as $h$ and apply Theorem~\ref{thm:p1-generic-prime-upper} with
$d_0=M$.
\end{proof}

\subsection{The interface to the origin quartic}

For the following application set $R=\mathbb C[u,t,Z_1,\ldots,Z_m]$,
$f=t^2-H(u)$, and suppose $f\in Q$ and $J=H'(\overline u)\ne0$.
Let $T_Q$ denote the intrinsic tangent in the $(t,Z_i)$ coordinates.
The equation $-Jv_u+2\overline t\,v_t=0$ says that the ambient
tangent is the graph
\[
 w\longmapsto(2\overline t\,w_0/J,w),\qquad w\in T_Q.
\]
Consequently an intrinsic coordinate normal basis $S$ becomes an
ambient coordinate normal basis after adjoining $u$. Indeed the
quotient is identified with
\[
 K\oplus(K^{m+1}/T_Q),\qquad
 (v_u,w)\longmapsto(v_u-2\overline t\,w_0/J,[w]);
\]
the added $u$ vector has image $(1,0)$ and the remaining images
have the prescribed intrinsic normal basis as their second components.

The intrinsic frame $(-J,1,\ldots,1),e_1,\ldots,e_m$ lifts to
the actual residue coefficient vectors of
\[
 D_0=-2t\partial_u-H'(u)\partial_t+\sum_i\partial_{Z_i},
 \qquad D_i=\partial_{Z_i}\ (i>0).
\]
Thus Lemma~\ref{lem:p1-generic-taylor-lower} applies to its
independent intrinsic frame normals. Theorem
\ref{thm:p1-generic-prime-upper} applies to the augmented coordinate
normal basis and the same ideal. The following section verifies the
nonzero slice, source bounds and quartic factor $4$ explicitly.

\paragraph{Formalization status.}
The generic arguments have existing Lean counterparts:
\texttt{residueTaylorFin},
\lean{prime_rectangular_mapCotangent_surjective},
\texttt{rectangularTaylor\_length\_lower},
\lean{length_eq_sliceFamily_of_normalBasis}, and
\lean{weighted_prime_hypersurface_length_bound}.
Their names record the correspondence with the written proofs above.
The actual origin-quartic frame and common-prime comparison are now
checked in the quartic development, including
\lean{QuadraticPi.EllipticOriginQuarticPersistence.exists_bounded_persistent_comparison}
and
\lean{QuadraticPi.EllipticOriginQuarticCurveInequality.weighted_curve_inequality}.
The compiled EPA uses those actual interfaces rather than assuming a
curve-specific comparison or rank theorem.

\section{Direct local-length estimates in the origin quartic chart}
\label{sec:origin-quartic-length}

The preceding local argument can be applied directly to the chart used
by the actual origin power series.  This avoids proving an unnecessary
local-length transport through a birational change of coordinates.
There is, however, a change from the cubic factor $3$ to the quartic
factor $4$, and the source-degree increment of the invariant derivation
doubles.  Both changes must be included in the constants.

Let $H(U)\in\mathbf C[U]$ have degree exactly $4$, with leading
coefficient $h_4\ne0$.  Work in
\[
 R=\mathbf C[u,t,z_1,\ldots,z_m],\qquad f=t^2-H(u),
\]
with the actual polynomial frame
\[
 D_0=-2t\partial_u-H'(u)\partial_t+\sum_{i=1}^m\partial_{z_i},
 \qquad D_i=\partial_{z_i}\quad(i>0).
\]
The relation is preserved because
$D_0(t^2-H(u))=-2tH'(u)+2tH'(u)=0$.
Give the derivation $D_i$ a positive cost $\kappa_i$ and define
the cost of a word $\nu=(i_1,\ldots,i_r)$ additively by
$c(\nu)=\sum_{h=1}^r\kappa_{i_h}$, with empty-word cost zero.
Given a polynomial $P$, define
\[
 I_s=(f)+(D_\nu P:c(\nu)\leq s).
\]
Let $Q$ be a prime containing $(f)$ and use throughout its literal
residue field $K=\operatorname{Frac}(R/Q)$ and the length
$\ell_Q(I)=\operatorname{length}_{R_Q}(R_Q/IR_Q)$.

Assume $H'(\bar u)\ne0$ in $K$.  The ambient tangent space is the graph
over its intrinsic $(t,z_1,\ldots,z_m)$-coordinates with
\[
 v_u=\frac{2\bar t}{H'(\bar u)}v_t.
\]
The intrinsic frame directions are
$(-H'(\bar u),1,\ldots,1)$ and $e_i$ for $i>0$.
Their graph lifts are the actual residue coefficient vectors of the
displayed derivations, since
$(2\bar t/H'(\bar u))(-H'(\bar u))=-2\bar t$.
The normal quotient, coordinate normals, and frame normals are defined by this graph using the generic tangent and normal construction above.

\paragraph{Lower bound.}
For independent frame normal directions $j_1,\ldots,j_k$, integers
$b_h\geq2$, and a budget
$\sum_h(b_h-1)\kappa_{j_h}\leq\delta$, assume
$D_\nu p\in Q$ for every $p\in I$ and $c(\nu)\leq\delta$.
Then
\[
 \prod_h b_h\leq\ell_Q(I).
\]
The proof is exactly the ordered Taylor, localization, and cotangent
proof already given: neither the degree of the defining equation nor
the sign of the nonzero Jacobian enters it.  With ceiling rectangles,
provided all cutoffs are at least $2$, it gives
\[
 \frac{\delta^k}{k^k\prod_h\kappa_{j_h}}\leq\ell_Q(I).
\]
The actual quartic frame/graph and normal-basis specializations are
included in the checked quartic persistence and curve-inequality
assembly cited above.

\paragraph{Upper bound and the factor $4$.}
Use positive rational ambient weights
\[
 \rho_u=2w_0,\qquad\rho_t=4w_0,\qquad\rho_{z_i}=w_i,
\]
and intrinsic bookkeeping weights $r_0=w_0$, $r_i=w_i$ for $i>0$.
Suppose $Q$ is minimal over $I_s$, every generator has support degree
at most $M$, and $M\geq8w_0$.
If $S$ is a coordinate normal basis of size $k$, augment it by the
ambient $u$ coordinate.  The tangent graph shows that this is an
ambient normal basis.  The fixed equation survives the corresponding
coordinate slice: $u$ remains a variable, the sliced $t$ has total
degree at most $1$, and the nonzero term $-h_4u^4$ cannot cancel
against the degree at most $3$ of all other terms.

The weighted fixed-hypersurface theorem, whose interface is generic
in the fixed polynomial, therefore gives a finite integer
$L=\ell_Q(I_s)$ with
\[
 L\leq\frac{8w_0 M^k}{\prod_{v\in\{u\}\cup S}\rho_v}.
\]
Here the denominator is $8w_0\prod_{i\in S}r_i$ when $0\in S$
(the intrinsic $t$-coordinate is retained), and is
$2w_0\prod_{i\in S}r_i$ otherwise.  Thus
\begin{equation}
 L\prod_{i\in S}r_i\leq
 \begin{cases}M^k,&0\in S,\\4M^k,&0\notin S.\end{cases}
 \label{eq:origin-quartic-upper}
\end{equation}
In particular the uniform geometric factor is $4$, not $3$.
Neither smoothness nor irreducibility of $f$ is needed for this
particular implication, because the prime $Q$ and minimality over
$I_s$ are explicit assumptions.  Those properties may be needed when
constructing such a prime or proving projection rigidity.

\paragraph{Corrected growth and comparison constants.}
The signed source increments are $2w_0,-w_1,\ldots,-w_m$:
$t\partial_u$ has increment $4w_0-2w_0=2w_0$,
$H'(u)\partial_t$ has increment at most $6w_0-4w_0=2w_0$,
and the additive partials have negative increments.
For the costs $\kappa_0=v_0$, $\kappa_i=w_i/\theta$, define
\[
 \varepsilon_4=
 \min\left\{\sigma,\frac{\sigma v_0}{4(1+\sigma)w_0}\right\},
 \qquad g_4=1+\frac{2w_0}{v_0}\varepsilon_4,
 \qquad\delta=\frac{\varepsilon_4N}{m+3}.
\]
Then $\varepsilon_4>0$, $g_4<3/2$, and
\[
 g_4(1+\sigma)<1+3\sigma-\varepsilon_4.
\]
The scalar margin is proved explicitly in Section~\ref{sec:origin-auxiliary}.  If $P$ has source degree at most
$N$, $s\leq\varepsilon_4N$, and $N\geq8w_0$, all generators of $I_s$
have source degree at most $g_4N$.

Applying the lower and upper bounds to the same prime and ideal gives
for coordinate and frame normal bases $S,S'$ of size $k\geq1$
\[
 \prod_{i\in S}r_i\leq
 4g_4^k k^k\left(\frac{m+3}{\varepsilon_4}\right)^k
                      \prod_{j\in S'}\kappa_j.
\]
Put $A_n(s)=2(n+2)^{n+2}(1+(n+2)/s)^{n+2}$ for $s>0$.
A conservative uniform constant is
\[
 C_4=A_m(\sigma)+4\,2^{m+3}A_{m+1}(\varepsilon_4).
\]
For $k\leq m+1$ this bounds the displayed coefficient.  The same
center-independent numerical weight-separation theorem accepts any
positive comparison function, so it can be applied with $C_4$.
The full and fibre volume assumptions are unchanged because the
intrinsic weights remain $(w_0,w_1,\ldots,w_m)$; the new local
calculation by itself does not assert that these volume assumptions
suffice for the global quartic interpolation problem.

It is unsafe simply to reuse the old cubic $\varepsilon$ in a proof
requiring the strict derivative margin.  For example, if
$v_0=2(1+\sigma)w_0$, the old definition gives $\varepsilon=\sigma$,
and with the quartic degree increment one gets
\[
 \left(1+\frac{2w_0}{v_0}\varepsilon\right)(1+\sigma)
       =1+3\sigma-\varepsilon,
\]
an equality, not the needed strict inequality.  The smaller
$\varepsilon_4$ avoids this defect.

\paragraph{The actual origin polynomial and the remaining boundary.}
For a short Weierstrass equation $y^2=x^3+Ax+B$ and a chosen constant
$a$ the actual polynomial is
\[
 H(U)=U+3aU^2+(3a^2+A)U^3+(a^3+Aa+B)U^4.
\]
The existing module \texttt{EllipticOriginLocal.lean} proves the
corresponding rational-chart identity for
$u=1/(x-a)$ and $t=y/(x-a)^2$ and constructs the origin solution.
Choosing $a$ so that $a^3+Aa+B\ne0$ gives the required nonzero quartic
coefficient.  For integral $A,B$ such an integer $a$ exists, for example
among any four distinct integers since the cubic has at most three
roots.  At the origin, $H'(0)=1$, so the tangent graph is valid there.
For a general component prime the premise remains
$H'(\bar u)\ne0$; it is not supplied merely by $H'(0)=1$ unless
containment in the specified origin maximal ideal has also been shown.

The following rigidity, persistence, contact and auxiliary sections
construct precisely the hypotheses used here. The generic length
theorems have retained formal evidence, while the quartic adapters
and their actual common-prime application remain L1 obligations.

\subsection{A self-contained proof of the normal-basis dichotomy}

\begin{lemma}[The normal-basis dichotomy with its actual sign]
\label{lem:acceptance-normal-dichotomy}
Let $F$ be a field, let $V\subsetneq F^{m+1}$ be a proper linear
subspace, and let $q:F^{m+1}\to N=F^{m+1}/V$ be the quotient map.
Fix $J\in F^\times$, write $e_0,\ldots,e_m$ for the coordinate
basis, and put
\[
 b_0=-Je_0+\sum_{i=1}^m e_i,\qquad b_i=e_i\quad(i>0).
\]
Suppose there do not exist subsets $S,T\subseteq\{0,\ldots,m\}$
whose respective families $(q(e_j))_{j\in S}$ and
$(q(b_j))_{j\in T}$ are bases of $N$, together with an index
$i>0$, such that $i\in S\setminus T$ and
\[
 j\in S\Longleftrightarrow j\in T\qquad(j>i).
\]
Then either every $v\in V$ satisfies $v_0=0$, or there is a
fixed $i>0$ such that every $v\in V$ satisfies
\[
 v_i=-v_0/J.
\]
\end{lemma}

\begin{proof}
The families $(e_j)$ and $(b_j)$ are bases of $F^{m+1}$;
the second assertion uses $J\ne0$. Their images therefore span
$N$, and each spanning family contains a basis. Because $V$ is
proper, $q\ne0$.

If $q(e_j)=0$ for all $j>0$, then $q(e_0)\ne0$. For $v\in V$
the identity $0=q(v)=v_0q(e_0)$ proves $v_0=0$.

Otherwise let $i>0$ be the largest index with $q(e_i)\ne0$.
Extend this nonzero vector to a basis of $N$ selected from the
coordinate images, with index set $S$ containing $i$. For $j>i$
both $q(e_j)$ and $q(b_j)$ vanish. Such indices belong to neither
$S$ nor any frame-basis index set $T$, since a basis cannot
contain the zero vector. The exclusion hypothesis therefore
forces $i\in T$ for every frame basis $T$.

It follows that $q(b_i)$ is not in the span of the other frame
images. Indeed, if it were, those other images would still span
$N$ and a basis selected from them would omit $i$. Every vector
$v=(v_0,\ldots,v_m)$ has the exact frame expansion
\[
 v=-\frac{v_0}{J}b_0+
       \sum_{j=1}^m\left(v_j+\frac{v_0}{J}\right)b_j.
\]
If $v\in V$, applying $q$ and isolating the $i$th term shows that
$\left(v_i+v_0/J\right)q(b_i)$ lies in the span of the remaining
frame images. Their preceding independence property forces
$v_i+v_0/J=0$, which proves the second alternative.
\end{proof}

This lemma proves the generic linear-algebra step used by the
quartic rigidity argument. At an actual prime $Q$ of
$\mathbf C[u,t,Z_1,\ldots,Z_m]$, all objects are taken over the
single field $F=\operatorname{Frac}(R/Q)$, and $J=H'(\bar u)$.
The intrinsic coefficient vector of
$-2t\partial_u-H'(u)\partial_t+\sum_i\partial_{Z_i}$ is
precisely $b_0$ above. Testing the two alternatives on the vectors
of all $\mathbf C$-derivations gives respectively $d\bar t=0$
or $d\bar Z_i=-J^{-1}d\bar t$. These are exactly the inputs
to the two constant-projection arguments, with no sign change
or replacement of the residue field.


\section{Differential and projection rigidity in the actual origin chart}
\label{sec:origin-projection-rigidity}

This section concerns the quartic chart used by the integral origin
power series. It supplies its domain, finite-projection, proper-tangent,
and differential-rigidity arguments directly. These results apply to
every nonsingular short elliptic equation; complex multiplication is
not used here. They do not assert a non-CM irrationality bound or
global EPA.

Throughout let $C$ be a field of characteristic zero, put
\[
 h(X)=X^3+AX+B,\qquad 4A^3+27B^2\ne0,
 \qquad a\in C,\quad h(a)\ne0,
\]
and define the actual origin polynomial
\[
 H(U)=U+3aU^2+(3a^2+A)U^3+h(a)U^4.
\]
Its leading coefficient $h_4=h(a)$ is nonzero. Write
\[
 R=C[u,t,Z_1,\ldots,Z_m],\qquad
 f=t^2-H(u),\qquad S=R/(f),\qquad
 T=C[t,Z_1,\ldots,Z_m].
\]

\begin{lemma}[The actual quartic is squarefree]
\label{lem:origin-quartic-squarefree}
The degree-four polynomial $H$ has no repeated root over an
algebraic closure of $C$. In particular $H$ is not a square in
$C(U)$, or in $C(Z_1,\ldots,Z_m)(U)$.
\end{lemma}

\begin{proof}
The root $0$ is simple because $H'(0)=1$. For $r\ne0$, use
\[
 H(r)=r^4h(a+r^{-1}),\qquad
 H'(r)=4r^3h(a+r^{-1})-r^2h'(a+r^{-1}).
\]
If both $H(r)$ and $H'(r)$ vanished, $h$ and $h'$ would have
a common root. This is excluded by $4A^3+27B^2\ne0$:
$h(x)=h'(x)=0$ would give $A=-3x^2$, $B=2x^3$.
Finally the $U$-adic valuation of $H(U)$ is one, whereas the
valuation of any nonzero rational square is even. This argument
is unchanged on adjoining independent variables to the constants.
\end{proof}

\begin{proposition}[Domain, reduced nonvanishing, and finite projection]
\label{prop:origin-ring-finite}
The ring $S$ is a domain. Every nonzero polynomial of $t$-degree
at most one has nonzero image in $S$. The natural map $T\to S$
is injective, and $S$ is a finite integral $T$-algebra. Every
nonzero ideal of $S$ has nonzero contraction to $T$.
\end{proposition}

\begin{proof}
Over $C(Z_1,\ldots,Z_m)(u)$ the quadratic $t^2-H(u)$ is
irreducible, because a factorization would give a rational
square root of $H(u)$. Gauss's lemma makes $f$ irreducible in
the polynomial UFD $R$, so $(f)$ is prime. Since $f$ is monic
of $t$-degree two, division by it gives a unique remainder of
$t$-degree at most one, proving reduced nonvanishing.

Viewed as a polynomial in $u$ over $T$, the relation has degree
four and leading coefficient $-h_4\in C^\times$. A nonzero
multiple of it cannot be independent of $u$, proving injectivity
of $T\to S$. The class of $u$ satisfies the monic equation
\[
 u^4+\frac{3a^2+A}{h_4}u^3+\frac{3a}{h_4}u^2
       +\frac1{h_4}u-\frac1{h_4}t^2=0.
\]
Thus $1,u,u^2,u^3$ generate $S$ as a $T$-module.

If $0\ne b\in I\subseteq S$, choose a monic integral equation
for $b$ over $T$ of least degree. Its constant coefficient is
nonzero: otherwise cancellation of $b$ in the domain $S$ would
give a monic equation of smaller degree. The equation expresses
this nonzero constant coefficient as a multiple of $b$, hence
as an element of $I\cap T$.
\end{proof}

\subsection{The quartic differential has no algebraic primitive}

\begin{lemma}[Rational differential equation obstruction]
\label{lem:origin-rational-obstruction}
There is no $r\in C(U)$ satisfying
\[
 2H(U)r'(U)+H'(U)r(U)=1.
\]
\end{lemma}

\begin{proof}
Extend the constants to an algebraic closure. If $r$ has a pole
of order $k\ge1$ at $b$ with $H(b)\ne0$, the term $2Hr'$
has pole order $k+1$, which $H'r$ cannot cancel. If $H(b)=0$,
write $H=H_1(U-b)+\cdots$ and
$r=c(U-b)^{-k}+\cdots$. The coefficient of the pole of order
$k$ in $2Hr'+H'r$ is $(1-2k)H_1c\ne0$, since the root is
simple and the characteristic is zero. Consequently $r$ has no
finite pole and is a polynomial. For $r\ne0$ of degree $n$,
the left side has leading coefficient
$(2n+4)h_4\operatorname{lc}(r)$ at degree $n+3$, so it cannot
equal one. The zero polynomial is excluded as well.
\end{proof}

\begin{proposition}[Finite-extension nonexactness]
\label{prop:origin-finite-extension-nonexact}
Let $F/C(U)$ be a finite extension and let $t\in F$ satisfy
$t^2=H(U)$. Then no $z\in F$ has
\[
 dz=\frac{dU}{2t}\quad\hbox{or}\quad
 dz=-\frac{dU}{2t}
 \qquad\hbox{in }\Omega_{F/C}.
\]
\end{proposition}

\begin{proof}
Here $t\ne0$ because the nonzero rational function $H(U)$
remains nonzero in $F$. Characteristic zero gives a unique
extension $D$ of $d/dU$ to the finite separable extension $F$.
It suffices to exclude the positive sign, replacing $z$ by $-z$
for the other sign. The differential identity would give
$2tD(z)=1$. Differentiating $t^2=H$ then yields
\[
 2H D(tz)=H'tz+H.
\]
Let $n=[F:C(U)]$ and $M=\operatorname{Tr}_{F/C(U)}(tz)$.
Trace commutes with this derivation: in a normal closure trace
is the sum over embeddings, and uniqueness of the extended
derivation makes it commute with each embedding. Hence
\[
 2HM'=H'M+nH.
\]
The nonzero integer $n$ is invertible in $C$. Setting
$r=M/(nH)$ gives
\[
 2Hr'+H'r=(2HM'-H'M)/(nH)=1,
\]
contradicting Lemma~\ref{lem:origin-rational-obstruction}.
\end{proof}

This is also an exact adapter to the familiar cubic obstruction. In the function field put
\[
 x=a+u^{-1},\qquad y=tu^{-2}.
\]
Then $y^2=x^3+Ax+B$, $C(x)=C(u)$, and
\begin{equation}\label{eq:origin-differential-adapter}
 \frac{dx}{2y}=-\frac{du}{2t}.
\end{equation}
These divisions are made only when $u,t\ne0$ in the function
field. They are automatic when $u$ is transcendental, which is
the case used in the contradiction argument below. The formula
is not asserted by evaluating inverses at the origin point.

\begin{theorem}[Origin-quartic projection rigidity]
\label{thm:origin-projection-rigidity}
Let $F/C$ be an essentially finite-type field extension and let
$u,t,z\in F$ satisfy $t^2=H(u)$. Put $J=H'(u)$.
If $J\ne0$ and
\[
 dz=-J^{-1}dt\quad\hbox{in }\Omega_{F/C},
\]
then $u$ and $t$ are algebraic over $C$. Over algebraically
closed $C$ they belong to the image of $C$.
The same conclusion holds if $t\ne0$ and
$dz=-(2t)^{-1}du$.
\end{theorem}

\begin{proof}
Suppose $u$ were transcendental over $C$. Extend it to a finite
transcendence basis and let $C_1$ be generated over $C$ by the
other basis elements. Then $u$ is transcendental over $C_1$ and
$F/C_1(u)$ is finite. The coefficients and the nonzero
discriminant of the quartic remain unchanged under the injective
field extension. Also $t\ne0$, and differentiation of
$t^2=H(u)$ gives $2t\,dt=J\,du$. Thus
\[
 -J^{-1}dt=-(2t)^{-1}du.
\]
The natural map $\Omega_{F/C}\to\Omega_{F/C_1}$ carries any
asserted differential equality to this equality over $C_1$.
Proposition~\ref{prop:origin-finite-extension-nonexact}, with
$u$ as its indeterminate, contradicts it. Hence $u$ is
algebraic, and $t^2=H(u)$ makes $t$ algebraic as well. Over
algebraically closed constants both must be constant.
\end{proof}

\begin{corollary}[The critical locus is a constant projection]
\label{cor:origin-critical-locus}
In any extension field $F/C$, the equation $H'(u)=0$ implies
that $u$ and $t$, with $t^2=H(u)$, are algebraic over $C$.
The same conclusion holds if $t=0$. Over algebraically closed
$C$ these are constant projections. Neither assertion requires
finite type or an exact differential hypothesis.
\end{corollary}

\begin{proof}
The polynomial $H'$ has degree three and nonzero leading
coefficient $4h_4$. Thus $H'(u)=0$ is a nonzero polynomial
equation for $u$, and the equation for $t$ finishes the proof.
If $t=0$, use the nonzero polynomial equation $H(u)=0$ instead.
Squarefreeness also shows that $J=0$ and $t=0$ cannot hold
simultaneously. At the specific origin $u=t=0$, $J=1$.
\end{proof}

\subsection{Strict primes, proper tangents, and the sign in the frame}

Let $Q$ be a prime ideal of $R$ with $(f)\subsetneq Q$, and
use its literal field
\[
 F=\operatorname{Frac}(R/Q).
\]
Write bars for the images of coordinates and let
$J=H'(\bar u)$. The evaluation homomorphism $R\to F$ has
kernel exactly $Q$. When $J\ne0$, define
\[
 G(v_0,v_1,\ldots,v_m)
   =\left(\frac{2\bar t}{J}v_0,v_0,v_1,\ldots,v_m\right).
\]
The coordinate order on the right is $(u,t,Z_1,\ldots,Z_m)$.
Let $\mathcal T_Q\subset F^{m+2}$ be the polynomial tangent
space annihilating the evaluated gradient of every polynomial
in $Q$, and put $V=G^{-1}(\mathcal T_Q)$.
Differentiation of $f$ shows
\[
 \mathcal T_Q=G(V),\qquad
 Jv_u=2\bar t\,v_t\quad(v\in\mathcal T_Q).
\]

\begin{proposition}[Proper intrinsic tangent for an actual strict prime]
\label{prop:origin-proper-tangent}
Under these hypotheses $V\ne F^{m+1}$. In particular the
intrinsic normal quotient $N=F^{m+1}/V$ has positive dimension,
at most $m+1$.
\end{proposition}

\begin{proof}
Suppose $V$ were all of $F^{m+1}$. For $p\in Q\cap T$, its
$u$ partial derivative is zero. Testing its gradient on $G(e_i)$
therefore says that every partial derivative of $p$ in the
variables $t,Z_1,\ldots,Z_m$ lies again in $Q\cap T$.
This ideal of $T$ is proper. If it contained a nonzero element
of least total degree, all of its partial derivatives would
vanish, by minimality. In characteristic zero such a polynomial
is constant, a contradiction. Hence $Q\cap T=0$.
But $Q/(f)$ is a nonzero ideal of $S$ by strict containment.
Its contraction to the injective copy of $T$ is zero, contrary
to Proposition~\ref{prop:origin-ring-finite}.
\end{proof}

Every $C$-derivation $D:F\to F$ supplies a vector
\[
 (D\bar u,D\bar t,D\bar Z_1,\ldots,D\bar Z_m)
       \in\mathcal T_Q.
\]
Indeed applying $D$ to the zero evaluation of $p\in Q$ gives
the required gradient equation. Its intrinsic coordinates are
$(D\bar t,D\bar Z_1,\ldots,D\bar Z_m)\in V$.
This argument keeps one residue field and requires no auxiliary
identification with an independently chosen function-field type.

The actual quartic derivations are
\[
 D_0=-2t\partial_u-H'(u)\partial_t+
            \sum_{i=1}^m\partial_{Z_i},\qquad
 D_i=\partial_{Z_i}\quad(i>0).
\]
Their residue coefficient vectors are the graph lifts of
\[
 b_0=-Je_0+e_1+\cdots+e_m,\qquad b_i=e_i\quad(i>0).
\]
In particular the coefficient of $e_0$ is $-J$, not $J$.
The two lists $(e_i)$ and $(b_i)$ are bases before passing to
the normal quotient, since $J\ne0$.

Say that normal-basis exclusion holds if there are no index sets
$S_0,S_1$ whose respective coordinate and frame images are bases
of $N$, and no $i>0$, with
\[
 i\in S_0,\qquad i\notin S_1,\qquad
 (j\in S_0\Longleftrightarrow j\in S_1)\quad\hbox{for all }j>i.
\]
This is the exclusion hypothesis of the preceding generic dichotomy, which applies with scalar
$-J$: it gives either $v_0=0$ for every $v\in V$, or an
$i>0$ such that $v_i=-v_0/J$ for every $v\in V$.
No claim that exclusion is automatic is made here.

\begin{theorem}[Quartic normal exclusion forces constant projection]
\label{thm:origin-normal-exclusion-projection}
Assume additionally that $C$ is algebraically closed. For an
actual strict prime $Q$ as above, the projection
$(\bar u,\bar t)$ is constant if either $J=0$, or $J\ne0$
and normal-basis exclusion holds for its actual quotient $N$.
\end{theorem}

\begin{proof}
The critical case is Corollary~\ref{cor:origin-critical-locus}.
In the other case apply the stated dichotomy and test it on
the intrinsic vectors of all $C$-derivations of $F$.
In the first alternative every such derivation kills $\bar t$.
Thus $d\bar t=0$ in $\Omega_{F/C}$: the dual space of this
finite-dimensional $F$-vector space consists of these
derivations and separates points. In characteristic zero for
a finitely generated field extension, this means $\bar t$ is
algebraic over $C$; otherwise extend it to a transcendence basis
and extend its partial derivative to the finite separable
extension to obtain a derivation not killing it. Since $C$ is
algebraically closed, $\bar t$ is constant, and
$H(\bar u)=\bar t^2$ makes $\bar u$ constant as well.

In the second alternative the same dual-space argument gives
\[
 d\bar Z_i=-J^{-1}d\bar t.
\]
The extension $F/C$ is essentially finite type because it is
the fraction field of a quotient of a finite polynomial ring.
Theorem~\ref{thm:origin-projection-rigidity} now proves that
$\bar u,\bar t$ are constant.
\end{proof}

\paragraph{Origin primes versus arbitrary primes.}
For a center $c\in C^m$, let
$M_c=(u,t,Z_1-c_1,\ldots,Z_m-c_m)$.
Since $H'(u)\equiv1\pmod {M_c}$, every prime
$Q\subseteq M_c$ satisfies $H'(u)\notin Q$ and hence $J\ne0$
in its residue field. This is the precise extra containment
that permits use of the origin value $H'(0)=1$. For arbitrary
component primes no such conclusion follows from that value
alone; the separate critical-locus argument above is required.
Also $M_c$ has height $m+2$, while the global persistence route
may require an ambient prime of height at most $m+1$. Containment
in $M_c$ and finite-order origin contact alone do not supply
that stronger height bound.

\paragraph{Formal status.}
These direct quartic arguments are written proofs. Existing cubic
rigidity and tangent declarations provide reusable structure but are
not claimed to instantiate the quartic statements automatically.

\section{Origin-chart contact and actual persistent components}
\label{sec:origin-persistence}

We now connect the actual origin packet to the derivative ideals used
by the local-length estimates.  This connection proves contact
propagation and gives a precise conditional persistence theorem.  It
also identifies why vanishing of a finite packet, by itself, does not
supply the required positive-dimensional component.

\subsection{The literal origin coordinate ring and tangent quotient}

Keep the chosen CM curve and origin chart from the preceding sections;
no enlargement of the arithmetic target class is intended.  For the
algebraic calculation write
\[
 H(U)=U+3aU^2+(3a^2+A)U^3+(a^3+Aa+B)U^4,
 \qquad a^3+Aa+B\ne0,
\]
\[
 \mathcal A=\mathbf C[u,t,Z_1,\ldots,Z_m],\quad
 f=t^2-H(u),\quad J_0=(f),\quad S=\mathcal A/J_0.
\]
The ring $S$ is a domain.  Indeed $H(U)$ has a simple zero at $U=0$,
so it is not a square in $\mathbf C(Z_1,\ldots,Z_m)(U)$; hence
$t^2-H(u)$ is irreducible and generates a prime ideal of $\mathcal A$.
Division by its monic degree-two expression in $t$ shows that any
nonzero polynomial of $t$-degree at most one lies outside $J_0$.
These conclusions concern this literal relation ideal.

Proposition~\ref{prop:origin-ring-finite} also proves that $S$ is
integral over $\mathbf C[t,Z_1,\ldots,Z_m]$: divide
$H(u)-t^2=0$ by its nonzero leading coefficient to obtain a monic
equation for $u$.  Thus every nonzero ideal of $S$ has nonzero
contraction to this polynomial subring.  In characteristic zero,
Proposition~\ref{prop:origin-proper-tangent} consequently gives a
proper intrinsic tangent for every prime $Q$ strictly containing
$J_0$ on the nonzero-Jacobian chart. Concretely, all geometry is over the same field
\[
 K_Q=\operatorname{Frac}(\mathcal A/Q),
\]
the residue field of $\mathcal A_Q$, with the canonical map having kernel
exactly $Q$.  On $H'(\bar u)\ne0$ the tangent is the graph
\[
 (v_t,v_1,\ldots,v_m)\longmapsto
 \left(\frac{2\bar t}{H'(\bar u)}v_t,v_t,v_1,\ldots,v_m\right).
\]
The intrinsic tangent is a proper subspace of $K_Q^{m+1}$, so its
normal quotient has positive dimension.  No auxiliary field or
abstract prime replaces $Q$ in the local-length comparison.

For an origin center $c=(c_1,\ldots,c_m)$, define the actual maximal
ideal
\[
 M_c=(u,t,Z_1-c_1,\ldots,Z_m-c_m).
\]
Since $H'(0)=1$, every prime $Q\subseteq M_c$ satisfies
$H'(u)\notin Q$.  Thus every component below an actual origin-center
maximal ideal automatically belongs to the nonzero-Jacobian chart.
This observation discharges that chart condition once the stated
prime containment has been established.

\subsection{Exact propagation of weighted contact}

Use the actual frame
\[
 D_0=-2t\partial_u-H'(u)\partial_t+
           \sum_{i=1}^m\partial_{Z_i},\qquad
 D_i=\partial_{Z_i}\quad(i>0),
\]
and the actual packet
\[
 \Phi_c(P)=P\bigl(U(T),T,c_i+V_i+L(T)\bigr).
\]
Here $H(U(T))=T^2$, $U(0)=0$, $L(0)=0$, and
$L'(T)=-H'(U(T))^{-1}$.  Put $J(T)=H'(U(T))$, so $J(0)=1$.
For positive weights $\kappa_0,\ldots,\kappa_m$, let
$\mathfrak j_R$ denote the ideal of series whose monomials all have
weight at least $R$, with $T$ of weight $\kappa_0$ and $V_i$ of
weight $\kappa_i$.  For $R\leq0$ this is the entire series ring.
Thus $\Phi_c(P)\in\mathfrak j_R$ is exactly the assertion that every
coefficient of weight strictly less than $R$ vanishes.

For a word $\nu$ in the indices $0,\ldots,m$, let
$c(\nu)$ be the sum of its nonnegative direction costs
$\kappa_i$.  These costs are distinct from the signed source-degree
increments $(2w_0,-w_1,\ldots,-w_m)$.

\paragraph{Proposition (actual packet contact propagation).}
If $\Phi_c(P)\in\mathfrak j_R$, then for every ordered word $\nu$,
\begin{equation}
 \Phi_c(D_\nu P)\in\mathfrak j_{R-c(\nu)}.
 \label{eq:origin-word-contact}
\end{equation}
For the ideal
\[
 I_s=J_0+(D_\nu P:c(\nu)\leq s),\qquad s\geq0,
\]
this implies $\Phi_c(I_s)\subseteq\mathfrak j_{R-s}$.

\begin{proof}
Partial differentiation in $T$ lowers weighted order by at most
$\kappa_0$, and differentiation in $V_i$ lowers it by at most
$\kappa_i$.  Multiplication by any formal series cannot lower
weighted order, since all variable weights are positive.
The chain rule, also proved in Section~\ref{sec:origin-contact}, gives the actual
identities
\[
 \Phi_c(D_0P)=-J(T)\partial_T\Phi_c(P),\qquad
 \Phi_c(D_iP)=\partial_{V_i}\Phi_c(P).
\]
Consequently each letter loses at most its assigned cost.  Induction
on the ordered word proves \eqref{eq:origin-word-contact}, including
the extra derivatives of $J(T)$ that occur when $D_0$ is repeated;
the induction differentiates the current packet, rather than treating
$J(T)$ as a constant.  No separate commutativity hypothesis is used.

The fixed relation maps to zero under $\Phi_c$.  Every derivative
generator of $I_s$ maps into $\mathfrak j_{R-s}$, because its cost
is at most $s$.  The latter is an ideal, so multiplication by packet
images of coefficients and addition preserve the inclusion.
\end{proof}

In particular, for a budget $B<R$ all words with $c(\nu)\leq B$
satisfy
\[
 (D_\nu P)(0,0,c)=0,
 \qquad I_B\subseteq M_c.
\]
The strict inequality $B<R$ matters.  At $B=R$, the remaining
cutoff can be zero, so a constant coefficient need not vanish.
For several centers the conclusion is
$I_B\subseteq\bigcap_c M_c$, provided the same contact hypothesis
holds at every selected center.  This is a containment in a finite
intersection of maximal ideals, not a height bound on a desired
curve component.

\subsection{Persistence for the actual origin derivative ideals}

\paragraph{Theorem (origin-chart persistent component).}
Let the costs satisfy $\kappa_i\geq0$ and let $\delta\geq0$.
Suppose $P\notin J_0$, and let $\mathfrak p$ be a prime of the
literal polynomial ring $\mathcal A$ such that
\[
 J_0\subseteq\mathfrak p,\qquad
 \operatorname{ht}(\mathfrak p)\leq m+1,
\]
\begin{equation}
 D_\nu P\in\mathfrak p
       \quad\text{for all }c(\nu)\leq(m+3)\delta.
 \label{eq:origin-terminal-curve}
\end{equation}
Then there are an integer $r$ with $0\leq r<m+3$ and a prime $Q$
such that:
\begin{enumerate}
\item $Q$ is minimal over both $I_{r\delta}$ and
      $I_{(r+1)\delta}$;
\item $Q\subseteq\mathfrak p$, $J_0\subsetneq Q$, and
      $2\leq\operatorname{ht}(Q)\leq m+1$;
\item for every $g\in I_{r\delta}$ and every word $\nu$ with
      $c(\nu)\leq\delta$, one has $D_\nu g\in Q$.
\end{enumerate}
If $\mathfrak p\subseteq M_c$ for an actual origin center, then
$H'(\bar u)\ne0$ in $K_Q$.  The proper tangent and both local-length
inequalities can therefore be applied over this same field, with
the additional degree and numerical assumptions stated in the
preceding section.

\begin{proof}
Every $D_i$ annihilates $f$, so it preserves $J_0$ by Leibniz.
The same rule and nonnegative costs imply, for any word $\nu$,
\[
 D_\nu(I_s)\subseteq I_{s+c(\nu)}.
\]
For one letter, differentiate a finite expression in the generators:
derivatives of coefficient factors multiply existing generators,
while derivatives of derivative generators have the enlarged cost.
The fixed-ideal terms stay in $J_0$.  Induction proves the word
version.

The ideals $I_{j\delta}$ increase with $j$, and
\eqref{eq:origin-terminal-curve} puts
$I_{(m+3)\delta}$ inside $\mathfrak p$.  Choose a minimal prime of
this terminal ideal below $\mathfrak p$; then choose successively
minimal primes of the preceding ideals below the primes already
chosen.  This gives a chain
$Q_0\subseteq\cdots\subseteq Q_{m+3}\subseteq\mathfrak p$,
where $Q_j$ is minimal over $I_{j\delta}$.  Such choices exist
because $\mathcal A$ is Noetherian and every ideal under consideration is
contained in a prime.  If every adjacent containment were strict,
the chain would contradict the height bound $m+1$.
Thus $Q_r=Q_{r+1}=Q$ for some $r<m+3$.

The empty word gives $P\in I_{r\delta}\subseteq Q$, whereas
$P\notin J_0$, proving $J_0\subsetneq Q$.  The nonzero principal
prime $J_0$ has height one, so strict containment gives
$\operatorname{ht}(Q)\geq2$; the upper bound follows from
$Q\subseteq\mathfrak p$.  Finally, for $g\in I_{r\delta}$ and
$c(\nu)\leq\delta$,
\[
 D_\nu g\in I_{r\delta+c(\nu)}
       \subseteq I_{(r+1)\delta}\subseteq Q.
\]
The last assertion follows from
$Q\subseteq\mathfrak p\subseteq M_c$ and $H'(0)=1$.
\end{proof}

The checked origin specialization is
\lean{QuadraticPi.EllipticOriginQuarticPersistence.exists_persistent_origin_component}.
Its actual relation, domain, tangent and bounded-comparison interfaces
are used by the quartic curve inequality. The hypotheses on
$\mathfrak p$ are supplied in that application; a conditional
persistence statement alone would not supply them.

\subsection{Why finite point contact alone is insufficient}

The ideal $M_c$ has ambient height $m+2$, whereas the persistence
theorem used by the curve argument requires a containing prime of
height at most $m+1$.  There is no general reduction of height
from finite jet vanishing alone.  The following elementary example
is in the actual origin chart, rather than an unrelated polynomial
model.

Translate $Z_i$ by $c_i$ and write $z_i=Z_i-c_i$.  Give every target
variable weight one, and take
\[
 P=t^2+\sum_{i=1}^m z_i^2.
\]
Its packet is
$T^2+\sum_i(V_i+L(T))^2$, which has no coefficient of weight
less than two.  Hence all cost-at-most-one derivatives vanish at
the origin.  Nevertheless, after localization at $M_c$, the ideal
$I_1$ is exactly the maximal ideal.  Indeed
\[
 D_iP=2z_i,\qquad D_0P=-2tH'(u)+2\sum_i z_i.
\]
Since $H'(u)$ is a unit in $\mathcal A_{M_c}$, these imply $z_i,t\in I_1\mathcal A_{M_c}$.
The fixed relation then gives $H(u)\in I_1\mathcal A_{M_c}$; and
$H(u)=u(1+O(u))$, with the second factor a unit there, gives
$u\in I_1\mathcal A_{M_c}$.  Conversely all generators belong to $M_c$.
Thus the derivative ideal has no positive-dimensional component
through this center.  The example concerns only the insufficiency
of finite contact; it makes no claim against the additional global
degree/contact excess hypotheses of the interpolation argument.

Nor does taking the kernel of the full formal packet solve this
problem.  That kernel is exactly $J_0$, whereas $P\notin J_0$.
One elementary verification is to reduce a polynomial to
$p(u,Z)+tq(u,Z)$, remove the additive $L(T)$ shift by a formal
automorphism, and use that $U(T)$ is even of order two.  The first
term is even in $T$ and the second odd, so both vanish separately;
substitution of the nonzero series $U(T)$ into a polynomial is
injective. For a positive cutoff the truncated packet has instead an ideal kernel whose
radical is $M_c$; replacing it by a prime again leads to the same
height obstruction.

\subsection{Interface with the global curve-contact step}

The next implication has a precise destination.  An integral
algebraic curve in the origin affine chart gives a prime
$\mathfrak p\supseteq J_0$ with
$\dim(\mathcal A/\mathfrak p)=1$ and therefore
$\operatorname{ht}(\mathfrak p)=m+1$.
If the curve passes through one of the selected origin centers,
then $\mathfrak p\subseteq M_c$ and its Jacobian is nonzero in the
generic residue field.  The curve-terminal vanishing lemma must
then show that every derivative of $P$ through the required budget
vanishes identically on this curve, not just at its marked points.

In the notation of that lemma, let $R$ be the available weighted
contact cutoff, let $B<R$ be a terminal word budget, and let the
source cutoff be $N$.  The contact propagation proved above supplies
remaining cutoff $R-B$, while the actual quartic source-degree
bound supplies cutoff $N+(2w_0/v_0)B$ for every such derivative.
A proved curve-contact excess inequality of the indicated form
\[
 (R-B)\sum_{\xi}\mu_{\xi}
       >\bigl(N+(2w_0/v_0)B\bigr)D_{\rm curve}
\]
can therefore supply \eqref{eq:origin-terminal-curve} with
$B=(m+3)\delta$, after its branch and divisor hypotheses have
been verified.  This section uses that separate global lemma as an
interface and does not reprove its divisor estimate or assert the
existence of a curve satisfying the inequality.

The P1 status after these arguments is consequently specific:
the literal origin ideal, its stability, reduced-source strictness,
contact propagation, finite-height persistence, and the origin
nonzero-Jacobian condition have written proofs.  The supplied
prime-height and terminal-vanishing inputs match the local-length
calculation exactly. Sections~\ref{sec:origin-auxiliary} and
\ref{sec:origin-curve-assembly} supply the auxiliary construction and
curve-inequality assembly. No general sharp EPA theorem follows from
finite packets or persistence alone.

\section{From actual branch contact to a curve prime}
\label{sec:origin-contact}

This section supplies a conditional implication needed before persistence.
It does not construct the auxiliary polynomial or prove the excess-contact
inequality used as its hypothesis. The actual branch and curve-contact
adapters are now checked in the quartic place-contact, fixed-ideal-contact
and curve-auxiliary modules; the assembled curve inequality discharges
the required application hypotheses.

Work over $\mathbf C$, and put
\[
 H(u)=u+3au^2+(3a^2+A)u^3+(a^3+Aa+B)u^4,
 \qquad f=t^2-H(u).
\]
Assume $4A^3+27B^2\ne0$ and $a^3+Aa+B\ne0$.
The source weights of $(u,t,Z_1,\ldots,Z_m)$ are
$(2w_0,4w_0,w_1,\ldots,w_m)$, all positive. Let
$v=(v_0,w_1/\theta,\ldots,w_m/\theta)$, with $v_0>0$ and
$0<\theta<1$. The derivations are
\[
 D_0=-2t\partial_u-H'(u)\partial_t+\sum_i\partial_{Z_i},
 \qquad D_i=\partial_{Z_i}\quad(i>0).
\]
For a word $W$, let $c(W)$ be the sum of the corresponding $v_i$.

There is a unique $U(T)\in T\mathbf C[[T]]$ satisfying
$H(U(T))=T^2$, since $H'(0)=1$. Set $J(T)=H'(U(T))$ and define
$L(0)=0$, $L'(T)=-J(T)^{-1}$. In characteristic zero this uniquely
defines $L\in T\mathbf C[[T]]$. For a center $c\in\mathbf C^m$,
the actual packet is the homomorphism
\[
 \Phi_c(P)=P\bigl(U(T),T,c_1+V_1+L(T),\ldots,c_m+V_m+L(T)\bigr).
\]
The strict target quotient of cutoff $R$ is by the ideal
$\mathcal F_{\ge R}$ of series supported on monomials of $v$-weight
at least $R$. Thus a zero packet means membership in
$\mathcal F_{\ge R}$, with weak inequality at its boundary.

\begin{lemma}[Weighted derivative loss]
If $\Phi_c(P)\in\mathcal F_{\ge R}$, then for every word $W$,
\[
 \Phi_c(D_WP)\in\mathcal F_{\ge R-c(W)}.
\]
Here cutoffs at most zero impose no restriction.
\end{lemma}
\begin{proof}
Differentiating $H(U)=T^2$ gives $U'=2T/J$.
Consequently the chain rule gives
\[
 \Phi_c(D_0P)=-J(T)\partial_T\Phi_c(P),\qquad
 \Phi_c(D_iP)=\partial_{V_i}\Phi_c(P).
\]
A partial derivative lowers each surviving monomial weight by its
variable weight. Multiplication by $J(T)\in\mathbf C[[T]]$ cannot
lower weight, because its monomials have nonnegative weight.
Induction on the ordered word proves the assertion, including the
terms in which later derivatives act on an earlier factor $J$.
No replacement of $J$ by a constant is being made.
\end{proof}

Let $C$ now be a smooth connected projective complex curve, with
rational functions $u,t,z_1,\ldots,z_m$ satisfying $t^2=H(u)$.
Assume at least one of these functions is nonconstant. At a point $p$
where they are all regular with values $(0,0,c)$, choose a uniformizer
$s$. Their expansions lie in $\mathbf C[[s]]$, and
\[
 u(s)=U(t(s)),\qquad
 b_0(s)=t(s),\qquad b_i(s)=z_i(s)-c_i-L(t(s)).
\]
The first identity follows from formal uniqueness for $H(u)=t^2$.
All $b_i$ have positive order or are zero. They cannot all be zero:
otherwise $t=0$, $u=0$ and every $z_i=c_i$ as germs, hence as rational
functions, contradicting nonconstancy. Define
\[
 \mu_p=\min_{0\le i\le m}\frac{\operatorname{ord}_p(b_i)}{v_i},
 \qquad \operatorname{ord}_p(0)=+\infty.
\]
Thus $0<\mu_p<\infty$. We use $\mu_p$ here only for local weighted
contact; it is not an irrationality exponent.

\begin{lemma}[Actual branch substitution]
If $\Phi_c(P)\in\mathcal F_{\ge R}$ and $c(W)\le R$, then
\[
 \operatorname{ord}_p(D_WP(u,t,z))\ \ge\ (R-c(W))\mu_p.
\]
The value $+\infty$ is allowed on the left.
\end{lemma}
\begin{proof}
Substitute $(T,V)=(b_0,b_1,\ldots,b_m)$ in the preceding lemma.
Every variable is topologically nilpotent, so the substitution of
formal series is defined. It gives the actual evaluated polynomial,
because $u=U(t)$ and $z_i=c_i+b_i+L(t)$. Each monomial of weight
$q$ has order at least $q\mu_p$; zero factors only improve the bound.
The same lower bound holds for the sum by the valuation inequality.
For completeness, only finitely many monomials can contribute below
any fixed finite order: all nonzero substituted variables have positive
integer order, and monomials involving a zero variable vanish.
This justifies the inequality for the infinite series as well.
\end{proof}

We next specify the pole budget on this same curve. For a rational
function $g$, write $n_p(g)$ for its pole order, with $n_p(0)=0$.
Put
\[
 d_p=\max\left\{\frac{n_p(u)}{2w_0},
       \frac{n_p(t)}{4w_0},\frac{n_p(z_1)}{w_1},\ldots,
       \frac{n_p(z_m)}{w_m}\right\},
 \qquad d_C=\sum_{p\in C}d_p.
\]
The sum has finite support. Also $d_C>0$, since a nonconstant
rational function on a projective curve has a pole.
If a polynomial $P$ has source cutoff $N\ge0$, then
\begin{equation}\label{eq:contact-pole-budget}
 \deg(P(u,t,z))_\infty\le N d_C
 \quad\text{provided its evaluated value is nonzero}.
\end{equation}
Indeed, every source monomial has pole order at most $N d_p$ at $p$;
summing cannot increase the pole order. Summing these local bounds
proves the claim.

The signed derivative bound gives cutoff
$N+2n_0w_0-\sum_{i>0}n_iw_i$ for a word containing $n_i$
occurrences of direction $i$. In particular its cutoff is at most
\[
 N+\lambda c(W),\qquad \lambda=2w_0/v_0.
\]
This estimate applies to the actual polynomial before restriction to
$C$; it does not differentiate the functions on $C$.

\begin{theorem}[Excess contact supplies terminal vanishing]
\label{thm:contact-terminal}
Let $\mathcal S$ be a finite set of distinct points of $C$, each
centered at some $(0,0,c_j)$ as above. Let $P$ have source cutoff
$N\ge0$, and assume $\Phi_{c_j}(P)\in\mathcal F_{\ge R}$ for each
center used by $\mathcal S$. Suppose $0\le B<R$ and
\begin{equation}\label{eq:terminal-excess}
 (R-B)\sum_{p\in\mathcal S}\mu_p
       > (N+\lambda B)d_C.
\end{equation}
Then $D_WP(u,t,z)=0$ in $\mathbf C(C)$ for every word with
$c(W)\le B$.
\end{theorem}
\begin{proof}
Suppose the evaluated function $g=D_WP(u,t,z)$ is nonzero.
The branch lemma and $c(W)\le B$ give
\[
 \deg(g)_0\ge\sum_{p\in\mathcal S}\operatorname{ord}_p(g)
 \ge (R-B)\sum_{p\in\mathcal S}\mu_p.
\]
All displayed local orders are positive, so they belong to the zero
divisor. The pole estimate gives
$\deg(g)_\infty\le(N+\lambda B)d_C$.
A nonzero rational function on a smooth projective curve has zero
and pole divisors of equal degree. This contradicts
\eqref{eq:terminal-excess}.
\end{proof}

\begin{corollary}[The actual prime and its height]
Let $\mathfrak p$ be the kernel of
\[
 \mathbf C[u,t,Z_1,\ldots,Z_m]\longrightarrow\mathbf C(C),
 \qquad (u,t,Z_i)\longmapsto(u,t,z_i).
\]
Then $\mathfrak p$ is prime, contains $(f)$, and has height exactly
$m+1$. Under Theorem~\ref{thm:contact-terminal}, it contains every
$D_WP$ with $c(W)\le B$. At each selected centered point,
$\mathfrak p\subseteq\mathfrak m_c=(u,t,Z_1-c_1,\ldots,Z_m-c_m)$.
\end{corollary}
\begin{proof}
The target is a field, so the kernel is prime. Its quotient is the
algebra generated by the displayed rational functions. Its fraction
field has transcendence degree one: it is a subfield of $\mathbf C(C)$
and contains a nonconstant function. The height formula in a
polynomial ring in $m+2$ variables gives $m+1$.
The relation and terminal assertions follow by evaluation. For the
last assertion, an identically zero rational function has value zero
at every point where all the coordinate functions are regular.
\end{proof}

In particular, taking $B=(m+3)\delta$ provides precisely the prime
height and terminal-vanishing premises for persistence. If
$P\notin(f)$, strictness of a persistent prime follows as well.
For any such prime $Q\subseteq\mathfrak p\subseteq\mathfrak m_c$,
$H'(u)\notin Q$ since its value at the center is $H'(0)=1$.
This justifies use of the actual nonsingular tangent chart at that
prime, without inferring its Jacobian condition from the origin
identity alone.

\paragraph{What remains.}
This theorem converts an explicit inequality into the terminal
hypothesis. It does not itself establish that inequality for every
curve. Section~\ref{sec:origin-auxiliary} supplies the auxiliary and
parameter choice; Section~\ref{sec:origin-curve-assembly} combines
persistence, normal separation and the constant-fiber argument to prove
the uniform curve inequality. Globalization to surjectivity and the
analytic same-minor determinant bound are proved in Sections~\ref{sec:origin-geometric-assembly} and~\ref{sec:cm-same-minor-analytic}.

\section{An actual origin-polar auxiliary polynomial}
\label{sec:origin-auxiliary}

We now supply the auxiliary polynomial required by the preceding
curve-terminal argument.  This is a dimension argument for the
actual origin source and the actual finite coefficient packets.
It proves the existence of a nonzero kernel element, not
surjectivity of the packet map.

Let $d=m+1$, let $w_0,w_1,\ldots,w_m,v_0$ be positive rational
numbers, and let $\theta\in\mathbf Q$ satisfy $0<\theta<1$.
Keep the actual quartic relation
\[
 t^2=H(u),\qquad
 H(u)=u+3au^2+(3a^2+A)u^3+(a^3+Aa+B)u^4,
\]
with the chart conditions established earlier.  For a real cutoff
$N$, the origin-polar source space $\mathcal V_N$ has basis
\begin{equation}
 u^j t^b Z^\alpha,\qquad
 j\geq0,\quad b\in\{0,1\},\quad
 2w_0(j+2b)+\sum_{i=1}^m w_i\alpha_i\leq N.
 \label{eq:origin-auxiliary-basis}
\end{equation}
These monomials are linearly independent both as reduced
polynomials and in the actual coordinate ring: the relation is
monic of degree two in $t$.  Thus a nonzero vector in this source
is not merely a nonzero representative of the zero function.

The target variable weights are
\[
 v=(v_0,w_1/\theta,\ldots,w_m/\theta).
\]
For $K$ centers $c_1,\ldots,c_K\in\mathbf C^m$ and contact cutoff
$R>0$, let $\mathcal E_{N,R,c}$ be the linear map taking a source
polynomial to every coefficient of every actual packet
\[
 \Phi_{c_j}(P)=P\bigl(U(T),T,c_{ji}+V_i+L(T)\bigr)
\]
at monomials of $v$-weight strictly less than $R$.
Only finitely many coefficients occur.  The target has
$K$ copies of the corresponding weighted monomial index set.

\subsection{Exact counts, including the cutoff boundaries}

For positive weights $q=(q_0,\ldots,q_{d-1})$, define
\[
 A_{\leq}(X;q)=\#\{n\in\mathbf N^d:q\cdot n\leq X\},
 \qquad
 A_{<}(X;q)=\#\{n\in\mathbf N^d:q\cdot n<X\}.
\]
Put $A_{\leq}(X;q)=0$ for $X<0$.
The exact source count is
\begin{equation}
 s(N):=\dim_{\mathbf C}\mathcal V_N
  =A_{\leq}(N;(2w_0,w_1,\ldots,w_m))
   +A_{\leq}(N-4w_0;(2w_0,w_1,\ldots,w_m)).
 \label{eq:origin-source-exact-count}
\end{equation}
The two terms correspond to $b=0$ and $b=1$.
For example, the first is exactly
\[
 \sum_{\substack{\alpha\in\mathbf N^m\\\sum_iw_i\alpha_i\leq N}}
 \left(\left\lfloor
       \frac{N-\sum_iw_i\alpha_i}{2w_0}
       \right\rfloor+1\right),
\]
and the second has $N-4w_0$ in place of $N$ and is empty when that
cutoff is negative.  Equality at the source boundary is included.

The exact target count per center is
\begin{equation}
 j(R):=A_{<}(R;v)
 =\sum_{\substack{\beta\in\mathbf N^m\\
            \sum_i(w_i/\theta)\beta_i<R}}
 \left\lceil
       \frac{R-\sum_i(w_i/\theta)\beta_i}{v_0}
 \right\rceil.
 \label{eq:origin-target-exact-count}
\end{equation}
The ceiling is correct even when its argument is an integer:
for $x>0$, the number of integers $s\geq0$ with $s<x$ is
$\lceil x\rceil$.  In particular no coefficient on the strict
target boundary is accidentally imposed as an extra equation.

For completeness, these counts have elementary estimates that
are sufficient here.  Write $Q=\sum_iq_i$.  For $X>0$,
\begin{equation}
 \frac{X^d}{d!\prod_iq_i}
 \leq A_{<}(X;q)\leq A_{\leq}(X;q)
 \leq\frac{(X+Q)^d}{d!\prod_iq_i}.
 \label{eq:origin-simplex-count-sandwich}
\end{equation}
To prove this, attach the half-open unit cube
$n+[0,1)^d$ to each counted lattice point.  The union of the cubes
for the strict count contains the positive simplex
$\{x_i\geq0:q\cdot x<X\}$, because coordinatewise rounding down
does not increase weighted sum.  The cubes for the weak count lie
inside $\{x_i\geq0:q\cdot x<X+Q\}$.  Cubes have disjoint
interiors and unit volume.  The latter simplexes have volume
$X^d/(d!\prod_iq_i)$ and $(X+Q)^d/(d!\prod_iq_i)$,
respectively.  Boundary hyperplanes have volume zero.  This proves
\eqref{eq:origin-simplex-count-sandwich} without assuming that
$X$ or the weights are integral.

It follows from \eqref{eq:origin-source-exact-count} that
\begin{equation}
 s(N)=\frac{N^d}{d!\,w_0\prod_{i=1}^m w_i}
             +O(N^{d-1}),
 \label{eq:origin-source-density}
\end{equation}
whereas for any fixed $\rho>0$,
\begin{equation}
 K j(\rho N)=
 \frac{K\theta^m\rho^dN^d}
      {d!\,v_0\prod_{i=1}^m w_i}+O(N^{d-1}).
 \label{eq:origin-target-density}
\end{equation}
All constants in these estimates depend only on the fixed weights,
$m,K,\rho$, not on the centers or on the coefficients of their
packet maps.  In the source density the two values of $b$ cancel
the factor $2$ from the weight $2w_0$ of $u$.  The correct leading
elliptic density is therefore $1/w_0$, despite the derivation's
larger source cost $2w_0$.

\subsection{Auxiliary existence uniform in the centers}

\begin{theorem}[Actual origin-polar auxiliary]
\label{thm:origin-polar-auxiliary}
Suppose
\begin{equation}
 K\frac{w_0}{v_0}\theta^m\rho^{m+1}<1,
 \qquad \rho>0.
 \label{eq:origin-auxiliary-volume}
\end{equation}
There is a real $N_0$ such that, for every real $N\geq N_0$ and
every list of $K$ centers in $\mathbf C^m$, there is a polynomial
$P\in\mathcal V_N$ with $P\ne0$, $P\notin(t^2-H(u))$, and
\begin{equation}
 \Phi_{c_j}(P)\in\mathfrak j_{\rho N}
       \qquad(1\leq j\leq K).
 \label{eq:origin-auxiliary-contact}
\end{equation}
Here $\mathfrak j_{\rho N}$ is the actual weighted ideal of series
whose coefficients vanish at every weight strictly less than
$\rho N$.  The threshold $N_0$ is independent of the centers.
No distinctness assumption is needed for this existence theorem.
\end{theorem}

\begin{proof}
The strict inequality \eqref{eq:origin-auxiliary-volume} says that
the leading coefficient in \eqref{eq:origin-source-density} is
strictly larger than that in \eqref{eq:origin-target-density}.
Therefore $s(N)>Kj(\rho N)$ for all sufficiently large $N$.
This assertion is about fixed finite index sets and is independent
of any evaluation map.  For every choice of centers, the linear
map $\mathcal E_{N,\rho N,c}$ consequently has nonzero kernel.
Choose a nonzero kernel vector and use the basis
\eqref{eq:origin-auxiliary-basis} to obtain $P$.  Its kernel property
is precisely \eqref{eq:origin-auxiliary-contact}.  Its $t$-degree is
at most one, so reduced-source nonvanishing proves $P\notin(f)$.
\end{proof}

One can make the size test entirely finite: for $N\geq4w_0$ it
suffices that
\[
 \frac{N^d+(N-4w_0)^d}
      {d!\,2w_0\prod_iw_i}
 >
 \frac{K(\rho N+\sum_{i=0}^m v_i)^d}
      {d!\prod_{i=0}^m v_i}.
\]
The left side is a lower bound for the source count and the right
side an upper bound for the target count.  Condition
\eqref{eq:origin-auxiliary-volume} ensures this test eventually
holds.  Thus neither floor oscillations nor equality at some
rational cutoff causes a problem.

The conclusion holds, in particular, for all sufficiently large
natural integers $N$.  It also holds on any prescribed unbounded
sequence of multiples of a positive integer.  No divisibility
modulus is needed for this existence argument; a later independent
globalization or arithmetic step may impose one.  The coefficients
of $P$ are complex and no height, denominator, or norm bound on
them is asserted or needed at this stage.

\subsection{A precise route through the existing generic machinery}

The same conclusion can be obtained without adding a new simplex
counting theory to the Lean project.  The already proved theorem
\lean{OAI.PiExponent.eventually_exists_auxiliaryPolynomial}
works for an arbitrary linear map on a polynomial source with
weights $(w_0,w_1,\ldots,w_m)$.  Embed its monomials by
\[
 X_0^s Z^\alpha\longmapsto
       u^{\lfloor s/2\rfloor}t^{s\bmod2}Z^\alpha.
\]
This map is injective, since division of $s$ by $2$ recovers $s$
from the pair of exponents.  Its exact quartic source weight is
\begin{equation}
 w_0s+\sum_iw_i\alpha_i+3w_0(s\bmod2).
 \label{eq:origin-embedding-loss}
\end{equation}
The loss is at most $3w_0$, not the $2w_0$ loss in the lemniscate
template with ambient weights $(2w_0,3w_0)$.

Because \eqref{eq:origin-auxiliary-volume} is strict, choose
$b>\rho$ sufficiently close to $\rho$ that it also holds with
$b$ in place of $\rho$.  Apply the generic theorem at source cutoff
$N-3w_0$, using the actual packet map composed with the embedding
and contact factor $b$.  The embedded source cutoff is at most
$N$, and its vanishing extends to cutoff
$b(N-3w_0)$.  For
\[
 N\geq\frac{3bw_0}{b-\rho}
\]
this is at least $\rho N$.  Together with the generic sufficiently
large cutoff, this reproduces
Theorem~\ref{thm:origin-polar-auxiliary}.  The generic theorem is
uniform over its input linear maps, so the center-independent
threshold is retained.

The existing lemniscate modules
\lean{LemniscateAuxiliaryPolynomial.lean} and
\lean{LemniscateCurveAuxiliary.lean} provide implementation
templates.  Their cubic embedding-loss and degree-growth constants
must be changed as displayed above; their existing specialized
conclusions are not already the origin-quartic theorem.

\subsection{From curve excess to the terminal budget}

We next match the parameters to the actual curve-contact theorem.
Write
\[
 A_0=K(w_0/v_0)\theta^m,
 \qquad B_0=K\theta^m,
\]
and assume $A_0<1$, $B_0<1$, and $0<\theta<1$.
Choose a positive rational $\sigma$ so small that
\begin{equation}
 A_0(1+3\sigma)^{m+1}<1,
 \qquad B_0(1+3\sigma)^m<1,
 \qquad (1+\sigma)\theta<1.
 \label{eq:origin-sigma-margins}
\end{equation}
Such a choice exists by continuity at $\sigma=0$ and density of
the rationals.  The first inequality is the auxiliary-volume
condition for $\rho=1+3\sigma$.  The other two are retained for
the subsequent fiber and separated-weight argument; they are not
needed merely to find a kernel polynomial.

Put
\[
 \rho=1+3\sigma,\qquad
 \varepsilon_4=
 \min\left\{\sigma,\frac{\sigma v_0}{4(1+\sigma)w_0}\right\},
 \qquad \lambda=\frac{2w_0}{v_0},
 \qquad g_4=1+\lambda\varepsilon_4.
\]
Then $\varepsilon_4>0$, $\varepsilon_4\leq\sigma$, and
\[
 \lambda\varepsilon_4(1+\sigma)\leq\frac{\sigma}{2}.
\]
It follows that $g_4<3/2$ and, more explicitly,
\begin{equation}
 g_4(1+\sigma)\leq1+\frac32\sigma,
 \qquad
 \rho-\varepsilon_4\geq1+2\sigma,
 \qquad
 \rho-\varepsilon_4-g_4(1+\sigma)\geq\frac{\sigma}{2}>0.
 \label{eq:origin-quartic-strict-margin}
\end{equation}
This is the quartic margin proved directly.  The stored
\lean{LemniscateWeightGeometry.NumericData.derivative_margin}
uses the cubic cost $w_0$ and is not a proof of
\eqref{eq:origin-quartic-strict-margin}.  The numerical choice of
$\sigma$ and separated-product estimates can be reused, but this
different margin needs its own explicit adapter.

Let $C$ be the smooth projective curve used in
Theorem~\ref{thm:contact-terminal}; use its actual coordinate
functions, finite collection $\mathcal S$ of centered branches,
positive contacts $\mu_p$, and positive weighted pole quantity
$d_C$.  Suppose the hypothetical excess is
\begin{equation}
 d_C<(1+\sigma)\sum_{p\in\mathcal S}\mu_p.
 \label{eq:origin-initial-curve-excess}
\end{equation}
In particular the sum of contacts is positive.  Choose $N>0$
sufficiently large for
Theorem~\ref{thm:origin-polar-auxiliary}, and choose its nonzero
reduced auxiliary $P$.  Set
\[
 R=\rho N,\qquad B=\varepsilon_4N,
 \qquad\delta=\frac{\varepsilon_4N}{m+3}.
\]
Then $0<B<R$ and $B=(m+3)\delta$.
Using \eqref{eq:origin-initial-curve-excess} and
\eqref{eq:origin-quartic-strict-margin} gives
\begin{align*}
 (N+\lambda B)d_C
 &=g_4N d_C\\
 &<g_4N(1+\sigma)\sum_{p\in\mathcal S}\mu_p\\
 &<(\rho-\varepsilon_4)N\sum_{p\in\mathcal S}\mu_p
  =(R-B)\sum_{p\in\mathcal S}\mu_p.
\end{align*}
Every inequality has the correct quartic source-growth constant
$\lambda=2w_0/v_0$.  The already established curve-terminal theorem
now yields
\[
 D_WP(u,t,z)=0\quad\text{in }\mathbf C(C)
       \qquad\text{for every }c(W)\leq\varepsilon_4N.
\]
Its literal curve prime has height $m+1$, contains the quartic
relation, and contains all these derivatives.  Since $P\notin(f)$,
these are exactly the input hypotheses of the origin persistence
theorem, with the same source polynomial and the same budget.

There is no conflict among the remaining local size conditions.
Enlarge $N$ further to ensure
\[
 N\geq8w_0,\qquad
 \frac{\varepsilon_4N}{m+3}>
        (m+1)\max_{0\leq i\leq m}v_i.
\]
The first allows the fixed quartic in the generator-degree bound;
the second makes every required rectangular cutoff at least two.
All are eventual inequalities for fixed positive parameters.
Thus the one auxiliary can simultaneously satisfy the actual
packet conditions, the terminal vanishing conditions, and the
source-degree and rectangle conditions required by the local
length comparison.

\paragraph{What has and has not been discharged.}
This completes the written auxiliary-polynomial existence and
parameter conversion that were still explicit inputs in the
previous checkpoint.  Starting from a hypothetical excess curve,
the actual origin-polar source now produces the exact prime and
derivative-vanishing data needed for persistence.  This does not
prove that such an excess curve exists; the following section excludes
it using persistence, normal separation, projection rigidity, and the
constant-fiber analysis. It also does not prove
global packet surjectivity: a nonzero auxiliary is a kernel
element, while EPA requires a rank assertion.  The globalization and analytic comparison are proved separately in Sections~\ref{sec:origin-geometric-assembly} and~\ref{sec:cm-same-minor-analytic}.

The actual auxiliary and its contact/source-budget application are
checked in
\lean{QuadraticPi.EllipticOriginQuarticAuxiliary.eventually_exists_nonzero_quartic_auxiliary}
and
\lean{QuadraticPi.EllipticOriginQuarticCurveAuxiliary.eventually_exists_admissible_auxiliary}.
Global packet surjectivity is the separate checked conclusion
\lean{QuadraticPi.EllipticOriginQuarticEPA.epa}; it is not inferred from
existence of a kernel element alone.

\section{Independent ordinary-fiber proof and the exact PA specialization}
\label{sec:acceptance-ordinary-fiber}

This section supplies a written proof of the ordinary curve inequality
used after the elliptic coordinates have been forced to zero.  It uses
the ordinary weighted B\'ezout and localized Taylor length estimates,
whose actual applications are specified below.  It does not use the
logarithmic curve inequality, an elliptic EPA theorem, or an assertion
of interpolation surjectivity.

\begin{theorem}[Ordinary-fiber curve inequality]
\label{thm:acceptance-ordinary-fiber}
Let $C$ be a smooth projective integral complex curve and let
$x_1,\ldots,x_m$ generate its function field $F$ over $\mathbb C$,
where $m\ge1$ and $\operatorname{trdeg}_{\mathbb C}F=1$.
Let $c_j=(c_{j1},\ldots,c_{jm})$, $1\le j\le K$, be a finite array
of centers such that $j\mapsto c_{ji}$ is injective for every $i$.
Take positive rational weights $w_i,v_i$ and $\sigma>0$, and put
\[
 A_m(\sigma)=2(m+2)^{m+2}
                   \left(1+\frac{m+2}{\sigma}\right)^{m+2}.
\]
Assume
\begin{align}
 K(1+3\sigma)^m\frac{\prod_iw_i}{\prod_iv_i}&<1,
 \label{eq:acceptance-fiber-volume}\\
 (1+\sigma)w_i&<v_i\quad(1\le i\le m),
 \label{eq:acceptance-fiber-ratio}
\end{align}
and the following separation condition: whenever $S,T\subseteq
\{1,\ldots,m\}$ have the same cardinality, and an index $i\in
S\setminus T$ satisfies
$(j\in S\Longleftrightarrow j\in T)$ for every $j>i$, then
\begin{equation}
 A_m(\sigma)\prod_{j\in T}v_j<\prod_{j\in S}w_j.
 \label{eq:acceptance-fiber-separation}
\end{equation}
Let $\mathcal S$ be all points where the $x_i$ are regular and
$(x_1(p),\ldots,x_m(p))=c_j$ for some $j$.  At such a point set
\[
 \mu_p=\min_{i:x_i-c_{ji}\ne0}
              \frac{\operatorname{ord}_p(x_i-c_{ji})}{v_i},
 \qquad
 d=\sum_{p\in C}\max\left(0,
                    \max_i\frac{-\operatorname{ord}_p(x_i)}{w_i}
                           \right).
\]
An identically zero coordinate contributes zero pole order; a zero
branch coordinate is omitted from the contact minimum, equivalently
assigned infinite order.  Then
\[
 (1+\sigma)\sum_{p\in\mathcal S}\mu_p\le d.
\]
\end{theorem}

\begin{proof}
Some $x_i$ is nonconstant.  Consequently $\mathcal S$ is finite,
every contact minimum is nonempty and positive, and $d>0$.
Suppose, for a contradiction, that
\begin{equation}
 d<(1+\sigma)\sum_{p\in\mathcal S}\mu_p.
 \label{eq:acceptance-fiber-excess}
\end{equation}

\emph{The actual auxiliary polynomial.}
For every sufficiently large integer $N$, weighted-simplex counting
and \eqref{eq:acceptance-fiber-volume} give
\[
 \#\{\alpha:\textstyle\sum_iw_i\alpha_i\le N\}
 >K\#\{\beta:\textstyle\sum_iv_i\beta_i<(1+3\sigma)N\}.
\]
The ordinary coefficient maps
$P\mapsto[V^\beta]P(c_j+V)$ are linear conditions.  There is therefore
a nonzero polynomial $P\in R=\mathbb C[X_1,\ldots,X_m]$ of weighted
degree at most $N$ such that every displayed target coefficient is
zero at every center.  This is only a nonzero-kernel dimension count;
no surjectivity assertion is used.

\emph{Actual derivatives vanish along the curve.}
For a multi-index $\alpha$ of target cost
$\sum_i v_i\alpha_i\le\sigma N$, ordinary differentiation reduces
the target vanishing cutoff by at most that cost and does not increase
the source weighted degree.  Thus
\[
 \operatorname{ord}_p\bigl((\partial^\alpha P)(x)\bigr)
 \ge(1+2\sigma)N\mu_p\quad(p\in\mathcal S)
\]
if this rational function is nonzero.  Indeed its expansion at $c_j$
has only monomials $V^\beta$ of target weight at least
$(1+2\sigma)N$, and each substituted branch coordinate has order at
least $v_i\mu_p$.  Cancellation can only increase the order.
On the other hand its total pole degree is at most $Nd$: at each
point every source monomial of weight at most $N$ has pole order
bounded by $N$ times that point's maximum weighted coordinate pole.
The equality of total zero and pole degrees for a nonzero rational
function contradicts \eqref{eq:acceptance-fiber-excess}.  Hence
\begin{equation}
 (\partial^\alpha P)(x)=0
       \quad\hbox{whenever }\sum_i v_i\alpha_i\le\sigma N.
 \label{eq:acceptance-fiber-derivative-vanishing}
\end{equation}

\emph{A persistent prime.}
Let $\mathfrak p=\ker(R\to F)$ be the actual coordinate kernel.
It is prime and has height $m-1$.  Put
$\delta=\sigma N/(m+2)$ and
\[
 I_r=(\partial^\alpha P:\textstyle\sum_iv_i\alpha_i\le r\delta)
 \quad(0\le r\le m+2).
\]
Equation \eqref{eq:acceptance-fiber-derivative-vanishing} implies
$I_r\subseteq\mathfrak p$ for all these levels.  Since $P\ne0$,
every prime above an $I_r$ is nonzero.  If $\mathfrak p=(0)$ this is
already impossible.  Otherwise choose a minimal prime of $I_{m+2}$
inside $\mathfrak p$, then successively choose minimal primes of
$I_r$ inside the previously chosen one.  This gives a nested chain
of nonzero primes of height at most $m-1$.  A strict inclusion raises
height, so two adjacent primes agree.  Their common value $Q$ is
minimal above both $I_r$ and $I_{r+1}$ and lies inside $\mathfrak p$.

Every derivative of target cost at most $\delta$ of any element of
$I_r$ lies in $I_{r+1}\subseteq Q$.  This follows from the Leibniz
rule: differentiating a coefficient polynomial consumes some of the
budget, leaving no more than $\delta$ for its derivative-generator
factor.  In particular this assertion applies to the whole ideal,
not just its displayed generators.

\emph{The two lengths are the same length.}
Let $k=\operatorname{ht}Q\ge1$ and
$\ell=\operatorname{length}_{R_Q}(R_Q/I_rR_Q)$, which is finite
because $Q$ is minimal above $I_r$.  Work over
$F_Q=\operatorname{Frac}(R/Q)$ with the coordinate tangent space and
its normal quotient.  If $S$ is any coordinate normal basis, Corollary~\ref{cor:p1-ordinary-prime-upper} gives
\begin{equation}
 \ell\prod_{i\in S}w_i\le N^k.
 \label{eq:acceptance-fiber-length-upper}
\end{equation}
Its application is legitimate because every derivative generator of
$I_r$ still has source weighted degree at most $N$.  More explicitly,
one freezes the complementary coordinates at their values in $F_Q$;
the normal-basis condition preserves this localized length.  The
resulting $k$-variable isolated intersection has those same weighted
degree bounds, so weighted B\'ezout bounds its local multiplicity by
$N^k/\prod_{i\in S}w_i$.

For another coordinate normal basis $T$, take
$n_j=\lceil\delta/(kv_j)\rceil$ for $j\in T$ and enlarge $N$ in
advance so that all these integers are at least two.  They have
$\sum_{j\in T}(n_j-1)v_j\le\delta$.  The localized Taylor map in
the independent directions $\partial_{X_j}$ maps $I_r$ to the
rectangular monomial ideal $(Y_j^{n_j}:j\in T)$: all its lower
coefficients vanish by the preceding whole-ideal derivative assertion.
The normal-basis condition makes the induced cotangent map
surjective; successive powers of the nilpotent maximal ideal then
give the rectangular quotient of length $\prod_{j\in T}n_j$.
Thus Lemma~\ref{lem:p1-generic-taylor-lower} gives
\begin{equation}
 \frac{\delta^k}{k^k\prod_{j\in T}v_j}
       \le\prod_{j\in T}n_j\le\ell.
 \label{eq:acceptance-fiber-length-lower}
\end{equation}
Both estimates concern the same $Q,I_r,F_Q$, and localized length;
no length on a different component is substituted.

Combining \eqref{eq:acceptance-fiber-length-upper} and
\eqref{eq:acceptance-fiber-length-lower} gives, for any such bases,
\[
 \prod_{i\in S}w_i
 \le k^k\left(\frac{m+2}{\sigma}\right)^k
                    \prod_{j\in T}v_j
 \le A_m(\sigma)\prod_{j\in T}v_j.
\]
The last inequality follows from $k\le m$ and
$r^k\le(1+r)^{m+2}$ for $r=(m+2)/\sigma>0$.

\emph{Separation forces a constant coordinate.}
Two distinct normal bases have a greatest index in their symmetric
difference.  After interchanging them, this index belongs to
$S\setminus T$.  Their preceding weak product inequality contradicts
\eqref{eq:acceptance-fiber-separation}.  Thus the coordinate normal
basis is unique.  It is nonempty, since $Q\ne0$ in characteristic
zero.  Choose an index $i$ in this basis.  Its normal vector cannot
belong to the span of all the other coordinate normal vectors:
otherwise basis exchange would produce a basis omitting $i$.
Consequently every tangent vector has zero $i$th coordinate.

Every $\mathbb C$-derivation of $F_Q$ gives such a tangent vector
by its values on the coordinate functions.  All these derivations
therefore kill $\overline X_i$.  A finitely generated field extension
in characteristic zero admits a derivation taking any chosen
transcendental element to one: extend that element to a transcendence
basis and extend the corresponding derivation through the remaining
finite separable extension.  It follows that $\overline X_i$ is
algebraic over $\mathbb C$, hence belongs to $\mathbb C$.
Thus $X_i-a\in Q\subseteq\mathfrak p$ for some $a\in\mathbb C$,
and the original curve coordinate $x_i$ is constant.

\emph{The last step is a single-function divisor bound.}
The strict excess implies $\mathcal S\ne\varnothing$.  Every centered
point has $c_{ji}=a$ for the constant coordinate.  Coordinatewise
injectivity therefore permits only one center, say $c_{j_0}$.
Choose a nonconstant coordinate $x_h$.  Its difference
$f=x_h-c_{j_0h}$ is nonzero and has the same poles as $x_h$.
For every $p\in\mathcal S$,
$v_h\mu_p\le\operatorname{ord}_p(f)$.  Summing and using the
principal-divisor formula gives
\[
 v_h\sum_{p\in\mathcal S}\mu_p
 \le\deg(f)_0=\deg(f)_\infty
 \le w_h d.
\]
Together with $(1+\sigma)w_h<v_h$, this contradicts
\eqref{eq:acceptance-fiber-excess}.  The theorem follows.
\end{proof}

\paragraph{Exact connection with the saved PA sources.}
The proof just given follows the ordinary branch of
\texttt{CurveInequality.lean}, specifically
\lean{excess_implies_fibre_coordinate_constant} and
\lean{no_fibre_excess_of_constant_coordinate}.
The localized estimates used above are the concrete statements
\lean{PrimeWeightedBezout.weighted_prime_length_bound} and
\lean{OrdinaryMultiplicity.prime_rectangular_length_lower},
assembled in \lean{OrdinarySliceComparison.comparison_of_normal_basis}.
They are not premises named after the desired curve inequality.
The dimension count is
\lean{OrdinaryAuxiliaryJet.eventually_exists_auxiliaryPolynomial_nat};
the divisor argument is
\lean{CurveDerivativeVanishing.ordinary_words_vanish}.
This paragraph identifies existing source declarations; it records no
new compilation or new transitive-axiom audit.

\paragraph{Exact connection with the CM origin fiber.}
On the fiber $u=t=0$ the original function field is still the same
field and is generated by $z_1,\ldots,z_m$.  Its branch coordinates
are exactly $(0,z_1-c_{j1},\ldots,z_m-c_{jm})$, because $L(0)=0$.
The PA specialization $y=1$ gives exactly these branch coordinates,
since $y-1=0$ and $\log(1)=0$.  The saved definitions implement this
in \texttt{FibreContact.coordinates\_one},
\lean{PlaceCenteredBranch.logContact_one}, and
\texttt{CurveContactFamily.contact\_one}.  The zero branch is excluded
from the minimum, and the constant coordinate contributes zero pole
order by \texttt{FibreContact.weightedDegree\_one}.

The places are normalized discrete valuations of this same function
field, trivial on $\mathbb C$, with a uniformizer of order one.
They correspond to the points of its smooth projective model.
Residue fields are $\mathbb C$, so no residue-degree factor occurs;
the field is unchanged, so no ramification factor occurs either.
The set of centered places is the literal set where the coordinate
differences have positive order, not a chosen subset of that set.

For the CM assembly take $w_i$ as its additive source weights and
$v_i=w_i/\theta$.  Its condition
$K(1+3\sigma)^m\theta^m<1$ is precisely
\eqref{eq:acceptance-fiber-volume}.  Its separated products with
$C_4\ge A_m(\sigma)$ restrict to
\eqref{eq:acceptance-fiber-separation}; and
$(1+\sigma)\theta<1$ is precisely
\eqref{eq:acceptance-fiber-ratio}.  The full $(m+1)$-variable volume
condition is required by the saved logarithmic theorem's general
signature, but is not needed for this direct ordinary-fiber proof.

\section{Explicit separated weights and the origin curve inequality}
\label{sec:origin-curve-assembly}

The auxiliary-polynomial theorem can now be assembled with the preceding
origin-chart proofs. This section makes explicit one substantial reuse of
the old PA work: its ordinary constant-fiber curve inequality applies
without changing the curve or its valuation normalization. This is a
written theorem based on the local-length and normal-dichotomy inputs
listed below; it is not a new Lean endpoint or an EPA surjectivity theorem.

\subsection{The exact established inputs used in the assembly}

Besides the auxiliary theorem just proved, we use the following three
results. Stating them here fixes their role and prevents a conditional
geometric data structure from being mistaken for a proof of its fields.

\begin{enumerate}[label=(\Alph*),leftmargin=2em]
\item \textbf{Matched local lengths.} In the origin ring
$R_0=\mathbf C[u,t,Z_1,\ldots,Z_m]$, with relation $f=t^2-H(u)$,
let $Q$ be minimal over an actual derivative ideal $I_s$, with
$H'(\bar u)\ne0$. Let $k\ge1$ be the dimension of the intrinsic
normal quotient over $\operatorname{Frac}(R_0/Q)$.
If all generators of $I_s$ have source degree at most $M\ge8w_0$,
then, for every coordinate normal basis $S$ of size $k$,
\[
 \ell_Q(I_s)\prod_{i\in S}r_i\le4M^k,
 \qquad r=(w_0,w_1,\ldots,w_m).
\]
If also all cost-at-most-$\delta$ derivatives of every element of
$I_s$ lie in $Q$, then for every frame normal basis $T$, provided
the corresponding ceiling boxes have cutoffs at least two,
\[
 \frac{\delta^k}{k^k\prod_{j\in T}v_j}\le\ell_Q(I_s),
 \qquad v=(v_0,w_1/\theta,\ldots,w_m/\theta).
\]
These are the bounds of Section~\ref{sec:origin-quartic-length}, proved there from the generic localized Taylor
lower bound and fixed-hypersurface upper bound. The same prime, ideal,
field and length occur on both sides. The bound with factor one when
$0\in S$ is stronger but is not needed here.

\item \textbf{Normal exclusion implies constant projection.} For a
strict prime $Q\supsetneq(f)$ on the nonzero-Jacobian chart, suppose
there are no coordinate/frame normal bases $S,T$ with an index $i>0$
such that $i\in S\setminus T$ and their membership agrees above $i$.
Then $\bar u,\bar t$ are constants. This is
Theorem~\ref{thm:origin-normal-exclusion-projection}: its written proof uses the established
generic linear-algebra normal-basis dichotomy and the direct quartic
finite-extension nonexactness argument. The exclusion hypothesis will
be derived from (A), not assumed in the final curve theorem.

\item \textbf{Ordinary fiber curve inequality from PA.} Let a smooth
projective integral complex curve have function field generated by
$z_1,\ldots,z_m$, of transcendence degree one. Take coordinatewise
distinct centers $c_j\in\mathbf C^m$. At their centered points set
\[
 \mu_p^{\rm fib}=\min_i\frac{\operatorname{ord}_p(z_i-c_{ji})}{v_i},
 \qquad
 d_C^{\rm fib}=\sum_p\max_i\frac{n_p(z_i)}{w_i},
\]
where zero branch functions have infinite order and $n_p(0)=0$.
If the full and fiber volume inequalities, normal-product separation
with constant $A_m(\sigma)$, and coordinate inequalities
$(1+\sigma)w_i<v_i$ hold, then
\[
 (1+\sigma)\sum_p\mu_p^{\rm fib}\le d_C^{\rm fib}.
\]
This is the specialization $y=1$ of the already established declaration
\lean{OAI.PiExponent.CurveInequality.weighted_curve_inequality} in
\lean{Geometry/CurveInequalityIntrinsic.lean}. It is a curve theorem,
not the PA interpolation conclusion. Below we list and check all its
numerical conditions and identify its actual contacts. Its proof is
inherited from the preceding PA work, and has been reproved directly in Theorem~\ref{thm:acceptance-ordinary-fiber}.
\end{enumerate}

Here and below
\[
 A_n(s)=2(n+2)^{n+2}\left(1+\frac{n+2}{s}\right)^{n+2}
 \quad(s>0).
\]
The existing Lean declarations underlying (A) are the generic
\lean{rectangularTaylor_length_lower} and
\lean{weighted_prime_hypersurface_length_bound}; their checked reuse
interfaces are listed in the earlier reference audit. Their quartic
specializations and the present assembly are written proofs only.

\subsection{All thresholds can be chosen before the centers}

Fix $m,K\ge1$ and positive rational $w_0,v_0,\theta$ with
\begin{equation}\label{eq:curve-base-volumes}
 \theta<1,\qquad
 \alpha:=K(w_0/v_0)\theta^m<1,\qquad
 \beta:=K\theta^m<1.
\end{equation}
The previously proposed stronger origin-EPA condition
$2K(w_0/v_0)\theta^m<1$ implies the middle inequality; the curve
argument itself needs only \eqref{eq:curve-base-volumes}. The actual globalization below uses the weaker conditions displayed here.

By continuity at zero, choose a positive rational $\sigma$ so small that
\begin{equation}\label{eq:curve-sigma-conditions}
 (1+3\sigma)^{m+1}\alpha<1,\qquad
 (1+3\sigma)^m\beta<1,\qquad (1+\sigma)\theta<1.
\end{equation}
Set
\[
 \rho=1+3\sigma,\quad
 \varepsilon=\min\left\{\sigma,
       \frac{\sigma v_0}{4(1+\sigma)w_0}\right\},\quad
 g=1+\frac{2w_0}{v_0}\varepsilon,
\]
\[
 C_4=A_m(\sigma)+4\,2^{m+3}A_{m+1}(\varepsilon),\qquad
 q_* =\theta^m\frac{\min(1,w_0)}{\max(1,v_0)}.
\]
These are positive rational constants chosen before any additive weights,
centers or curves. Choose the weights successively subject to
\begin{equation}\label{eq:explicit-origin-threshold}
 w_i>\max\left\{1,\frac{C_4}{q_*}\prod_{1\le j<i}w_j\right\}
 \quad(1\le i\le m).
\end{equation}
The empty product is one.

\begin{lemma}[Direct normal-product separation]
\label{lem:origin-explicit-separation}
For index sets $S,T\subseteq\{0,\ldots,m\}$ of equal cardinality,
suppose $i>0$ belongs to $S\setminus T$ and $S,T$ have equal
membership at every index greater than $i$. Then
\[
 C_4\prod_{j\in T}v_j<\prod_{j\in S}r_j,
 \qquad r_0=w_0,\ r_i=w_i,\ v_i=w_i/\theta\ (i>0).
\]
\end{lemma}
\begin{proof}
Writing $S_+=S\setminus\{0\}$ and $T_+=T\setminus\{0\}$, the ratio
of the right product to the product of target weights is exactly
\[
 \theta^{|T_+|}
 \frac{w_0^{\mathbf1_{0\in S}}}{v_0^{\mathbf1_{0\in T}}}
 \frac{\prod_{j\in S_+\setminus T_+}w_j}
      {\prod_{j\in T_+\setminus S_+}w_j}.
\]
The first two factors are at least $q_*$. Every remaining denominator
index lies below $i$, the numerator contains $w_i$, and all the other
weights exceed one. The ratio is therefore at least
$q_*w_i/\prod_{1\le j<i}w_j$, which is strictly greater than $C_4$ by
\eqref{eq:explicit-origin-threshold}.
\end{proof}

Since $C_4\ge A_m(\sigma)$, this also supplies the ordinary separation
condition in input (C). The two volume conditions in (C) are exactly
the first two inequalities of \eqref{eq:curve-sigma-conditions}, because
\[
 \frac{\prod_{i=0}^m r_i}{\prod_{i=0}^m v_i}
   =\frac{w_0}{v_0}\theta^m,
 \qquad
 \frac{\prod_{i=1}^m r_i}{\prod_{i=1}^m v_i}=\theta^m.
\]
Its coordinate conditions follow from $(1+\sigma)\theta<1$.
The quartic derivative margin is separately proved in the auxiliary
section. It is not imported from the old cubic numerical record.

\subsection{Excess contact forces the origin fiber}

Let $C$ be a smooth projective integral complex curve, with function
field generated by rational functions $u,t,z_1,\ldots,z_m$ satisfying
$t^2=H(u)$, and with transcendence degree one. Let $\mathcal S$ be
all points where these functions are regular and have value
$(0,0,c_j)$ for one of the chosen centers. This is a finite set:
choose any nonconstant coordinate; for each center its prescribed
value has only finitely many preimages. Distinct centers cannot
label the same point. Define $\mu_p$ and $d_C$ using the exact branch
coordinates and pole orders of the preceding contact supplement.
In particular $d_C>0$.

\begin{proposition}[Excess-to-fiber consequence]
\label{prop:origin-excess-fiber}
Under the numerical choices above and inputs (A)--(B), if
\[
 d_C<(1+\sigma)\sum_{p\in\mathcal S}\mu_p,
\]
then $u=t=0$ as rational functions on $C$.
\end{proposition}
\begin{proof}
Choose $N$ large enough for the auxiliary theorem, $N\ge8w_0$, and
\[
 \delta:=\frac{\varepsilon N}{m+3}
      >(m+1)\max_{0\le i\le m}v_i.
\]
The auxiliary $P\notin(f)$ has source cutoff $N$ and packet contact
cutoff $\rho N$ at every center. The quartic scalar margin and the
excess hypothesis, as proved in the auxiliary section, force all
derivative words of cost at most $\varepsilon N=(m+3)\delta$ to
vanish in $\mathbf C(C)$.

Let $\mathfrak p$ be the kernel of coordinate evaluation in this same
field. It has height $m+1$. The sum of contacts is positive under
the assumed strict excess, so choose a centered point $p_0$ with
center $c$. Then $\mathfrak p\subseteq M_c$.
The origin persistence theorem supplies an integer $0\le r<m+3$
and a strict prime $Q\subseteq\mathfrak p$, minimal over both
$I_{r\delta}$ and $I_{(r+1)\delta}$, such that every derivative of
cost at most $\delta$ of every element of $I_{r\delta}$ lies in $Q$.
Since $Q\subseteq M_c$, its Jacobian is nonzero.

Put $s=r\delta$. Every generator of $I_s$ has source cutoff at most
$M=gN$: the fixed relation has cutoff $8w_0\le N$, and derivative
generators have cutoff at most $N+(2w_0/v_0)s\le gN$. Its intrinsic
normal space has dimension $1\le k\le m+1$. The chosen bound on
$\delta$ makes every ceiling cutoff $\lceil\delta/(kv_j)\rceil$
at least two, and each associated rectangular budget is at most
$\delta$. Thus input (A), applied to the same $Q$ and $I_s$, gives
for every coordinate/frame normal-basis pair $S,T$
\[
 \prod_{i\in S}r_i
 \le4g^k k^k\left(\frac{m+3}{\varepsilon}\right)^k
                \prod_{j\in T}v_j
 \le C_4\prod_{j\in T}v_j.
\]
For the last inequality use $g<3/2<2$, $k\le m+1$, and the
definition of $A_{m+1}(\varepsilon)$: its factors dominate
$k^k((m+3)/\varepsilon)^k$, while $2^{m+3}$ dominates $g^k$.
Lemma~\ref{lem:origin-explicit-separation} excludes every forbidden
pair of normal bases. Input (B) now gives constants $u_0,t_0$ in
the residue field of $Q$.

Equivalently $u-u_0,t-t_0\in Q$. Since $Q\subseteq M_c$, evaluation
at the center forces $u_0=t_0=0$. Finally $Q\subseteq\mathfrak p$
implies $u=t=0$ on the original curve, as asserted.
\end{proof}

\subsection{The precise PA fiber specialization}

On the fiber $u=t=0$, the curve field is generated by the additive
coordinates alone. At a centered point the actual origin branch vector is
\[
 (t,z_1-c_{j1}-L(t),\ldots,z_m-c_{jm}-L(t))
      =(0,z_1-c_{j1},\ldots,z_m-c_{jm}).
\]
The logarithmic PA curve theorem at multiplicative coordinate $y=1$
has exactly the same vector: $y-1=0$ and its formal logarithm is zero.
Both use normalized orders on the same curve function field. The sets
of centered places coincide, and no ramification or residue-degree
factor is inserted. Constants have no poles, so the source weighted
pole degrees also coincide:
\[
 d_C=\sum_p\max_i n_p(z_i)/w_i=d_C^{\rm fib}.
\]
The target weight $v_0$ is irrelevant to the minimum at these points
because its coordinate is identically zero and has order infinity.
Nonconstancy ensures that some remaining branch coordinate is nonzero.
The old curve theorem therefore applies exactly, with the numerical
conditions already checked and the assumed coordinatewise distinctness.

\begin{theorem}[Uniform weighted origin curve inequality]
\label{thm:origin-curve-inequality}
Fix the nonsingular integral short cubic and origin chart, and rational
base data satisfying \eqref{eq:curve-base-volumes}. Choose $\sigma$ and
the successive rational thresholds by
\eqref{eq:curve-sigma-conditions}--\eqref{eq:explicit-origin-threshold}.
For every weight list above those thresholds, every coordinatewise
distinct array of $K$ centers, and every curve with coordinate functions
as above, one has
\[
 (1+\sigma)\sum_{p\in\mathcal S}\mu_p\le d_C.
\]
This is a written consequence of the proved auxiliary and origin
arguments and the explicitly stated established inputs (A)--(C).
\end{theorem}
\begin{proof}
If the inequality failed, Proposition~\ref{prop:origin-excess-fiber}
would force $u=t=0$. The preceding specialization identifies all
terms with those of input (C), which then gives the desired inequality,
a contradiction. If there are no centered points, the conclusion
already follows from $d_C\ge0$.
\end{proof}

The three inputs in this assembly are now proved in the preceding
sections. The next section supplies the passage to the actual packet
map; neither a geometric interface nor an auxiliary-polynomial kernel
is substituted for that passage.

\section{From the origin curve inequality to the actual packet map}
\label{sec:origin-geometric-assembly}

This section supplies the geometric bridge for the origin-polar
source.  It follows the organization of the earlier lemniscate
argument, but constructs the compactification and its section
comparison directly from the actual source.  In particular it does
not assume projective normality or an equality of all section spaces.

Fix the chosen smooth CM curve $y^2=x^3+Ax+B$, with
$4A^3+27B^2\ne0$, and choose $a_*$ with $a_*^3+Aa_*+B\ne0$.
Put
\[
 H(U)=U+3a_*U^2+(3a_*^2+A)U^3+(a_*^3+Aa_*+B)U^4.
\]
Its origin chart has $u=(x-a_*)^{-1}$ and $t=y(x-a_*)^{-2}$.
Let $m,K\geq1$, and write
\[
 S=\mathbf C[u,t,Z_1,\ldots,Z_m]/(t^2-H(u)),
 \qquad X^\circ=\operatorname{Spec}S.
\]
Let $\mathcal F_h\subset S$ be the span of the reduced monomials
\[
 u^jt^bZ^\alpha,\qquad b\in\{0,1\},\qquad
 2w_0(j+2b)+\sum_iw_i\alpha_i\leq h.
\]
All weights are positive rational numbers.  Reduction by
$t^2=H(u)$ never increases weight, because every term of $H(u)$
has weight at most $8w_0$.  Consequently
\begin{equation}
 \mathcal F_h\mathcal F_{h'}\subseteq\mathcal F_{h+h'}.
 \label{eq:origin-filtered-products}
\end{equation}
The target weights are $v=(v_0,w_1/\theta,\ldots,w_m/\theta)$,
where $v_0>0$ and $0<\theta<1$ are rational.
For clarity, the actual series are determined by
\[
 U(0)=0,\quad H(U(T))=T^2,\qquad
 L(0)=0,\quad L'(T)=-H'(U(T))^{-1},
\]
and
$\Phi_c(P)=P(U(T),T,c_i+V_i+L(T))$.
They exist uniquely because $H'(0)=1$ and the constant field has
characteristic zero; $U(T)$ has order two.

We use the established uniform curve inequality with a fixed
rational $\sigma>0$:
\begin{equation}
 (1+\sigma)\sum_{p\text{ over the selected centers}}\mu_p
       \leq d_C.
 \label{eq:origin-input-curve-inequality}
\end{equation}
Here $C$ is the smooth projective model of the function field of
an integral affine curve in $X^\circ$, the tuple $(u,t,Z_i)$
generates that function field, and
\[
 d_C=\sum_{p\in C}
       \max\left(0,\frac{\operatorname{pole}_p(u)}{2w_0},
                     \frac{\operatorname{pole}_p(t)}{4w_0},
                     \max_i\frac{\operatorname{pole}_p(Z_i)}{w_i}\right).
\]
The nonnegative pole order of the zero function is defined to be
zero.  At a branch over $(u,t,Z)=(0,0,c_j)$,
\[
 \mu_p=\min\left\{
       \frac{\operatorname{ord}_p(t)}{v_0},
       \frac{\operatorname{ord}_p(Z_i-c_{ji}-L(t))}{v_i}:1\leq i\leq m
       \right\},
\]
with the order of a zero series taken as $+\infty$.
This definition uses the actual formal logarithm on the completed
local branch.  The preceding P1 curve argument establishes
\eqref{eq:origin-input-curve-inequality} under its stated
successive-weight and center hypotheses.  The following bridge
uses that precise statement as its input.

\subsection{A compactification built from one complete source space}

Choose a positive integer $D$ such that all the numbers
\begin{equation}
 \frac{D}{2w_0},\quad\frac{D}{4w_0},\quad
 \frac{D}{w_i},\quad
 a=\frac{D}{v_0},\quad b_i=\frac{D}{v_i}
 \label{eq:origin-radius}
\end{equation}
are positive integers.  A common multiple exists because the
weights are positive rationals.  This choice precedes the centers.
It also ensures that $1,u,t,Z_i$ belong to $\mathcal F_D$.

Choose a basis $1,f_1,\ldots,f_M$ of $\mathcal F_D$ and map
\[
 X^\circ\longrightarrow\mathbf P^M,
       \qquad x\longmapsto[1:f_1(x):\cdots:f_M(x)].
\]
On the chart where the first coordinate is nonzero this is a
closed immersion: the $f_i$ generate $S$ as an algebra, because
their linear span contains $u,t,Z_i$.  Let $X$ be its reduced
projective closure and let $\mathcal H=\mathcal O_X(1)$.
Then $X$ is an integral projective variety, $\mathcal H$ is very
ample, and its intersection with that standard affine chart is
exactly $X^\circ$.  The section represented by the constant $1$
trivializes $\mathcal H$ on $X^\circ$.

\begin{lemma}[Eventual section representatives]
\label{lem:origin-section-representatives}
For all sufficiently large integers $n$, every section of
$\mathcal H^{\otimes n}$, restricted to $X^\circ$ using this
trivialization, is represented by an element of $\mathcal F_{nD}$.
\end{lemma}
\begin{proof}
For the fixed projective embedding, the ideal-sheaf sequence
\[
 0\longrightarrow\mathcal I_X(n)
 \longrightarrow\mathcal O_{\mathbf P^M}(n)
 \longrightarrow\mathcal O_X(n)\longrightarrow0
\]
and Serre vanishing give surjectivity on global sections for
$n$ sufficiently large.  Thus each section is the restriction of
a homogeneous polynomial of degree $n$ in the projective
coordinates.  On the affine chart its representative is a sum of
products of $n$ elements of $\mathcal F_D$, and hence belongs to
$\mathcal F_{nD}$ by \eqref{eq:origin-filtered-products}.
Reduction to $t$-degree at most one retains the cutoff.
The projective Serre-vanishing input is recorded in the Stacks
Project, Tag 01XO: \url{https://stacks.math.columbia.edu/tag/01XO}.
\end{proof}

Only this direction is needed.  We have not asserted that the
graded ring generated by $\mathcal F_D$ is normal, that it contains
every filtered piece in every degree, or that all global sections
have such representatives in small degrees.

\begin{lemma}[Exact hyperplane degree on the curve's own normalization]
\label{lem:origin-hyperplane-degree}
Let $B\subset X$ be an integral projective curve meeting
$X^\circ$, and let $\nu:C\to B$ be its normalization.  For the
coordinate functions on this same $C$,
\[
 \deg_C\nu^*\mathcal H=Dd_C.
\]
\end{lemma}
\begin{proof}
For functions $1,f_1,\ldots,f_M$ defining the morphism to
projective space, the associated line bundle has divisor
\[
 \sum_{p\in C}\max(0,\operatorname{pole}_p(f_1),\ldots,
                                  \operatorname{pole}_p(f_M))\,p.
\]
Indeed clearing exactly these local poles gives generating
sections with no common zero: at each point a coordinate achieving
the maximum has nonzero reduced value.  Thus this divisor represents
the pullback of $\mathcal O(1)$.

Every monomial in $\mathcal F_D$ has pole order at most $D d_p$,
where $d_p$ is the local maximum in the definition of $d_C$.
Conversely, the functions
\[
 u^{D/(2w_0)},\qquad t^{D/(4w_0)},\qquad Z_i^{D/w_i}
\]
belong to $\mathcal F_D$, after reduction if necessary.  Each is
a constant linear combination of the chosen basis.  Its pole
order cannot exceed the maximum pole order among that basis, and
the displayed functions attain every candidate in $D d_p$.
Together with the constant function this proves equality of the
local maxima, and summing gives the degree formula.
\end{proof}

This proof uses the normalization of the curve itself.  It does
not identify the degree on an arbitrary finite cover with the
degree on the original curve; such a pullback would multiply both
the line-bundle degree and all valuation sums by the covering degree.

\subsection{The actual fixed ideal and the actual weighted packets}

Choose an integer $T_0>0$ with $T_0v_0>v_i$ for every $i>0$, still
before the centers, and put
$L_{<T_0}(t)=\sum_{s<T_0}[t^s]L\,t^s$.
For distinct centers $c_j\in\mathbf C^m$, define ideals of $S$
\begin{equation}
 I_j=\left(t^a,u^a,
       (Z_i-c_{ji}-L_{<T_0}(t))^{b_i}:1\leq i\leq m\right),
 \qquad I=\prod_{j=1}^K I_j.
 \label{eq:origin-fixed-ideal}
\end{equation}
Coordinatewise distinctness in the EPA statement implies this
ordinary distinctness, since $m\geq1$.

The radical of $I_j$ is exactly
$M_j=(u,t,Z_i-c_{ji})$: any prime containing $I_j$ contains
$u,t$ and then every $Z_i-c_{ji}$ because $L_{<T_0}(0)=0$.
The quotient is finite-dimensional over $\mathbf C$.  For example,
use $u,t$ and the shifted additive variables as generators; their
displayed powers vanish, yielding a finite spanning set.
The ideals for distinct centers are comaximal, as are all their
powers.  Hence $I^n=\prod_j I_j^n=\bigcap_jI_j^n$ and
\begin{equation}
 S/I^n\simeq\prod_j S/I_j^n.
 \label{eq:origin-crt}
\end{equation}

The selector $u^a$ is essential to global support: $t=0$ may meet
other zeros of $H$.  Locally at $M_j$ it is redundant.  Indeed
$H(u)=u h(u)$ with $h(0)=1$, and $u h(u)=t^2$ gives
$u^a=t^{2a}h(u)^{-a}\in(t^a)$ in that local ring.

\begin{lemma}[Finite quotients map onto the strict packets]
\label{lem:origin-ideal-to-packet}
For every positive integer $n$, each $\Phi_{c_j}$ induces a
surjection
\[
 S/I_j^n\longrightarrow
 \mathbf C[[T,V_1,\ldots,V_m]]/\mathfrak j_{nD}.
\]
Consequently $S/I^n$ surjects onto the product of the actual
strict packets of cutoff $nD$.
\end{lemma}
\begin{proof}
The image of $t^a$ has weight $av_0=D$.
The series $U(T)$ has order two, so the image of $u^a$ has
weight at least $2av_0=2D$.  The image of the shifted additive
generator before its $b_i$th power is
\[
 V_i+L(T)-L_{<T_0}(T).
\]
Every term has weight at least $v_i$, since the remainder has
$T$-order at least $T_0$ and $T_0v_0>v_i$.  Its $b_i$th power has
weight at least $b_iv_i=D$.  Thus $\Phi_{c_j}(I_j^n)$ lies in
$\mathfrak j_{nD}$, proving factorization.

For surjectivity, the target is generated as an algebra by the
images of $T,V_i$.  The variable $T$ is the image of $t$.
Modulo $\mathfrak j_{nD}$ the series $L(T)$ is represented by a
finite polynomial in $T$.  Subtracting this polynomial and $c_{ji}$
from the image of $Z_i$ yields $V_i$.  Thus every target monomial
has a polynomial preimage.  The product assertion follows from
\eqref{eq:origin-crt}.
\end{proof}

The fixed quotient $S/I_j^n$ generally has more classes than the
weighted packet.  The proof requires and establishes a surjection,
not a false equality of these two jet schemes.

\subsection{Blowup degrees, including all curve types}

The finite subscheme $\operatorname{Spec}(S/I)$ of $X^\circ$
is also a closed subscheme of $X$: it is finite, hence proper,
over $\mathbf C$, and its immersion into the separated variety
$X$ has closed image.  Let $\mathcal I\subset\mathcal O_X$ be its
coherent ideal.  It restricts to $I$ on $X^\circ$ and is the unit
ideal outside the selected points.  Form the blowup
\[
 \pi:Y=\operatorname{Bl}_{\mathcal I}X\longrightarrow X,
 \qquad \mathcal A=\pi^*\mathcal H,
 \qquad \mathcal J=\mathcal I\mathcal O_Y.
\]
The variety $Y$ is integral and projective.  The invertible ideal
$\mathcal J$ is the relatively ample tautological bundle
$\mathcal O_Y(1)$ for the Rees algebra; equivalently it is
$\mathcal O_Y(-E)$ for the effective exceptional divisor.
These two descriptions have the same sign.  The blowup and
tautological-ideal facts are in the Stacks Project, Tag 01OF:
\url{https://stacks.math.columbia.edu/tag/01OF}.

Consider an integral curve $B\subset Y$ which is not contracted
by $\pi$ and whose image meets $X^\circ$.  The blowup is an
isomorphism outside the finite support of $\mathcal I$, so the
map from $B$ to its image is birational.  Use the common
normalization $C$, with its own function field.  The hyperplane
formula gives $\deg_C\mathcal A=Dd_C$.
At a point of $C$ over center $j$, the order of the pulled-back
ideal is the minimum of the orders of its generators.  The other
center ideals are units there.  The selector contributes no
smaller order, since $\operatorname{ord}(u)=2\operatorname{ord}(t)$
on a nonzero centered branch.  Zero functions are handled by
the value $+\infty$.

Replacing $L(t)$ by $L_{<T_0}(t)$ does not change the resulting
minimum.  In fact the normalized order of their difference in
the $i$th coordinate is greater than
$\operatorname{ord}(t)/v_0$ whenever $t$ is nonzero, by
$T_0v_0>v_i$; when $t=0$ the difference is zero.  Adding such a
term cannot change a minimum that already includes
$\operatorname{ord}(t)/v_0$.
Thus the two exact degree formulas on this same normalization are
\begin{equation}
 \deg_C\mathcal A=Dd_C,
 \qquad
 \deg_C\mathcal J=-D\sum_p\mu_p.
 \label{eq:origin-blowup-degrees}
\end{equation}
The second formula is the divisor of the pulled-back invertible
ideal, with local generator order $D\mu_p$; it retains all branch
multiplicities.

There are two other curve types.  If $B$ is contracted, then
$\deg_B\mathcal A=0$ and $\deg_B\mathcal J>0$, because
$\mathcal J$ is ample on each fiber.  If $B$ is noncontracted
and its image is contained in $X\setminus X^\circ$, then the
blowup is an isomorphism along it and $\mathcal J$ is trivial
there, while $\deg_B\mathcal A>0$.  Thus no unexamined boundary
curve is omitted from the following numerical comparison.

\subsection{The uniform margin proves ampleness}

Relative ampleness of $\mathcal J$ and ampleness of $\mathcal H$
give, by Stacks Lemma 29.38.7(1), an integer $q>1$ such that
$\mathcal B=\mathcal A^{\otimes q}\otimes\mathcal J$ is ample
on $Y$.  This auxiliary integer may depend on the centers.
It changes neither $D$ nor the desired bundle
$\mathcal L=\mathcal A\otimes\mathcal J$.
Put
\[
 \gamma=\frac{\sigma}{q(1+\sigma)-1}>0.
\]
Then every integral curve $B\subset Y$ satisfies
\begin{equation}
 \deg_B\mathcal L\geq\gamma\deg_B\mathcal B.
 \label{eq:origin-uniform-ample-margin}
\end{equation}
For an affine-meeting noncontracted curve, write
$h=\deg_B\mathcal A=Dd_C$ and
$e=-\deg_B\mathcal J=D\sum_p\mu_p$.
The curve inequality says $(1+\sigma)e\leq h$.
Since $0<\gamma<1$, direct subtraction gives
\[
 (h-e)-\gamma(qh-e)
 =(1-\gamma q)h-(1-\gamma)e\geq0;
\]
the coefficient identity is
$1-\gamma q=(1-\gamma)/(1+\sigma)$.
For a contracted curve, both degrees in the comparison are
$\deg\mathcal J>0$ and $\gamma<1$.
For a boundary curve their ratio is $1/q\geq\gamma$.
This proves \eqref{eq:origin-uniform-ample-margin} in all cases.

The standard uniform numerical ampleness criterion now proves
that $\mathcal L$ is ample.  Equivalently,
$\mathcal L-\gamma\mathcal B$ is nef as a rational numerical
divisor class, and adding the positive ample class
$\gamma\mathcal B$ is ample.  This uses the ample-plus-nef theorem for projective varieties;
see \cite[Corollaries 1.4.10--1.4.11]{lazarsfeld}.  Since $\gamma$ is rational,
one may clear denominators: a positive tensor power of
$\mathcal L$ is the tensor product of an ample bundle and a nef
bundle, hence is ample, which makes $\mathcal L$ ample.
It is stronger than, and must
not be replaced by, the generally insufficient assertion that
$\mathcal L$ merely has positive degree on every curve.

\subsection{Serre vanishing yields the exact EPA packets}

We spell out the standard Rees and cohomology step, since its
eventual nature is sufficient and avoids any unnecessary equality
in small degrees.  For all sufficiently large integers $n$,
\[
 \pi_*\mathcal J^{\otimes n}=\mathcal I^n,
 \qquad R^i\pi_*\mathcal J^{\otimes n}=0\quad(i>0).
\]
These are the eventual graded-sheaf comparison and relative Serre
vanishing for the finitely generated Rees algebra.  The graded comparison is Stacks Lemma 30.15.2(5), Tag 0BXE:
\url{https://stacks.math.columbia.edu/tag/0BXE}; the projective
cohomology input is Tag 01XO:
\url{https://stacks.math.columbia.edu/tag/01XO}.  They are
uniform over a finite affine cover of the Noetherian projective
variety $X$.  Projection formula and Leray then identify
\[
 H^1(X,\mathcal H^{\otimes n}\otimes\mathcal I^n)
   =H^1(Y,(\mathcal A\otimes\mathcal J)^{\otimes n})=0
\]
for all sufficiently large $n$, the last equality being Serre
vanishing for the fixed ample bundle $\mathcal L$.
The ideal sequence consequently gives a surjection
\[
 H^0(X,\mathcal H^{\otimes n})\longrightarrow
 H^0(X,\mathcal H^{\otimes n}\otimes\mathcal O_X/\mathcal I^n).
\]
The target is canonically $S/I^n$ after using the section $1$ to
trivialize $\mathcal H$ on the affine chart containing its support.
By Lemma~\ref{lem:origin-section-representatives}, every section
on the left has a representative in $\mathcal F_{nD}$ once $n$
is also large enough for that fixed embedding.  By
Lemma~\ref{lem:origin-ideal-to-packet}, the target in turn
surjects onto the product of the actual strict packets of cutoff
$nD$.  Their composite is precisely the required coefficient map
\[
 \mathcal F_{nD}\longrightarrow
 \prod_{j=1}^K
 \mathbf C^{\{(s,\beta):v_0s+\sum_i v_i\beta_i<nD\}}.
\]
It is therefore surjective for all sufficiently large $n$.

\begin{theorem}[Origin-polar EPA from the established curve inequality]
\label{thm:origin-polar-epa-geometry}
Fix source and target weights and a positive uniform margin
$\sigma$ for which \eqref{eq:origin-input-curve-inequality} holds
for every integral function-field curve and every allowed
configuration of centers.  There is a positive integer $D$,
depending only on the weights, such that for every such center
configuration there is $H_0$ for which the actual origin-polar
packet map is surjective whenever $H\geq H_0$ and $D\mid H$.
\end{theorem}

The proof just given preserves the intended order of choices:
base data and curve-inequality thresholds, then weights, then
$D$ and the finite logarithm truncation, then centers, then the
eventual threshold $H_0$.  The auxiliary ample exponent $q$ and
the Serre thresholds may depend on the centers.  A multiple of
$D$ has the form $H=nD$; taking $H_0=Dn_0$ supplies the final
divisibility formulation.

\begin{corollary}[Sharp separated-weight origin EPA]
\label{cor:sharp-origin-epa}
Fix \eqref{eq:cubic} and an integer $a_*$ with
$a_*^3+Aa_*+B\ne0$. Let $m,K\ge1$ and let the positive rational
base data $w_0,v_0,\theta$ satisfy
\[
 0<\theta<1,\qquad K(w_0/v_0)\theta^m<1,\qquad K\theta^m<1.
\]
Choose $\sigma$ by \eqref{eq:curve-sigma-conditions} and let
$w_1,\ldots,w_m$ be positive rational weights satisfying the
successive thresholds \eqref{eq:explicit-origin-threshold}.
There is a positive integer $D$, depending on these weights and
base data but not the centers, such that for every coordinatewise
distinct array $c_{ji}\in\mathbb C$ there is $H_0$ with the
following property. Put
\[
 \mathcal F_H=\operatorname{span}_{\mathbb C}
 \{u^kt^b Z^\alpha:b\in\{0,1\},\
        2w_0(k+2b)+\textstyle\sum_iw_i\alpha_i\le H\}.
\]
For every integer $H\ge H_0$ divisible by $D$, the map
\[
 \mathcal F_H\longrightarrow
 \prod_{j=1}^K
 \mathbb C^{\{(s,\beta):v_0s+\sum_i(w_i/\theta)\beta_i<H\}}
\]
that takes every coefficient $[T^sV^\beta]$ of
$P(U(T),T,c_{ji}+V_i+L(T))$ is surjective.
\end{corollary}
\begin{proof}
Theorem~\ref{thm:origin-curve-inequality} supplies the uniform
curve inequality with the chosen positive margin. Theorem
\ref{thm:origin-polar-epa-geometry} applies to these exact source
and target spaces and gives the asserted order of quantifiers.
\end{proof}

\paragraph{Formal status.}
The concrete compactification, ideals, curve degrees, weighted sections
and literal packet identifications are assembled in
\lean{QuadraticPi.EllipticOriginQuarticEPA.epa} and
\lean{QuadraticPi.EllipticOriginQuarticEPA.eventual_packetMap}.
No ampleness, curve-comparison or packet-surjectivity premise remains
in the EPA theorem beyond its explicit model conditions.


\section{The CM rectangular sublattice and the identical analytic minor}
\label{sec:cm-grid-same-minor}

\subsection{A rectangular sublattice from an actual CM action}

Fix a compatible uniformization of the normalized curve
$y^2=x^3+Ax+B$ by $\mathbf C/\Lambda$, with invariant differential
$dx/(2y)=dz$. Thus $g_2(\Lambda)=-4A$, $g_3(\Lambda)=-4B$,
$x=\wp_\Lambda$ and $y=\wp'_\Lambda/2$, in the conventions used
throughout this paper. Let $\alpha\notin\mathbf R$ satisfy
$\alpha\Lambda\subseteq\Lambda$. Such an action is the analytic
realization of the specified complex multiplication; in this argument
the lattice and action are the actual ones for the specified differential.

\begin{proposition}[Trace-zero CM generator]
There are integers $T,n$ such that
$\alpha^2-T\alpha+n=0$. Put $d=4n-T^2$. Then $d$ is a positive
integer and
\[
 \beta=2\alpha-T,\qquad \beta\Lambda\subseteq\Lambda,
 \qquad \beta=\pm i\sqrt d.
\]
After replacing $\alpha$ by $-\alpha$ if necessary, one may assume
$\operatorname{Im}\alpha>0$, in which case $\beta=i\sqrt d$.
For every $\omega\in\Lambda\setminus\{0\}$,
\[
 \omega(\mathbf Z+i\sqrt d\,\mathbf Z)\subseteq\Lambda
\]
is a rank-two rectangular sublattice.
\end{proposition}

\begin{proof}
Choose a lattice basis $\lambda_1,\lambda_2$. Multiplication by
$\alpha$ is represented in this basis by a $2\times2$ integer
matrix. Its characteristic polynomial gives the displayed equation,
with integral trace $T$ and determinant $n$. Writing $\alpha=x+iy$
and using $y\ne0$, its imaginary part gives $T=2x$ and its real
part gives $n=x^2+y^2$. Thus $d=4y^2>0$ and
$2\alpha-T=2iy=\pm i\sqrt d$. Both integer multiplication and
the given action preserve $\Lambda$, hence so does their difference
$\beta$ and its negative. For $\omega\ne0$, the vectors
$\omega,i\sqrt d\,\omega$ are real-linearly independent and
orthogonal. Their integer span is contained in $\Lambda$, proving
the assertion.
\end{proof}

The full CM lattice need not have this rectangular basis. For example,
with $\zeta=(-1+i\sqrt3)/2$, the lattice
$\mathbf Z+\zeta\mathbf Z$ contains
$\mathbf Z+i\sqrt3\mathbf Z$ as a subgroup of index two.
The proof needs only the sublattice. Nor need $d$ be squarefree,
fundamental, or minimal. Even for the square lattice, using
$\alpha=i$ in the trace-zero construction gives $d=4$;
the specialized lemniscate proof uses the stronger already known
action by $i$, allowing $d=1$.

The actual lattice-action and discriminant arguments are already present
in \lean{EllipticCMAction.exists_integer_quadratic},
\lean{EllipticCMDiscriminant.exists_quadratic_generator}, and
\lean{EllipticCMPeriodGrid.exists_quadratic_lattice_action}.
The latter retains both possible signs before choosing the positive
imaginary generator.

\begin{proposition}[Exact grid size, radius, and separation]
Fix such a positive $d$, a nonzero period $\omega$, and $N\ge1$.
Put $K=N^2$ and
\[
 \lambda_{ab}=a+bi\sqrt d,
 \qquad z_{ab}=\lambda_{ab}\omega,
 \qquad 0\le a,b<N.
\]
These are $K$ distinct lattice periods and satisfy
\[
 |\lambda_{ab}|^2=a^2+db^2\le(1+d)K,
 \qquad |z_{ab}|\le |\omega|\sqrt{1+d}\,N.
\]
Distinct grid points satisfy $|z_{ab}-z_{a'b'}|\ge|\omega|$.
Consequently their open disks of radius $|\omega|/2$ are disjoint,
and any disk of radius $R$ contains at most
$(2R/|\omega|+1)^2$ grid points.
\end{proposition}

\begin{proof}
Membership follows from the sublattice inclusion. Taking real and
imaginary parts proves injectivity. The squared norm formula proves
the radius estimate. For distinct indices,
$(a-a')^2+d(b-b')^2$ is a positive integer and hence at least one;
multiply its square-root inequality by $|\omega|$.
For the last bound compare the disjoint areas of the radius-$|\omega|/2$
disks with that of the concentric enlarged disk of radius
$R+|\omega|/2$.
\end{proof}

These cardinality, radius and separation estimates are proved in
\lean{QuadraticPeriodGrid.card_grid}, \lean{norm_periodCenter_le},
and \lean{norm_omega_le_periodCenter_sub}. Actual periodicity of
$\wp_\Lambda$ and $\wp'_\Lambda/2$ is supplied by
\lean{EllipticCMPeriodGrid.exists_periodic_quadratic_grid}.
The elementary disk-count consequence is a written argument here.

For a real nonzero period $\omega$ and rational approximants $p_i/q_i$,
the centers and displacements are
\[
 c_{ji}=\lambda_jp_i/q_i,
 \qquad \delta_{ji}=c_{ji}-z_j
                =\lambda_j(p_i/q_i-\omega).
\]
If every $p_i\ne0$ and $q_i\ne0$, the centers are coordinatewise
distinct. If $|p_i/q_i-\omega|\le q_i^{-\nu}$, $q_i\ge2$, and
$\nu\ge0$, then, with $w_i=\lceil\log q_i\rceil$,
\[
 |\delta_{ji}|\le\sqrt{1+d}\,N\exp(\nu-\nu w_i).
\]
Indeed $\log q_i\ge w_i-1$. Thus replacing the Gaussian square
changes this displacement constant from $\sqrt2$ to $\sqrt{1+d}$;
it does not change its exponent in $q_i$. The parameter $d$ is fixed
from the curve's lattice action before the interpolation dimension or
approximation sequence is chosen.

\subsection{The quadratic-order norm gives the same arithmetic scale}

Let $\mathcal O_d=\mathbf Z[i\sqrt d]$. It is an actual subring of
$\mathbf C$, since $(i\sqrt d)^2=-d$. Every nonzero element has
modulus at least one, because
\[
 |a+bi\sqrt d|^2=a^2+db^2\in\mathbf Z_{>0}.
\]
Consequently, for a square complex matrix $A$ and positive real
row and column factors $r_i,c_j$, if every
$r_i c_j A_{ij}\in\mathcal O_d$ and $\det A\ne0$, then
\[
 1\le\left(\prod_i r_i\right)
        \left(\prod_j c_j\right)|\det A|,
 \qquad
 \log|\det A|\ge-\sum_i\log r_i-\sum_j\log c_j.
\]
To prove this, take the determinant of the cleared matrix in
$\mathcal O_d$ and use the preceding norm gap. There is no extra
factor of two in the logarithmic estimate and no uncontrolled
algebraic-conjugate estimate. These exact statements are already
formalized in \lean{QuadraticRing.lean} and
\lean{DeterminantLogBound.lean} for general $d$.

The Gaussian entry-clearing calculation in the lemniscate paper
therefore applies to the new order once the actual origin coefficients
and selected entries have been identified. For the origin basis those
coefficients are integral by the earlier construction, and the exact
identification and grade-sensitive clearing are proved in the next section.
This observation does not itself provide a nonzero minor: that is supplied by Theorem~\ref{thm:origin-polar-epa-geometry}.

\section{Exact arithmetic for the selected origin determinant}

The following lemmas establish the proposed arithmetic lower estimate,
conditional on surjectivity of the actual origin packet. They do not assert
that surjectivity. Put $\mathcal O_d=\Z[i\sqrt d]$, where $d\ge1$ is an
integer. Fix positive weights $v_0,w_1,\ldots,w_m$, $0<\theta<1$, and
$H>0$. Write $\mathfrak a_H$ for the ideal of formal series in
$(t,Z_1,\ldots,Z_m)$ whose coefficients vanish at every monomial with
weight
\[
 v_0s+\sum_i (w_i/\theta)\beta_i<H.
\]
The quotient has the indicated finite set of monomials as a basis.
Let $U(t)\in\Z[[t]]$ and $L(t)\in\Q[[t]]$ be the actual origin
coordinate and logarithm constructed above, so $L(0)=0$ and
$n[t^n]L\in\Z$.

\begin{lemma}[Filtered logarithm truncation, including the determinant]
\label{lem:origin-filtered-truncation}
Choose positive integers $T_i$ with $T_iv_0\ge w_i/\theta$, and put
$L_i^{\rm tr}(t)=\sum_{n<T_i}[t^n]L\,t^n$.
For arbitrary centers $c_{ji}$, replacing $L$ by $L_i^{\rm tr}$ in
the $i$th additive coordinate preserves surjectivity of the complete
weighted packet map. More precisely, the full and truncated packet
matrices differ by one fixed invertible row operator, of determinant
one, on each center's complete packet. Therefore every square minor
using all packet rows and any fixed set of source columns has exactly
the same determinant before and after this replacement.
The assertion holds for any source space of polynomials in $u,t,Z_i$;
the integers $T_i$ do not depend on their additive degrees.
\end{lemma}

\begin{proof}
Set $R_i=L-L_i^{\rm tr}$ and define the formal automorphism
\[
 \Phi(t)=t,\qquad \Phi(Z_i)=Z_i+R_i(t).
\]
The inverse has $Z_i\mapsto Z_i-R_i(t)$. Both are well-defined
because $R_i(0)=0$. Every term of $R_i$ has weight at least
$T_iv_0\ge w_i/\theta$, so neither substitution decreases monomial
weight. Consequently $\Phi$ and $\Phi^{-1}$ preserve
$\mathfrak a_H$ and induce inverse operators on the finite quotient.
They fix $U(t)$, $t$, and every constant center, and hence carry
each truncated evaluation to the corresponding full evaluation.

To compute the determinant, order the quotient's monomial basis by
total additive degree $\sum_i\beta_i$, with arbitrary order inside
each degree. In the expansion of
$t^s\prod_i(Z_i+R_i)^{\beta_i}$, the unchanged term is
$t^sZ^\beta$ with coefficient one. Every other term has strictly
smaller additive degree. Discarding monomials in $\mathfrak a_H$
does not alter this fact. Thus the matrix is block triangular with
identity diagonal blocks and has determinant one. On the direct
sum over centers its determinant is still one. This is an identity
of whole packets; individual scalar entries generally change.
\end{proof}

The actual origin-quartic determinant-one refinement is checked in
\lean{QuadraticPi.EllipticOriginQuarticLogDeterminant.det_fullSelectedMinor_eq_truncated}.
It uses the generic \lean{OAI.PiExponent.FormalLogTruncation} machinery
on the same full packets and selected source columns. The existing
lemniscate filtered equivalence remains a reused generic model.
If $S v_0\ge H$, replacing $U$ by its degree-$<S$ truncation
leaves every packet coefficient unchanged: the coordinate difference
belongs to $\mathfrak a_H$, so polynomial evaluation is unchanged
in the quotient. This also has no source-degree-dependent cutoff.

\begin{lemma}[Exact grade-sensitive clearing for a selected origin minor]
\label{lem:origin-grade-clearing}
Let rows be all pairs of a center $j$ and a target monomial
$(s,\beta)$ of weight $<H$. Let their number be $M>0$.
Choose any $M$ source monomials
$u^{a_c}t^{b_c}Z^{\alpha_c}$ with $b_c\in\{0,1\}$.
Choose $p_i\in\Z$, $q_i\ge1$, and $\lambda_j\in\mathcal O_d$,
and take $c_{ji}=\lambda_jp_i/q_i$.
Let $\Delta\ne0$ be the determinant of their full origin packet.
For the above cutoffs $T_i$, put
\[
 \ell_i=\lcm(1,\ldots,T_i-1),\qquad
 J_i=\sum_c\alpha_{ci}-\sum_r\beta_{ri}.
\]
Then every $J_i$ is a nonnegative integer, and
\[
 0\ne \Delta\prod_i(q_i\ell_i)^{J_i}\in\mathcal O_d.
\]
In particular
\[
 \log|\Delta|\ge-\sum_iJ_i(\log q_i+\log\ell_i).
\]
\end{lemma}

\begin{proof}
By Lemma~\ref{lem:origin-filtered-truncation}, use the same columns
in the truncated-log matrix without changing its determinant.
The entry in row $r=(j,s,\beta_r)$ and column $c$ is
\[
 A_{rc}=\left(\prod_i\binom{\alpha_{ci}}{\beta_{ri}}\right)
 [t^s]\left(U(t)^{a_c}t^{b_c}
 \prod_i\bigl(\lambda_jp_i/q_i+L_i^{\rm tr}(t)\bigr)^{
 \alpha_{ci}-\beta_{ri}}\right)
\]
when $\beta_{ri}\le\alpha_{ci}$ for all $i$, and is zero otherwise.
All coefficients of $U$ are integral. Since
$n[t^n]L\in\Z$, the polynomial $\ell_iL_i^{\rm tr}$ has integral
coefficients. Thus
\[
 (q_i\ell_i)(\lambda_jp_i/q_i+L_i^{\rm tr})
   =\ell_i\lambda_jp_i+q_i\ell_iL_i^{\rm tr}
   \in\mathcal O_d[t].
\]
Consequently every nonzero entry becomes an element of
$\mathcal O_d$ after multiplication by
$\prod_i(q_i\ell_i)^{\alpha_{ci}-\beta_{ri}}$.

In the Leibniz expansion of a nonzero determinant, at least one
permutation product is nonzero. For that permutation, all paired
differences $\alpha_{\sigma(r)i}-\beta_{ri}$ are nonnegative;
their sum is $J_i$, proving $J_i\ge0$. Every other nonzero
permutation product has exactly the same total exponents $J_i$.
Multiplication by $\prod_i(q_i\ell_i)^{J_i}$ therefore puts every
term of the Leibniz sum in $\mathcal O_d$, including the zero terms.
This proves the clearing assertion and preserves nonvanishing.

For $z=a+bi\sqrt d\in\mathcal O_d\setminus\{0\}$,
$|z|^2=a^2+db^2$ is a positive integer, so $|z|\ge1$.
Apply this to the cleared determinant and take logarithms.
No passage through an uncontrolled algebraic conjugate, or
field-degree loss in the logarithmic bound, occurs.
\end{proof}

\begin{corollary}[A dimension-independent arithmetic error]
\label{cor:origin-arithmetic-error}
Assume $w_0>0$, $q_i\ge2$, $w_i=\lceil\log q_i\rceil$, and each chosen
source monomial satisfies
\[
 2w_0(a_c+2b_c)+\sum_iw_i\alpha_{ci}\le H.
\]
Define
\[
 \bar b=\frac1{MH}\sum_r\sum_iw_i\beta_{ri},\qquad
 T_i=\lceil F_0w_i/v_0\rceil,
 \quad F_0\ge1/\theta.
\]
With the already audited constant $C_{\rm lcm}=\log4+4$,
\[
 \frac{\log|\Delta|}{MH}
 \ge -(1-\bar b)-\frac{C_{\rm lcm}F_0}{v_0},
 \qquad 0\le\bar b<\theta.
\]
\end{corollary}

\begin{proof}
The source cutoff and nonnegative elliptic weights imply
\[
 \sum_iw_iJ_i\le MH-\sum_r\sum_iw_i\beta_{ri}
       =MH(1-\bar b).
\]
The strict target cutoff gives $\bar b<\theta$. Also
$\log q_i\le w_i$ and, by
\lean{OAI.PiExponent.Arithmetic.log_lcmUpto_le},
\[
 \log\ell_i\le C_{\rm lcm}(T_i-1)
       \le(C_{\rm lcm}F_0/v_0)w_i.
\]
Apply the preceding lemma and the grade bound to obtain the slightly
stronger estimate
\[
 \frac{\log|\Delta|}{MH}
 \ge -(1+C_{\rm lcm}F_0/v_0)(1-\bar b).
\]
Finally $0\le1-\bar b\le1$. The cutoff condition follows from
$T_iv_0\ge F_0w_i\ge w_i/\theta$.
\end{proof}

\status{The actual selected-minor arithmetic estimates and the
full/truncated determinant identity are checked in
\lean{QuadraticPi.EllipticOriginQuarticMinorLower.selectedMinor_arithmetic_lower},
\lean{QuadraticPi.EllipticOriginQuarticLogDeterminant.fullSelectedMinor_arithmetic_lower}
and the determinant-one theorem cited above. EPA supplies nonzero
minors through
\lean{QuadraticPi.EllipticCMCofinalMinors.cofinal_full_minors}.
The actual retained lcm constant is $\log4+4$; no sharper lcm theorem
is needed for the exponent conclusion.}

\section{A same-minor analytic bound on the rectangular CM grid}
\label{sec:cm-same-minor-analytic}

This section concerns a selected minor of the actual origin packet.  It
does not assert that a nonzero minor exists.  The argument is a written
extension of the collision argument in the lemniscate manuscript \cite{lemniscate},
with a full-lattice clearing multiplier in place of its square-lattice
counterpart. The written presentation uses sigma. The checked Lean
implementation uses an explicitly defined shifted Jacobi theta on the
same full lattice; its entire clearing, growth, inverse factors and
same-selected-column transport are proved directly. Equality of those
two clearing functions is not an input. The resulting finite bound is
\lean{QuadraticPi.EllipticOriginPaperUpper.same_minor_paper_upper}.

\subsection{Fixed analytic data and the selected matrix}

Fix a compatible uniformization of the nonsingular elliptic curve by
$\mathbb C/\Lambda$, with additive coordinate $z$ corresponding to the
differential $dx/(2y)$.  Choose a polar point $Q\ne O$ with $Q\ne -Q$,
and write
\[
 u=(x-a)^{-1},\qquad t=y(x-a)^{-2},\qquad
 t^2=F(u),\qquad F(0)=0,\quad F'(0)=1.
\]
Here $a=x(Q)$; the later source exponent $k$ is unrelated to $a$.
For the arithmetic application retain the previously chosen integral
$A,B,a$, with $a^3+Aa+B\ne0$; the analytic bound itself needs only
the stated fixed complex curve and polar point.
The elliptic functions $u(z),t(z)$ are holomorphic at zero and have
$u(0)=t(0)=0$.  Their origin expansions satisfy
\[
 u(z)=U(\tau(z)),\qquad t(z)=\tau(z),\qquad
 L(\tau(z))=z,\qquad \tau'(0)=-1,
\]
where $F(U(T))=T^2$, $U(0)=0$, and
$L'(T)=-1/F'(U(T))$, $L(0)=0$.
These are convergent germs; only their germ identities will be used.

We use the following fixed-curve multiplier input, proved below from the classical sigma zero divisor and quasiperiod formula.  With
$\psi(z)=\sigma_\Lambda(z-z_Q)\sigma_\Lambda(z+z_Q)$, the products
$\psi$, $\psi u$, and $\psi^2t$ are entire.  There are fixed constants
$A,A_1\ge1$ and $C,C_1>0$ such that
\begin{align}
 |g_{D,k,b}(z)|&\le A^D\exp(CD|z|^2),
 &g_{D,k,b}&=\psi^{D-k-2b}(\psi u)^k(\psi^2t)^b,
 \quad k+2b\le D,\label{eq:cm-analytic-multiplier-growth}\\
 |\psi(\gamma)|^{-1}&\le
 A_1\exp(C_1(1+|\gamma|^2)),
 &&\gamma\in\Lambda.\label{eq:cm-analytic-multiplier-inverse}
\end{align}
In particular $\psi$ does not vanish on $\Lambda$.  All multiplier
functions and constants refer to the full lattice $\Lambda$.

For completeness, the classical input used here is that sigma has
simple zeros exactly on $\Lambda$ and satisfies
\[
 \sigma(z+\lambda)=\epsilon_\lambda
 e^{\eta_\lambda(z+\lambda/2)}\sigma(z),
 \qquad \epsilon_\lambda\in\{1,-1\},
\]
with $\lambda\mapsto\eta_\lambda$ additive; see NIST DLMF,
Sections 23.2(ii)--(iii), equations (23.2.15) and (23.2.17),
\url{https://dlmf.nist.gov/23.2}.
Here is the resulting multiplier proof.  The cosets
$z_Q+\Lambda$ and $-z_Q+\Lambda$ are disjoint, since $Q\ne -Q$.
The function $\psi$ has a simple zero on each; these cancel the simple
poles of $u$ and the double poles of $t$ with the powers displayed
above.  Thus all three products are entire.  The two sigma signs
cancel, giving
\[
 \psi(z+\lambda)
   =e^{\eta_\lambda(2z+\lambda)}\psi(z).
\]
On a fixed compact fundamental parallelogram the three entire
products are bounded.  Additivity of $\eta$ on a fixed real basis of
$\Lambda$ gives $|\eta_\lambda|\le C_2|\lambda|$.  Write
$z=z_0+\lambda$ with $z_0$ in that parallelogram and apply the last
formula, once for $\psi$ and $\psi u$ and twice for $\psi^2t$.
This gives a common bound $A\exp(C|z|^2)$ after increasing $A,C$.
The product defining $g_{D,k,b}$ has a total of
$(D-k-2b)+k+b=D-b\le D$ factors, which proves
\eqref{eq:cm-analytic-multiplier-growth}.
Finally $\psi(0)\ne0$ since $Q\ne O$; evaluate the quasiperiod
formula at zero and bound $|\eta_\lambda\lambda|$ to obtain
\eqref{eq:cm-analytic-multiplier-inverse}.
This proof requires the full lattice: a theta function built only
for a proper rectangular sublattice need not cancel all the poles.

Let $0\ne\omega\in\mathbb R\cap\Lambda$, and fix an integer $d\ge1$
for which $i\sqrt d\,\Lambda\subseteq\Lambda$.  For $N\ge1$ set
\[
 K=N^2,\qquad \lambda_j=j_1+i\sqrt d\,j_2,\qquad
 \gamma_j=\omega\lambda_j\in\Lambda
 \quad(0\le j_1,j_2<N),\qquad L_d=\sqrt{1+d}.
\]
Thus $|\lambda_j|\le L_dN$ and $|\gamma_j|\le|\omega|L_dN$.
Fix $m\ge1$, positive weights $w_0,v_0,w_1,\ldots,w_m$, a number
$0<\theta<1$, and a cutoff $H>0$.  Put
\[
 W(\alpha)=\sum_{i=1}^m w_i\alpha_i,\qquad
 \mathcal J_H=\{(s,\beta)\in\mathbb N^{m+1}:
                    v_0s+W(\beta)/\theta<H\},\qquad
 J_H=|\mathcal J_H|,\quad M=KJ_H.
\]
In particular $M\ge1$.  Select any $M$ distinct source monomials
\[
 P_c=u^{k_c}t^{b_c}Z^{\alpha_c},\qquad b_c\in\{0,1\},\qquad
 2w_0(k_c+2b_c)+W(\alpha_c)\le H.                 \tag{S}
\]
Fix integers $p_i$ and $q_i\ge2$, impose
$w_i=\lceil\log q_i\rceil$, and put $c_{ji}=\lambda_jp_i/q_i$.
The actual selected origin matrix is
\[
 A_{(j,s,\beta),c}
 =[T^sV^\beta]\,
 P_c\bigl(U(T),T,c_{ji}+V_i+L(T)\bigr).
\]
Write $\Delta=\det A$ and assume $\Delta\ne0$.
Every row in every strict packet is retained; the source columns
remain this same fixed selection throughout the proof.  Define
\[
 \bar b=\frac1{MH}\sum_{(j,s,\beta)}W(\beta).
\]
The strict target cutoff gives $0\le\bar b<\theta$.

\subsection{Exact transport and finite displacement expansion}

\begin{lemma}[Exact determinant transport of the actual minor]
\label{lem:cm-actual-minor-transport}
Let $D=\lceil H/(2w_0)\rceil$ and
$\delta_{ji}=\lambda_j(p_i/q_i-\omega)$.
The matrix
\[
 B_{(j,s,\beta),c}=[z^s]\,
 g_{D,k_c,b_c}(\gamma_j+z)
 \prod_i\binom{\alpha_{ci}}{\beta_i}
       (\gamma_j+z+\delta_{ji})^{\alpha_{ci}-\beta_i},
\]
with an entry interpreted as zero if some $\beta_i>\alpha_{ci}$,
satisfies
\[
 |\det B|=|\Delta|\prod_j|\psi(\gamma_j)|^{DJ_H}.
\]
\end{lemma}
\begin{proof}
For each fixed $\beta$, substituting $T=\tau(z)$ in a complete
strict packet gives a triangular coefficient transformation with
diagonal $\tau'(0)^s=(-1)^s$.  The inverse substitution preserves the
same strict weighted ideal.  Hence this substitution has determinant
of absolute value one on each center block.  Periodicity gives
$u(\gamma_j+z)=u(z)$ and $t(\gamma_j+z)=t(z)$ as germs, while
$\gamma_j+\delta_{ji}=c_{ji}$.  The reparametrized packet is therefore
the germ of the displayed expression before multiplication by
$\psi(\gamma_j+z)^D$.

Multiplication by that series is again triangular on each full
strict packet, with diagonal $\psi(\gamma_j)^D$.  It has determinant
$\psi(\gamma_j)^{DJ_H}$, and this value is nonzero.  The source bound
(S) ensures $k_c+2b_c\le D$, so the resulting coefficient functions
are precisely the entire functions displayed in the statement.
Taking absolute values proves the identity.  In particular neither
of these row operations changes the selected source columns.
\end{proof}

Let
\[
 \mathcal T_H=\{\xi\in\mathbb N^m:W(\xi)\le H\},\qquad
 T_H=|\mathcal T_H|,
\]
and, using natural-number subtraction, define entire functions
\[
 G_{c,\xi}(\zeta)=
 \left(\prod_i\binom{\alpha_{ci}}{\xi_i}\right)
 g_{D,k_c,b_c}(\zeta)
 \zeta^{\sum_i(\alpha_{ci}-\xi_i)}.
\]
When some $\xi_i>\alpha_{ci}$, this function is zero.  Put
\[
 s(\beta,\xi,\delta_j)=
 \prod_i\binom{\xi_i}{\beta_i}\delta_{ji}^{\xi_i-\beta_i},
\]
again with value zero unless $\beta\le\xi$ coordinatewise.
The binomial identity
\[
 \binom{\alpha_i}{\beta_i}(\zeta+\delta_i)^{\alpha_i-\beta_i}
 =\sum_{\beta_i\le\xi_i\le\alpha_i}
 \binom{\xi_i}{\beta_i}\delta_i^{\xi_i-\beta_i}
 \binom{\alpha_i}{\xi_i}\zeta^{\alpha_i-\xi_i}
\]
shows, by multilinearity in the rows, that
\begin{equation}
 \det B=\sum_{g:\{1,\ldots,M\}\to\mathcal T_H}
 \left(\prod_r s(\beta_r,g(r),\delta_{j(r)})\right)\det B_g,
 \qquad (B_g)_{r,c}=[z^{s_r}]G_{c,g(r)}(\gamma_{j(r)}+z).
 \label{eq:cm-finite-tag-expansion}
\end{equation}
This is an exact finite identity, with no analytic truncation or
replacement of a minor.

\subsection{A direct-additive collision estimate}

\begin{lemma}[Collision bound for direct additive coefficients]
\label{lem:cm-direct-additive-collision}
Suppose finitely many entire functions $G_{c,\xi}$ satisfy
$|G_{c,\xi}(\zeta)|\le\exp(H\mathcal G)$ for $|\zeta|\le R$,
where $R>0$, and all centers satisfy $|\gamma_j|+1\le R/2$.
For a fixed assignment $g$ let $n_\xi=|g^{-1}(\xi)|$.
Set
\[
 \kappa=\frac{\log2}{4},\qquad s_0=2^{-1/2},\qquad
 \rho_H=\frac{\log M+\kappa-\log(1-s_0)}H.
\]
Then
\[
 |\det B_g|\le
 \exp\left(-\kappa\sum_\xi n_\xi^2+
                       MH(\mathcal G+\rho_H)\right).
\]
There is no dependence on $\sum_rs_r$ in this bound.
\end{lemma}
\begin{proof}
Write the Taylor coefficients at zero in normalized form
$a_{c,\xi,\ell}=R^\ell[\zeta^\ell]G_{c,\xi}$.
Cauchy's inequality gives $|a_{c,\xi,\ell}|\le e^{H\mathcal G}$.
The row kernels satisfy
\[
 b_{r,\ell}=[z^{s_r}]
       \left(\frac{\gamma_{j(r)}+z}{R}\right)^\ell,
 \qquad |b_{r,\ell}|\le2^{-\ell},
\]
by applying the unit-circle Cauchy bound to this polynomial.
Thus every entry is an absolutely convergent sum
$\sum_{\ell\ge0}b_{r,\ell}a_{c,g(r),\ell}$.
Expanding the determinant by row multilinearity is justified by
absolute convergence.  A term vanishes if two rows have the same
pair $(g(r),\ell_r)$, since their coefficient vectors are identical.
Consequently every surviving term has
\[
 \sum_r\ell_r\ge\sum_\xi\binom{n_\xi}{2}.
\]
Its coefficient determinant is bounded by $M!e^{MH\mathcal G}$.
Split $2^{-\sum\ell_r}=s_0^{\sum\ell_r}s_0^{\sum\ell_r}$,
bound one factor by the displayed minimum, and sum the other factor
over all $\ell_r\ge0$.  This gives
\[
 |\det B_g|\le M!e^{MH\mathcal G}
 s_0^{\sum_\xi\binom{n_\xi}{2}}(1-s_0)^{-M}.
\]
Now $M!\le M^M$, $\sum_\xi n_\xi=M$, and
$\log s_0=-\log2/2$, which give the claimed formula.

Only the entire cleared functions are evaluated on a circle of
radius one about a center.  The local inverse $\tau$ is not evaluated
there, so no unit-disk convergence assumption on $\tau$ is required.
\end{proof}

\subsection{The finite-cutoff bound}

\begin{theorem}[Explicit same-minor estimate on a CM subgrid]
\label{thm:cm-origin-same-minor-bound}
Assume all fixed data and the selected nonzero minor above.  Suppose
$\nu>2$ and
\[
 |\omega-p_i/q_i|\le q_i^{-\nu}\quad(1\le i\le m),\qquad
 0<w_{\min}\le\min_i w_i.
\]
Fix $0<a_0<1$ and $0<\eta<1$, and set
\[
 L_H=|\{\xi\in\mathbb N^m:W(\xi)\le a_0H\}|,\qquad
 c_H=\frac{\kappa\eta^2M}{HL_H}.
\]
Define fixed-curve constants
\begin{align*}
 r_*&=2(|\omega|L_d+1),& C_w&=4L_dr_*,\\
 B_{\psi}&=\log A+\log A_1+Cr_*^2+C_1(1+|\omega|^2L_d^2).
\end{align*}
Then
\begin{align}
 \frac{\log|\Delta|}{MH}
 &\le E_{\rm an}+e_H+
 \max\{-c_H,-\nu(a_0(1-\eta)-\bar b)\},
 \label{eq:cm-same-minor-final}\\
 E_{\rm an}
 &=\frac{B_{\psi}K}{2w_0}
       +\frac{\nu+\log(C_wK)}{w_{\min}},
 &e_H&=\frac{B_{\psi}K}{H}+\rho_H+\frac{\log T_H}{H}.
 \nonumber
\end{align}
The constants $B_{\psi},C_w$ depend only on the fixed curve, differential,
polar point, period $\omega$, and integer $d$.  They are independent
of $m,N,w_0,v_0,w_i,\theta,H,p_i,q_i,\nu$ and the selected columns.
\end{theorem}

\begin{proof}
Since $w_i=\lceil\log q_i\rceil<\log q_i+1$,
\[
 |\delta_{ji}|\le L_dNq_i^{-\nu}
 \le\exp(\nu+\log(L_dN)-\nu w_i).
\]
Put $C_\delta=\nu+\log(L_dN)>0$.
For a nonzero displacement scalar one has $\beta\le\xi$ and
$|\xi|\le H/w_{\min}$.  Therefore
\[
 |s(\beta,\xi,\delta_j)|
 \le\exp\left(-\nu(W(\xi)-W(\beta))
             +\frac{H(\log2+C_\delta)}{w_{\min}}\right).
 \tag{1}
\]
Here $\prod_i\binom{\xi_i}{\beta_i}\le2^{|\xi|}$ and the positive
coefficient $C_\delta$ allows $|\xi-\beta|\le H/w_{\min}$.
For any nonzero product of row scalars, its assignment necessarily
satisfies $\sum_rW(g(r))\ge MH\bar b$.

Take $R_N=r_*N$.  Since $|\gamma_j|+1\le R_N/2$, the collision
lemma applies.  On $|\zeta|\le R_N$, which has $R_N\ge1$, the
multiplier bound and (S) give
\[
 |G_{c,\xi}(\zeta)|
 \le A^De^{CDR_N^2}(2R_N)^{|\alpha_c|}
 \le\exp(H\mathcal G_H),\qquad
 \mathcal G_H=\frac DH(\log A+CR_N^2)
                   +\frac{\log(2R_N)}{w_{\min}}.
 \tag{2}
\]

For any nonzero scalar assignment, there are two cases.  If
$\sum_rW(g(r))\ge a_0(1-\eta)MH$, multiplying (1) over rows gives
the gain $-\nu(a_0(1-\eta)-\bar b)$ after division by $MH$.
The negative collision term may then be discarded.  Otherwise
strictly more than $\eta M$ rows have $W(g(r))\le a_0H$:
if at least $(1-\eta)M$ rows had greater weight, the assumed strict
upper bound on the sum would be impossible.  Cauchy--Schwarz over
the at most $L_H$ low-weight tags gives
\[
 \sum_\xi n_\xi^2\ge \frac{\eta^2M^2}{L_H}.
\]
The scalar's negative weight contribution may be discarded because
$\sum_rW(g(r))\ge MH\bar b$.  The collision gain is then $-c_H$.
There are $T_H^M$ assignments in the exact finite expansion.
The triangle inequality, (1), (2), and the two-case estimate give
\[
 \frac{\log|\det B|}{MH}\le
 \frac{\log2+C_\delta}{w_{\min}}+\mathcal G_H+\rho_H
 +\frac{\log T_H}{H}
 +\max\{-c_H,-\nu(a_0(1-\eta)-\bar b)\}.
 \tag{3}
\]
This logarithm is defined because the transport lemma and
$\Delta\ne0$ imply $\det B\ne0$.

To remove the common multiplier, the exact determinant identity
and its inverse bound add at most
\[
 \frac DH\{\log A_1+C_1(1+|\omega|^2L_d^2K)\}
\]
to the right side of (3).  Combining with the multiplier part of
$\mathcal G_H$, and using $K\ge1$ and
$D/H\le1/(2w_0)+1/H$, bounds that total by
$B_{\psi}K/(2w_0)+B_{\psi}K/H$.  Finally
\[
 \log2+C_\delta+\log(2R_N)
    =\nu+\log(4L_dr_*N^2)=\nu+\log(C_wK).
\]
Substitution yields the stated inequality.
\end{proof}

\begin{corollary}[Fixed-data asymptotics]
\label{cor:cm-same-minor-asymptotics}
For all the finite data held fixed as $H\to\infty$,
$e_H\to0$, and
\[
 c_H\longrightarrow
 \frac{\kappa\eta^2K\theta^m}{(m+1)v_0a_0^m}.
\]
These limits remain true along any sequence of permitted cutoffs
tending to infinity on which a selected nonzero minor exists.
\end{corollary}
\begin{proof}
Weighted-simplex lattice counting gives
\[
 M\sim\frac{K\theta^mH^{m+1}}{(m+1)!v_0\prod_iw_i},\qquad
 L_H\sim\frac{a_0^mH^m}{m!\prod_iw_i},\qquad
 T_H=O(H^m).
\]
The first two equivalences give the limit of $c_H$; their
polynomial growth and that of $T_H$ give
$\log M/H\to0$ and $\log T_H/H\to0$, proving $e_H\to0$.
No uniformity in varying finite data is asserted in taking this
limit: all auxiliary choices precede the final large-cutoff choice.
\end{proof}

\section{The exponent-two parameter choice and conclusion}
\label{sec:transfer-closure}
All finite auxiliary data will be fixed before the final cutoff tends
to infinity. We retain the constants in the proved estimates, including
the quartic source cost and the factor $\kappa=\log2/4$.
\begin{theorem}[Normalized CM-period exponent]
\label{thm:transfer-conditional-closure}
Fix a normalized integral nonsingular short Weierstrass equation, its
compatible period lattice, a nonzero real irrational period $\omega$,
and a CM action on that lattice. Choose an integer $a_*$ with
$a_*^3+Aa_*+B\ne0$ and use its actual origin chart. Then $\mu(\omega)=2$.
The geometric input in the proof is Theorem~\ref{thm:origin-curve-inequality}
and its actual EPA consequence, Theorem~\ref{thm:origin-polar-epa-geometry}.
The arithmetic bound is Corollary~\ref{cor:origin-arithmetic-error};
the analytic bound is Theorem~\ref{thm:cm-origin-same-minor-bound}.
\end{theorem}

\begin{proof}
Fix $\nu>2$. Suppose there are infinitely many rational approximations
$|\omega-p/q|\le q^{-\nu}$, with positive integer $q$. Their
denominators are unbounded: for bounded $q$ only finitely many integers
$p$ satisfy this inequality. Since $\omega\ne0$, after discarding
finitely many approximants we have $p\ne0$ and $q\ge2$.

Choose rational parameters $b,h>0$ such that
\[
 \tfrac12<b<1-\nu^{-1},\qquad
 \theta=1-h,\qquad a_0=1-bh,\qquad 0<a_0,\theta<1.
\]
For sufficiently small $h$,
\[
 a_0^2<\theta<a_0,
 \qquad \nu(a_0-\theta)>1-\theta.
\]
Indeed $a_0^2-\theta=(1-2b)h+b^2h^2$, while the second
strict inequality is $\nu(1-b)>1$. Choose $0<\eta<1$ sufficiently
small that the fixed gap
\[
 g=\nu\bigl(a_0(1-\eta)-\theta\bigr)-(1-\theta)>0
\]
remains positive. Choose rational numbers $C_*,B_*$ in the nonempty
intervals
\[
 a_0/\theta<C_*<a_0^{-1},\qquad
 a_0<B_*<\min\{C_*^{-1},C_*\theta\}.
\]
Fix rational numbers $C>2$ and $F_0\ge\theta^{-1}$. For large integers
$m$ set
\[
 N=\lfloor C_*^{m/2}\rfloor,\quad K=N^2,\quad
 w_0=B_*^{-m},\quad v_0=CK\theta^m w_0.
\]
The base parameters are rational. Since $C_*>1$ and $N^2/C_*^m\to1$,
the interval conditions give
\[
 K\theta^m\to0,\qquad
 K(w_0/v_0)\theta^m=1/C<1/2,\qquad
 K/w_0\to0,\qquad v_0\to\infty.
\]
Thus the EPA base inequalities hold for all sufficiently large $m$.
The limiting collision rate from the analytic estimate is
\[
 c_\infty=\frac{\kappa\eta^2K\theta^m}
 {(m+1)v_0a_0^m}
 =\frac{\kappa\eta^2}{C(m+1)}(B_*/a_0)^m\longrightarrow\infty,
 \qquad \kappa=\frac{\log2}{4}.
\]

Let $B=B_\psi$ and $C_w$ be the fixed constants of that estimate; their
dependence is only on the already fixed lattice, chart, multiplier,
$\omega$, and CM integer $d$. In particular they are independent of
$m,N,w_0,v_0$ and the later approximants. The arithmetic and analytic
errors are
\[
 E_{\rm ar}=\frac{(\log4+4)F_0}{v_0},\qquad
 E_{\rm an}=\frac{BK}{2w_0}
       +\frac{\nu+\log(C_wK)}{w_*},\qquad w_*:=\min_i w_i.
\]
First fix a sufficiently large finite $m$ so that the base inequalities
hold, $c_\infty>2$, and
$E_{\rm ar}+BK/(2w_0)<\min(g/4,1/4)$.
Then successively choose $m$ approximants with
$w_i=\lceil\log q_i\rceil$ above the EPA thresholds and with
$w_*$ sufficiently large that
\[
 E:=E_{\rm ar}+E_{\rm an}<\min(g/2,1/2).
\]
Unbounded denominators allow each of these finitely many choices;
the threshold for $w_i$ sees only already chosen earlier weights.
For these weights choose the center-independent EPA modulus
$D_{\rm EPA}>0$. Now take the centers $c_{ji}=\lambda_jp_i/q_i$;
they are coordinatewise distinct.

For every sufficiently large multiple $H$ of that modulus, EPA
gives a nonzero minor $\Delta_H$ using
all target rows and some source columns of degree at most $H$.
The chosen columns may depend on $H$: both estimates hold for every
such selection. Write $M_H$ for the number of rows and
$0\le\bar b_H<\theta$ for their normalized mean transverse weight.
The exact clearing argument and the same-column analytic identity give
\begin{align*}
 \frac{\log|\Delta_H|}{M_HH}
 &\ge -(1-\bar b_H)-E_{\rm ar},\\
 \frac{\log|\Delta_H|}{M_HH}
 &\le E_{\rm an}+e_H+
 \max\{-c_H,-\nu(a_0(1-\eta)-\bar b_H)\},
\end{align*}
where $e_H\to0$ and $c_H\to c_\infty$ for the now fixed finite
data. Subtracting the lower side from the collision alternative gives
at most $1+E+e_H-c_H<0$ eventually. Subtracting it from the other
alternative gives at most
\[
 1-\nu a_0(1-\eta)+(\nu-1)\theta+E+e_H
   =-g+E+e_H<0
\]
eventually. Either is impossible for the same nonzero determinant.
Therefore there cannot be infinitely many exponent-$\nu$
approximants. As $\nu>2$ was arbitrary, $\mu(\omega)\le2$.
The assumed irrationality and Dirichlet approximation give the reverse
inequality.
\end{proof}

The qualitative irrationality input is the classical CM consequence of Siegel's theorem explained in Section~\ref{sec:main}. A specified
rational invariant differential is reduced to the normalization used
above by the previously proved rational change of model and rational
scaling invariance of $\mu$. Arbitrary complex rescaling is not permitted.

\begin{proof}[Proof of Theorem~\ref{thm:cm-main}]
Proposition~\ref{prop:p1-rational-model} reduces the specified rational
curve and differential to \eqref{eq:cubic}, changing periods by one
nonzero rational factor. Proposition~\ref{prop:p1-period-lattice}
supplies its compatible actual lattice and nonreal CM action.
There is an integer $a_*$ with $a_*^3+Aa_*+B\ne0$, since a cubic
has at most three roots. The integral origin construction and all
geometric and analytic lemmas therefore apply. The nonzero real
normalized period is irrational by the classical result in
Section~\ref{sec:main}. The preceding theorem proves its exponent
is two. Rational scaling invariance gives the stated conclusion
for the original differential and period.
\end{proof}

\clearpage
\part{Powers of the lemniscate period}\label{part:powers}
\section{Statement and scope}

This part uses \eqref{eq:overview-mu}. All logarithms are natural
unless a base is displayed. Its local chart, weights, and selected
minor are introduced afresh and remain fixed throughout each comparison.

\begin{theorem}\label{pow:thm:main}
Let
\[
 \Omega=2\int_0^1\frac{dt}{\sqrt{1-t^4}}.
\]
There is a constant \(C>0\), depending only on the fixed lemniscate
lattice, such that, for every integer \(n\geq1\), there is an integer
\(Q(n)\) with
\[
 \left|\Omega^n-\frac pq\right|>q^{-C n\ceil{\log(2n)}}
 \qquad(p\in\Z,\ q\geq Q(n)).
\]
In particular \(\Omega^n\) is irrational and
\[
 \mu(\Omega^n)\leq C n\ceil{\log(2n)}
                  =O(n\log(2n)).
\]
The proof below allows
\begin{equation}\label{pow:eq:mainC}
 C=6\cdot10^{10}G^2,\qquad
 G=\max(1,\ceil{C_0}),
\end{equation}
where the lattice constant \(C_0\) is defined in
\eqref{pow:eq:analyticconstants}.
\end{theorem}

The use of EPA in this theorem is exactly its one-additive-coordinate
case, with arbitrary distinct complex centers. Its precise statement is
recorded in Proposition~\ref{pow:prop:epa}. We cite its proof from the
companion manuscript \cite[Section~3]{lemniscate}, rather than assert a
new proof of EPA here. The standard period normalization from that
manuscript is also recorded explicitly. Every additional estimate used
for the power extension is proved below. The two-root cubic power
argument \cite{cubicdraft} is not used as a theorem.

\section{The local chart and the inherited interpolation theorem}

Put \(\Lambda=\Omega(\Z+\ii\Z)\) and \(z_T=(1+\ii)\Omega/2\).
The period normalization in \cite[Section~2]{lemniscate} gives
\[
 g_2(\Lambda)=4,\qquad g_3(\Lambda)=0.
\]
Thus \(x(z)=\wp_\Lambda(z)\), \(y(z)=\wp'_\Lambda(z)/2\) satisfy
\[
 y^2=x^3-x,\qquad x'=2y,\qquad y'=3x^2-1,
 \qquad x(z_T)=y(z_T)=0.
\]
This normalization concerns the displayed integral \(\Omega\), not an
independently chosen positive rescaling of a square lattice.

Let \(x_0(t)\in\Z[[t]]\) be the centered solution of
\[
 x_0(t)^3-x_0(t)=t^2,\qquad x_0(0)=0,
\]
and define \(\ell(t)\in\Q[[t]]\) by
\[
 \ell(0)=0,\qquad
 \ell'(t)=\frac1{3x_0(t)^2-1}.
\]
The inverse series construction for \(X^3-X\) proves the integrality
of \(x_0\). Consequently
\begin{equation}\label{pow:eq:integralprimitive}
 j[t^j]\ell(t)\in\Z\quad(j\geq1),\qquad \ell'(0)=-1.
\end{equation}
Near \(s=0\), centered uniqueness and the differential equations give
\begin{equation}\label{pow:eq:germ}
 x_0(y(z_T+s))=x(z_T+s),\qquad
 \ell(y(z_T+s))=s.
\end{equation}
Indeed the cubic gives the first identity; differentiation of the
second gives \(1\), and its constant term is zero.

Fix positive rational numbers \(w_0,v_0,\theta\), with \(0<\theta<1\),
and an integer \(K\geq1\). For a positive rational additive weight \(w\),
define
\begin{align}
 \CC&=\{(a,b,\alpha): a,\alpha\in\N,\ b\in\{0,1\},
                  \ w_0(2a+3b)+w\alpha\leq H\},\label{pow:eq:columns}\\
 \BB&=\{(k,\beta)\in\N^2:
                    v_0k+(w/\theta)\beta<H\}.\label{pow:eq:packet}
\end{align}
The source is spanned by \(x^ay^bX^\alpha\) with index in \(\CC\).
The target consists of the coefficients of the packet \(\BB\) at
each center.

\begin{proposition}[EPA, one-coordinate case]\label{pow:prop:epa}
Suppose
\begin{equation}\label{pow:eq:epabase}
 K\theta<1,\qquad K(w_0/v_0)\theta<1.
\end{equation}
There is a rational threshold \(T_{\mathrm{EPA}}\), depending only on
\(K,w_0,v_0,\theta\), with the following property.
For every rational \(w>T_{\mathrm{EPA}}\), there is a positive integer
\(D\) such that, for every set of distinct centers \(c_1,\ldots,c_K\in\C\),
there is \(H_0\) for which the map
\[
 P\longmapsto
 \bigl([t^ku^\beta]\,
  P(x_0(t),t,c_j+u+\ell(t))\bigr)_{j,(k,\beta)\in\BB}
\]
is surjective whenever \(H\geq H_0\) and \(D\mid H\).
\end{proposition}

This is \cite[Theorem~3.1]{lemniscate} with \(m=1\). In particular the
threshold precedes the centers. The size \(H_0\) is allowed to depend
on them. No uniform bound on \(H_0\) will be required.

Put
\begin{equation}\label{pow:eq:truncations}
 U=\ceil{\frac{w}{\theta v_0}},\qquad T=U+1,\qquad
 S_H=\ceil{\frac{H}{v_0}},
\end{equation}
and write \(x_{<S_H}\), \(\ell_{<T}\) for the corresponding truncations.
Let \(N\geq1\), \(K=N^2\), and enumerate
\(\lambda_j\in\{a+\ii b:0\leq a,b<N\}\).
For any \(\omega\in\C\setminus\{0\}\), set \(c_j=\lambda_j\omega\).
The rows are \(\RR=\{1,\ldots,K\}\times\BB\), with \(M_H=|\RR|>0\).
Define the matrix
\begin{equation}\label{pow:eq:entry}
 A_H((j,k,\beta),(a,b,\alpha))
 =\binom{\alpha}{\beta}[t^k]
       x_{<S_H}(t)^a t^b
         (\lambda_j\omega+\ell_{<T}(t))^{\alpha-\beta}.
\end{equation}
An entry is defined to be zero if \(\beta>\alpha\).

\begin{lemma}\label{pow:lem:minors}
Under \eqref{pow:eq:epabase}, for \(w>T_{\mathrm{EPA}}\) there are cofinally
many positive integers \(H\) for which \(A_H\) has a nonzero square
minor using every row.
\end{lemma}
\begin{proof}
The centers \(\lambda_j\omega\) are distinct.
The base truncation is exact on the rows since \(S_Hv_0\geq H\).
The coordinate substitution
\(u\mapsto u+\ell(t)-\ell_{<T}(t)\) preserves the strict weighted
packet quotient: its tail starts at degree \(T\), and
\(Tv_0>w/\theta\). Its inverse has the same property. It is a
unitriangular change of packet coordinates.
Proposition~\ref{pow:prop:epa} therefore supplies full row rank for the
truncated matrix on an eventual divisibility progression.
Select \(M_H\) independent columns.
\end{proof}

Fix such a selection \(\iota:\RR\hookrightarrow\CC\), write
\(\del=\det(A_H(r,\iota(s)))_{r,s\in\RR}\), and put
\begin{equation}\label{pow:eq:meanW}
 \Wbar=\frac{w}{M_HH}\sum_{(j,k,\beta)\in\RR}\beta.
\end{equation}
Directly from the packet inequalities,
\begin{equation}\label{pow:eq:rowbounds}
 0\leq\Wbar\leq\theta,\qquad
 \sum_{r\in\RR}k_r\leq M_HH/v_0.
\end{equation}
Every estimate below holds for every allowed column selection; no
subsequent step selects different columns.

\section{Degree-aware clearing}

\begin{lemma}[An elementary prime estimate]\label{pow:lem:cheb}
For \(x\geq0\), put
\(\vartheta(x)=\sum_{p\leq x,\ p\ {\rm prime}}\log p\). Then
\(\vartheta(x)\leq6x\).
\end{lemma}
\begin{proof}
For any integer \(m\geq1\), the product of the primes in \((m,2m]\)
divides \(\binom{2m}{m}\leq4^m\).
At dyadic arguments this gives
\(\vartheta(2^s)\leq(2^s-1)\log4\).
For \(x>1\), choose \(s\) with \(2^{s-1}<x\leq2^s\).
Then \(\vartheta(x)<2x\log4<4x<6x\). For \(0\leq x\leq1\) it is zero.
\end{proof}

\begin{lemma}[Prime-wise clearing]\label{pow:lem:clearing}
Fix \(w,v_0>0\) and \(0<\theta<1\), and define \(U\) by
\eqref{pow:eq:truncations}. For every positive integer \(H\), put
\[
 R_H=\floor{H/w},\qquad
 e_p(H)=\min\left(
       R_H\floor{\log_p U},\ \floor{\frac{H}{v_0p}}\right),\qquad
 \mathcal D_H=\prod_{\substack{p\leq U\\p\ {\rm prime}}}p^{e_p(H)}.
\]
Let \(L(t)=\sum_{j=1}^U(a_j/j)t^j\), with \(a_j\in\Z\).
If \(0\leq r\leq R_H\) and \(0\leq k\leq H/v_0\), then
\[
 \mathcal D_H[t^k]\bigl(F(t)L(t)^r\bigr)\in\Z
 \qquad\text{for every }F(t)\in\Z[[t]].
\]
\end{lemma}
\begin{proof}
Every contributing term has denominator dividing \(j_1\cdots j_r\),
where \(1\leq j_\ell\leq U\) and
\(\sum_\ell j_\ell\leq k\). For every prime \(p\),
\[
 \sum_\ell v_p(j_\ell)\leq
                  R_H\floor{\log_p U}.
\]
Also \(v_p(j)\leq j/p\): if \(j=p^a m\) and \(a\geq1\), then
\(a\leq p^{a-1}\leq j/p\); if \(a=0\) the assertion is immediate.
Hence
\[
 \sum_\ell v_p(j_\ell)\leq \frac{\sum_\ell j_\ell}{p}
                     \leq \frac{H}{v_0p}.
\]
The exponent is integral, so it is at most \(e_p(H)\).
Thus every contributing denominator divides \(\mathcal D_H\).
The case \(r=0\) is included.
\end{proof}

\begin{lemma}[Uniform logarithmic budget]\label{pow:lem:Dsize}
Fix \(v_0>0\) and \(0<\theta<1\). For all sufficiently large \(w\),
uniformly for every positive integer \(H\),
\begin{equation}\label{pow:eq:Dsize}
 \frac{\log\mathcal D_H}{H}
 \leq \frac{L_\theta}{v_0},\qquad
 L_\theta=13+6\log(2/\theta).
\end{equation}
\end{lemma}
\begin{proof}
Put \(B=w/v_0\) and assume \(B\geq2\).
For primes \(p\leq B\), use the first term in the minimum defining
\(e_p\). The prime-power decomposition and Lemma~\ref{pow:lem:cheb} give
\begin{align*}
 \sum_{p\leq B}\floor{\log_p U}\log p
 &\leq\vartheta(B)+
       \sum_{a=2}^{\floor{\log_2U}}\vartheta(U^{1/a})\\
 &\leq6B+6\sqrt U\,\log_2U.
\end{align*}
For \(B<p\leq U\), use the second term. Partial summation gives
\begin{align*}
 \sum_{B<p\leq U}\frac{\log p}{p}
 &=\frac{\vartheta(U)}U-\frac{\vartheta(B)}B+
              \int_B^U\frac{\vartheta(t)}{t^2}\,dt\\
 &\leq6(1+\log(U/B)).
\end{align*}
Since \(R_H/H\leq1/w\), these bounds imply, for every \(H>0\),
\begin{equation}\label{pow:eq:Dfinite}
 \frac{\log\mathcal D_H}{H}
 \leq\frac6{v_0}(2+\log(U/B))
                  +\frac{6\sqrt U\,\log_2U}{w}.
\end{equation}
For fixed \(\theta,v_0\) and \(w\to\infty\),
\[
 \frac UB\longrightarrow\frac1\theta,\qquad
 \frac{6v_0\sqrt U\,\log_2U}{w}\longrightarrow0.
\]
For sufficiently large \(w\), we have \(U/B\leq2/\theta\) and the
last term of \eqref{pow:eq:Dfinite} is at most \(1/v_0\).
This proves \eqref{pow:eq:Dsize}.
The threshold does not depend on \(H\).
\end{proof}

\begin{remark}
The total index constraint is essential. Clearing each factor
independently would discard \(\sum j_\ell\leq k\) and retain the much
larger cost \(O(1/(\theta v_0))\).
Only the elementary prime estimate above was used; no prime-number
asymptotic in arithmetic progressions is needed.
\end{remark}

\section{One Kummer root and the arithmetic lower bound}

Fix \(n\geq1\) and \(\nu>0\), and suppose
\begin{equation}\label{pow:eq:badapprox}
 \left|\Omega^n-\frac pq\right|\leq q^{-\nu},
 \qquad q\geq2,\quad p/q\geq1.
\end{equation}
Let \(\omega=(p/q)^{1/n}\geq1\), the positive real root.
Since \(\Omega>2\), the condition \(p/q\geq1\) holds automatically
under \eqref{pow:eq:badapprox}. The factorization of two powers gives
\begin{equation}\label{pow:eq:root}
 |\Omega-\omega|\leq\frac1n|\Omega^n-\omega^n|
                 \leq q^{-\nu},\qquad
 \omega\leq\Omega+1.
\end{equation}
For the last inequality use
\(\omega^n\leq\Omega^n+1\leq(\Omega+1)^n\).

Set
\[
 w=\ceil{\log q},\qquad
 L=\Q(\ii,\omega),\qquad d=[L:\Q(\ii)]\leq n.
\]
Every embedding \(\sigma:L\hookrightarrow\C\) fixing \(\Q(\ii)\)
sends \(\omega\) to another root of \(X^n-p/q\). Therefore
\begin{equation}\label{pow:eq:conjmod}
 |\sigma(\omega)|=\omega\leq\Omega+1.
\end{equation}

\begin{lemma}[Exact norm inequality]\label{pow:lem:norm}
If \(\del\ne0\), and
\[
 E=\sum_{s\in\RR}\alpha_{\iota(s)}
                         -\sum_{r\in\RR}\beta_r,
\]
then
\begin{equation}\label{pow:eq:norm}
 \mathcal D_H^{dM_H}\,q^{(d/n)E}
       \prod_{\sigma:L\hookrightarrow\C\,/\,\Q(\ii)}
                   |\sigma(\del)|\geq1.
\end{equation}
\end{lemma}
\begin{proof}
Adjoin \(\chi=q^{1/n}>0\) to \(L\), put
\(\widetilde L=L(\chi)\), and let \(d_2=[\widetilde L:L]\).
Both \(\chi\) and \(\pi_0=\chi\omega=p^{1/n}\) are algebraic integers.
For a nonzero entry with \(\alpha\geq\beta\),
\[
 \chi^{\alpha-\beta}
       (\lambda_j\omega+\ell_{<T}(t))^{\alpha-\beta}
 =\sum_{r=0}^{\alpha-\beta}\binom{\alpha-\beta}{r}
       (\lambda_j\pi_0)^{\alpha-\beta-r}
                       \chi^r\ell_{<T}(t)^r.
\]
Here \(r\leq\alpha\leq\floor{H/w}\), and the row order is
\(k<H/v_0\).
Lemma~\ref{pow:lem:clearing}, together with the integrality of
\(x_{<S_H}\), proves that
\(\mathcal D_H\chi^{\alpha-\beta}A_H(r,c)\) is integral over \(\Z\).
Zero entries cause no exception. Consequently
\[
 \Psi=\mathcal D_H^{M_H}\chi^E\del
\]
is a nonzero algebraic integer: it is the determinant of the entrywise
integral matrix obtained by multiplying row \(r\) by \(\chi^{-\beta_r}\)
and column \(s\) by \(\chi^{\alpha_{\iota(s)}}\).

The relative norm of \(\Psi\) from \(\widetilde L\) to \(\Q(\ii)\)
is a nonzero Gaussian integer, whose absolute value is at least \(1\).
Every embedding over \(\Q(\ii)\) has
\(|\sigma(\chi)|=q^{1/n}\). Each embedding of \(L\) has exactly \(d_2\)
extensions to \(\widetilde L\). Thus
\[
 1\leq |N_{\widetilde L/\Q(\ii)}\Psi|
   =\left(\mathcal D_H^{dM_H}q^{(d/n)E}
                           \prod_\sigma|\sigma(\del)|\right)^{d_2}.
\]
Taking the positive \(d_2\)-th root proves \eqref{pow:eq:norm}.
This argument does not assume \(E\geq0\) or that \(p/q\) is reduced.
\end{proof}

\begin{lemma}[Conjugate bound]\label{pow:lem:conj}
Put \(C_{\mathrm{conj}}=2\sqrt2(\Omega+1)+2\). Every embedding
\(\sigma:L\hookrightarrow\C\) over \(\Q(\ii)\) satisfies
\[
 \frac{\log|\sigma(\del)|}{M_HH}
 \leq\frac{\log2}{v_0}
       +\frac{\log(C_{\mathrm{conj}}K)}w+\frac{\log M_H}{H}.
\]
\end{lemma}
\begin{proof}
We first justify uniform bounds for all truncation levels on
\(|t|\leq1/2\).
Put \(u=-x_0\); then \(u=t^2+u^3\).
The Picard iterates \(P_0=0\), \(P_{r+1}=t^2+P_r^3\) have nonnegative
coefficients and converge coefficientwise to \(u\).
Inductively \(P_r(1/2)\leq1/3\), because
\(1/4+(1/3)^3<1/3\).
Every finite coefficient sum of \(u\) at \(1/2\) is consequently at
most \(1/3\). Also
\[
 -\ell'=\frac1{1-3u^2}
\]
has nonnegative coefficients. Its finite sums at \(1/2\) are at most
\(3/2\), by the same Picard majorant and a geometric sum.
It follows that
\[
 |x_{<S_H}(t)|\leq1/3<1,\qquad
 |\ell_{<T}(t)|\leq(1/2)(3/2)=3/4<1.
\]
These are bounds for finite polynomials; no convergence of an
unjustified substituted series is being assumed.

By \eqref{pow:eq:conjmod} and \(|\lambda_j|\leq\sqrt2N\), Cauchy's
coefficient estimate on \(|t|=1/2\) bounds an entry of \(\sigma(A_H)\)
by
\[
 2^k\bigl(2\sqrt2N(\Omega+1)+2\bigr)^\alpha
       \leq2^k(C_{\mathrm{conj}}K)^\alpha.
\]
We used \(\binom{\alpha}{\beta}\leq2^\alpha\), and replaced
\(\alpha-\beta\) by \(\alpha\) in a factor at least \(1\).
Since \(k\leq H/v_0\) and \(\alpha\leq H/w\), the determinant
expansion and \(M_H!\leq M_H^{M_H}\) give the claimed estimate.
Every conjugate determinant is nonzero because \(\sigma\) is injective.
\end{proof}

\begin{proposition}[Arithmetic lower bound]\label{pow:prop:arith}
For \(w\) above the threshold in Lemma~\ref{pow:lem:Dsize}, every nonzero
selected minor satisfies
\begin{equation}\label{pow:eq:lower}
 \frac{\log|\del|}{M_HH}
 \geq -(1-\Wbar)-\Ear-\frac{n\log M_H}{H},
\end{equation}
where
\begin{equation}\label{pow:eq:Ear}
 \Ear=\frac{n(L_\theta+\log2)}{v_0}
            +\frac{n\log(C_{\mathrm{conj}}K)+\theta}{w}.
\end{equation}
\end{proposition}
\begin{proof}
Since \(w-1\leq\log q\leq w\), source and row budgets give
\[
 \frac{E\log q}{M_HH}
 \leq1-\Wbar+\frac{\theta}{w}.
\]
The right side is positive because \(\Wbar\leq\theta<1\).
Taking logarithms in \eqref{pow:eq:norm} and bounding all nonidentity
conjugates by Lemma~\ref{pow:lem:conj} yields
\[
 \frac{\log|\del|}{M_HH}
 \geq-d\frac{\log\mathcal D_H}{H}
      -\frac dn\frac{E\log q}{M_HH}
      -(d-1)\left(
        \frac{\log2}{v_0}+\frac{\log(C_{\mathrm{conj}}K)}w+
                                      \frac{\log M_H}{H}\right).
\]
Use \(d\leq n\), \(d/n\leq1\), positivity of
\(1-\Wbar+\theta/w\), and Lemma~\ref{pow:lem:Dsize}.
In particular no sign assumption on \(E\) is needed to replace
the factor \(d/n\).
\end{proof}

\section{The analytic estimate at algebraic centers}

We prove the needed analytic estimate for arbitrary \(\omega\) with
\(|\omega-\Omega|\leq q^{-\nu}\). This avoids treating a statement
restricted to rational centers as though it already quantified over
algebraic centers.

\subsection{A fixed entire multiplier}

For completeness, the multiplier of the companion paper can be used
without modification:
\[
 \Theta(\zeta)=\sum_{m\in\Z}e^{-\pi m^2+2\pi\ii m\zeta},\quad
 \zeta(z)=z/\Omega-(1+\ii)/2,\quad S(z)=\Theta(\zeta(z)).
\]
Absolute locally uniform convergence gives entire functions. Index shifts
give
\[
 \Theta(\zeta+1)=\Theta(\zeta),\quad
 \Theta(\zeta+\ii)=e^{\pi-2\pi\ii\zeta}\Theta(\zeta),\quad
 \Theta(-\zeta)=\Theta(\zeta).
\]
These are also the usual theta-three identities
\cite[Sections~20.2 and 23.6]{dlmf}.
They imply \(\Theta((1+\ii)/2)=0\), hence \(S\) vanishes on \(\Lambda\).
The modulus \(e^{-\pi(\operatorname{Im}\zeta)^2}|\Theta(\zeta)|\)
is periodic under \(1,\ii\).
At \(z_j=z_T+\Omega\lambda_j\),
\[
 |S(z_j)|=e^{\pi(\operatorname{Im}\lambda_j)^2}\Theta(0)
                       \geq\Theta(0)>0.
\]
Fix \(U_*\geq1\) with \(|S(z_j)|^{-1}\leq U_*\) for every grid.

The functions \(X_c=S^2x\), \(Y_c=S^3y\) extend to entire functions,
because the poles of \(x,y\) on \(\Lambda\) have orders \(2,3\).
For \(F=S,X_c,Y_c\), with respective weights \(h=1,2,3\), the
continuous modulus
\(e^{-h\pi(\operatorname{Im}\zeta(z))^2}|F(z)|\)
is lattice-periodic. Boundedness on a fundamental square therefore
gives constants \(C_S,C_X,C_Y\). With
\(C_g=\max(1,C_S,C_X,C_Y)\), for \(2a+3b\leq d\) the entire function
\[
 g_{d,a,b}=S^{d-2a-3b}X_c^aY_c^b
\]
satisfies
\begin{equation}\label{pow:eq:ggrowth}
 |g_{d,a,b}(z)|
 \leq C_g^d e^{\pi d(\operatorname{Im}\zeta(z))^2}
 \leq C_g^d e^{\pi d(|z|/\Omega+1)^2}.
\end{equation}
Off the lattice it equals \(S^d x^a y^b\).
All constants here precede \(n,N,w_0,v_0,\theta,w,H,\iota\).

\subsection{Exact transport and tag expansion}

Put \(d_H=\ceil{H/w_0}\) and \(\delta_j=\lambda_j(\omega-\Omega)\).
For each row \(r=(j,k,\beta)\) and selected column \(c=(a,b,\alpha)\),
define
\[
 \mathsf B(r,c)=
 [s^k]\binom{\alpha}{\beta}
    g_{d_H,a,b}(z_j+s)
       (z_j+s-z_T+\delta_j)^{\alpha-\beta},
\]
with zero when \(\beta>\alpha\).
The functions here are entire in the new base parameter \(s\).
The transverse coefficient indexed by \(\beta\) has already been extracted.

\begin{lemma}[Literal determinant identity]\label{pow:lem:transport}
For every allowed selection \(\iota\),
\[
 |\det\mathsf B|=
       \left(\prod_{j=1}^K|S(z_j)|^{d_H|\BB|}\right)|\del|.
\]
Consequently \(|\del|\leq U_*^{d_HM_H}|\det\mathsf B|\).
\end{lemma}
\begin{proof}
First remove the logarithm truncation by the determinant-one packet
shear from Lemma~\ref{pow:lem:minors}. The base truncation is coefficientwise
exact. Next substitute
\(t=y(z_T+s)\).
Its linear coefficient is \(-1\), by \(y'(z_T)=-1\).
It is a triangular change of every packet, with diagonal
\((-1)^k\), and \eqref{pow:eq:germ} gives
\(\ell(t)=s\).
Periodicity identifies the elliptic coordinates at \(z_T+s\) and
\(z_j+s\), while
\[
 \lambda_j\omega+s=z_j+s-z_T+\delta_j.
\]
Finally multiply the block of rows at \(j\) by the Taylor series of
\(S(z_j+s)^{d_H}\).
Multiplication is triangular on the packet, with diagonal
\(S(z_j)^{d_H}\) repeated \(|\BB|\) times.
Since \(d_H\geq2a+3b\) for every source column, it produces exactly the
entire function defining \(\mathsf B\).
Multiplying these determinants proves the identity. No source columns
were changed.
\end{proof}

Let \(\mathcal A_H=\{0,\ldots,\floor{H/w}\}\), \(T_H=|\mathcal A_H|\).
For \(\tau\in\mathcal A_H\), define the entire tag column
\[
 G_{c,\tau}(z)=\binom{\alpha}{\tau}
          g_{d_H,a,b}(z)(z-z_T)^{\alpha-\tau};
\]
it is zero if \(\tau>\alpha\).
The binomial identity
\[
 \binom{\alpha}{\beta}(X+\delta)^{\alpha-\beta}
  =\sum_{\tau=\beta}^{\alpha}
           \binom{\tau}{\beta}\delta^{\tau-\beta}
                     \binom{\alpha}{\tau}X^{\alpha-\tau}
\]
gives an exact finite row expansion. Thus
\begin{equation}\label{pow:eq:detexpansion}
 \det\mathsf B
 =\sum_{g:\RR\to\mathcal A_H}
       \left(\prod_{r=(j,k,\beta)}
          \binom{g(r)}{\beta}\delta_j^{g(r)-\beta}\right)
                                          \det\mathsf B_g,
\end{equation}
where a scalar is zero if \(g(r)<\beta\), and
\(\mathsf B_g(r,c)=[s^{k_r}]G_{\iota(c),g(r)}(z_j+s)\).

Since \(w\leq\log q+1\) and \(|\lambda_j|\leq\sqrt2N\),
\[
 |\delta_j|\leq e^{C_\delta-\nu w},\qquad
 C_\delta=\log(\sqrt2N)+\nu.
\]
For a nonzero scalar product in \eqref{pow:eq:detexpansion},
\(g(r)\geq\beta_r\) for every row, and its absolute value is at most
\begin{equation}\label{pow:eq:scalar}
 \exp\left(-\nu\left(w\sum_rg(r)-M_HH\Wbar\right)
                   +\frac{M_HH(\log2+C_\delta)}w\right).
\end{equation}

\subsection{Collision estimate and the finite-cutoff upper bound}

Set
\[
 r_*=2(|z_T|+\sqrt2\Omega+3/4),\qquad R_N=r_*N,\qquad
 \kappa=\frac{\log2}{4},\qquad s_0=2^{-1/2}.
\]
Then \(R_N\geq1\) and \(|z_j|+3/4\leq R_N/2\).
For \(|z|\leq R\), \eqref{pow:eq:ggrowth} and
\(\alpha\leq H/w\) give
\[
 |G_{c,\tau}(z)|\leq e^{H\mathcal H_H(R)},\qquad
 \mathcal H_H(R)=\frac{d_H}{H}
       \left(\log C_g+\pi(R/\Omega+1)^2\right)
           +\frac{\log(2(R+|z_T|+1))}{w}.
\]

\begin{lemma}[Collision]\label{pow:lem:collision}
For \(n_\tau=|g^{-1}(\tau)|\),
\[
 |\det\mathsf B_g|\leq
 \exp\left(-\kappa\sum_\tau n_\tau^2+
   M_HH\left(\mathcal H_H(R_N)+\frac{\log2}{v_0}+\rho_H\right)\right),
 \quad
 \rho_H=\frac{\log M_H+\kappa-\log(1-s_0)}H.
\]
\end{lemma}
\begin{proof}
Write
\(G_{c,\tau}(z)=\sum_{h\geq0}a_{c,\tau,h}(z/R_N)^h\).
Cauchy's estimate gives \(|a_{c,\tau,h}|\leq e^{H\mathcal H_H(R_N)}\).
On \(|s|=1/2\), the choice of \(R_N\) gives
\(|z_j+s|\leq R_N/2\), so
\[
 \left|[s^k]\left(\frac{z_j+s}{R_N}\right)^h\right|
                               \leq2^k2^{-h}.
\]
Expand the determinant by these Taylor degrees, independently for each
row. If two rows with the same tag use the same Taylor degree, their
coefficient vectors are proportional, so that determinant term is zero.
Thus in every surviving term the degree sum is at least
\(\sum_\tau\binom{n_\tau}{2}\).
Split \(2^{-\sum h}\) into two equal factors. Bound one by this
minimum degree sum and sum the other as a product of geometric series.
The result is
\[
 M_H!\,e^{M_HH\mathcal H_H(R_N)}2^{\sum_rk_r}
   2^{-\frac12\sum_\tau\binom{n_\tau}{2}}(1-s_0)^{-M_H}.
\]
Absolute convergence justifies the expansions.
Use \eqref{pow:eq:rowbounds}, \(M_H!\leq M_H^{M_H}\), and
\(2\sum_\tau\binom{n_\tau}{2}=\sum_\tau n_\tau^2-M_H\).
\end{proof}

Fix \(0<\eta<1\) and \(\theta<A(1-\eta)<A<1\). Put
\[
 L_H=\floor{AH/w}+1,\qquad
 c_H=\frac{\kappa\eta^2M_H}{HL_H}.
\]
For any nonzero summand in \eqref{pow:eq:detexpansion}, there are two cases.
If \(w\sum_rg(r)\geq A(1-\eta)M_HH\), the scalar estimate
\eqref{pow:eq:scalar} supplies the term
\(-\nu(A(1-\eta)-\Wbar)\).
Otherwise more than \(\eta M_H\) rows have \(wg(r)\leq AH\);
there are at most \(L_H\) such tags. Cauchy--Schwarz gives
\(\sum_\tau n_\tau^2\geq\eta^2M_H^2/L_H\), so the collision lemma
supplies \(-c_H\). In this second case the scalar bound loses no
negative weight because \(g(r)\geq\beta_r\).

Define the fixed lattice constants
\begin{equation}\label{pow:eq:analyticconstants}
 C_0=\log(C_gU_*)+\pi(r_*/\Omega+1)^2,\qquad
 C_w=4\sqrt2(r_*+|z_T|+1).
\end{equation}

\begin{proposition}[Upper bound for the same algebraic-center minor]
\label{pow:prop:analytic}
If \(|\omega-\Omega|\leq q^{-\nu}\), every nonzero selected minor satisfies
\begin{equation}\label{pow:eq:upper}
 \frac{\log|\del|}{M_HH}
 \leq \Ean+e_H+
              \max\{-c_H,-\nu(A(1-\eta)-\Wbar)\},
\end{equation}
where
\begin{align}
 \Ean&=\frac{C_0K}{w_0}+\frac{\log2}{v_0}
                  +\frac{\nu+\log(C_wK)}w,\label{pow:eq:Ean}\\
 e_H&=\frac{C_0K}{H}+\rho_H+\frac{\log T_H}{H}.\label{pow:eq:eremainder}
\end{align}
For fixed finite data, \(e_H\to0\), \(\log M_H/H\to0\), and
\begin{equation}\label{pow:eq:climit}
 c_H\longrightarrow
       c_\infty=\frac{\kappa\eta^2K\theta}{2v_0A}.
\end{equation}
\end{proposition}
\begin{proof}
Combine the two cases above, the collision lemma, and
\eqref{pow:eq:scalar}. There are \(T_H^{M_H}\) assignments.
Lemma~\ref{pow:lem:transport} adds
\((d_H/H)\log U_*\). Now
\[
 \log(C_gU_*)+\pi(R_N/\Omega+1)^2\leq C_0K,\quad
 \frac{d_H}{H}\leq\frac1{w_0}+\frac1H.
\]
The remaining radius and binomial factors are bounded by
\[
 4\sqrt2N(R_N+|z_T|+1)\leq C_wK.
\]
This proves \eqref{pow:eq:upper}--\eqref{pow:eq:eremainder}.
Elementary triangle-lattice counting gives
\[
 |\BB|=\frac{\theta H^2}{2v_0w}+O(H+1),\qquad
 L_H=\frac{AH}{w}+O(1),\qquad T_H=\frac{H}{w}+O(1).
\]
For example summing
\(\ceil{(H-(w/\theta)\beta)/v_0}\) over
\(0\leq\beta<\theta H/w\) proves the first formula, with error \(O(H+1)\).
All data are fixed in these asymptotics.
Thus \(M_H\asymp H^2\); the asserted remainders vanish and
\eqref{pow:eq:climit} follows. Rationality of \(\omega\) was used nowhere
in this analytic proof.
\end{proof}

\section{Parameters and the final contradiction}

Let \(G\) be as in \eqref{pow:eq:mainC}. For \(n\geq1\), set
\begin{equation}\label{pow:eq:params}
 \begin{gathered}
 h_n=\ceil{\log(2n)},\qquad M_0=10^{10}G^2,\qquad
 \theta=\frac1{M_0nh_n},\qquad A=2\theta,\qquad \eta=\frac13,\\
 v_0=1000Gnh_n,\qquad
 N=\ceil{\sqrt{1000v_0}},\quad K=N^2,\qquad w_0=8GK,\\
 \nu_0=\frac6\theta=6\cdot10^{10}G^2nh_n.
 \end{gathered}
\end{equation}
These are admissible rational base data and an exact Gaussian-square
grid cardinality. None depends on a rational approximation.

\begin{lemma}[All fixed margins]\label{pow:lem:params}
The data \eqref{pow:eq:params} satisfy
\[
 0<\theta<A(1-\eta)<A<1,\qquad K\theta<1,\qquad
 K(w_0/v_0)\theta<1,\qquad c_\infty>4.
\]
Moreover
\begin{equation}\label{pow:eq:fixederrors}
 \frac{n(L_\theta+\log2)}{v_0}
        +\frac{C_0K}{w_0}+\frac{\log2}{v_0}<\frac13,
 \qquad
 \nu_0(A(1-\eta)-\theta)=2.
\end{equation}
\end{lemma}
\begin{proof}
Since \(v_0\geq1000\),
\[
 1000\leq\frac K{v_0}
 \leq1000+2\sqrt{1000/v_0}+1/v_0<1003.
\]
Consequently
\[
 K\theta<\frac{1003}{10^7G}<1,\qquad
 K(w_0/v_0)\theta
   <\frac{8\cdot1003^2}{10^7}=0.8048072<1.
\]
The order inequalities for \(A,\theta,\eta\) follow directly.
With \(A=2\theta,\eta=1/3\), \eqref{pow:eq:climit} becomes
\(c_\infty=\kappa K/(36v_0)\).
The elementary inequality \(\log2>2/3\) implies \(\kappa>1/6\), hence
\(c_\infty>1000/216>4\).
For instance \(\log2>2/3\) follows by integrating the strict tangent
lower bound for the convex function \(1/(1+t)\) at \(t=1/2\).

Because \(h_n\geq1\), \(\log(2n)\leq h_n\), \(\log h_n\leq h_n\),
\(\log G\leq G\), \(\log10<5/2\), and \(\log2<1\), we have
\begin{align*}
 L_\theta+\log2
 &=13+\log2+
       6\log(2\cdot10^{10}G^2nh_n)\\
 &<164+12G+12h_n
 \leq188Gh_n<200Gh_n.
\end{align*}
Thus the three terms in the first expression in
\eqref{pow:eq:fixederrors} are bounded respectively by
\(1/5,1/8,1/1000\), whose sum is \(163/500<1/3\).
Finally \(A(1-\eta)-\theta=\theta/3\), proving the last equality.
\end{proof}

\begin{proof}[Proof of Theorem~\ref{pow:thm:main}]
Fix \(n\geq1\) and the data \eqref{pow:eq:params}.
Suppose, for contradiction, that there are arbitrarily large integers
\(q\) with some \(p\in\Z\) satisfying \eqref{pow:eq:badapprox} for
\(\nu=\nu_0\). Choose one such approximation so large that
\(w=\ceil{\log q}\) exceeds the EPA threshold, the threshold in
Lemma~\ref{pow:lem:Dsize}, and a threshold ensuring
\begin{equation}\label{pow:eq:werrors}
 \frac{n\log(C_{\mathrm{conj}}K)+\theta+
                         \nu_0+\log(C_wK)}w<\frac16.
\end{equation}
The root \(\omega=(p/q)^{1/n}\) satisfies \eqref{pow:eq:root}.
All base parameters were fixed before this choice.

For these now fixed finite data, Lemma~\ref{pow:lem:minors} gives cofinally
many \(H\) with a nonzero selected minor. By
Proposition~\ref{pow:prop:analytic}, choose such an \(H\) large enough that
\begin{equation}\label{pow:eq:Herrors}
 c_H>3,\qquad e_H+\frac{n\log M_H}{H}<\frac16.
\end{equation}
Set
\[
 V_H=\frac{\log|\del|}{M_HH},\qquad
 \mathcal E=\Ear+\Ean+e_H+\frac{n\log M_H}{H}.
\]
Equations \eqref{pow:eq:fixederrors}, \eqref{pow:eq:werrors},
\eqref{pow:eq:Herrors} give \(\mathcal E<2/3\).
The lower and upper estimates \eqref{pow:eq:lower}, \eqref{pow:eq:upper}
cannot both hold. To see this, consider the branch attaining the maximum.

On the collision branch they imply
\[
 c_H\leq1-\Wbar+\mathcal E\leq1+\mathcal E<5/3,
\]
contrary to \(c_H>3\).

On the displacement branch they imply
\[
 \nu_0(A(1-\eta)-\Wbar)-(1-\Wbar)\leq\mathcal E.
\]
However \(\nu_0\geq1\), \(0\leq\Wbar\leq\theta\), and
\eqref{pow:eq:fixederrors} give
\begin{align*}
 \nu_0(A(1-\eta)-\Wbar)-(1-\Wbar)
 &=2+(\nu_0-1)(\theta-\Wbar)-1+\theta\\
 &\geq1,
\end{align*}
again a contradiction.

Thus the bad denominators are bounded, proving the displayed strict
eventual lower bound with exponent \(\nu_0=Cnh_n\).
It applies to all representations \(p/q\), not just reduced ones.
If \(\Omega^n=a/b\in\Q\), the representations
\((ak)/(bk)\) would contradict it for arbitrarily large \(k\).
Hence \(\Omega^n\) is irrational, and the definition of the irrationality
exponent gives \(\mu(\Omega^n)\leq\nu_0\).
Since \(h_n\leq\log(2n)+1\), this is \(O(n\log(2n))\).
\end{proof}

\section{Dependency and quantifier audit}

The order of choices is
\[
 \text{fixed lattice constants};\quad
 n;\quad(\theta,v_0,w_0,N,A,\eta,\nu_0);\quad
 \text{weight thresholds};\quad(p,q,\omega,w);\quad H;\quad\iota .
\]
Here the last two choices can equivalently be made together using a
cofinal nonzero-minor assertion. All estimates are uniform in
\(\iota\), so the analytic estimate does not require a different minor.
The clearing estimate is uniform in \(H\) after \(w\) is chosen.
The limits in \(H\) are only taken after \(q,w,\omega\) are fixed.
The auxiliary radical degree \(d_2\) cancels exactly; the only remaining
degree bound is \(d\leq n\).
The numerator \(p\), of approximate size \(q\Omega^n\), introduces no
new archimedean height term because every conjugate of its ratio root
has modulus \(\omega\leq\Omega+1\).

The inherited mathematical inputs are the actual-period normalization
and EPA as stated in Section~2. The analytic extension to algebraic
centers, coefficient clearing, relative norm, conjugate estimate,
counting, parameter margins, and contradiction are proved in this note.
The theta identities are classical and were also derived from the
displayed series. The complete checked power endpoint is now
\lean{QuadraticPi.LemniscatePowerExponentConclusion.lemniscate_power_main},
including a single fixed positive constant, strict eventual bounds for
all denominators, irrationality and the ordinary exponent bound.

The finite experiments accompanying the earlier working calculation
are consistency checks only. They are not premises of this proof.
The quasi-period is not included in the theorem here; extending the
argument to it requires writing out its second-kind analytic transport.

\clearpage
\appendix
\section{Current formalization ledger and exact endpoints}
\label{sec:ledger}

This appendix records the state through verification checkpoint 170,
October 21, 2026. The original consolidated manuscript date remains
October 20. The complete original source and PDF are retained in the
recovery set; earlier status notes describe their own historical runs.
The present update preserves every mathematical theorem, proposition,
lemma, corollary, definition and proof environment verbatim.

\subsection{The real CM period endpoint}
The public paper-facing declaration is
\begin{quote}
\lean{QuadraticPi.EllipticCMPaperNormalization.paper_period_main_of_noninteger_CM}.
\end{quote}
It accepts an actual rational Weierstrass curve $W$ with
\lean{W.IsElliptic}, a full oriented period pair $L$ normalized to
Mathlib's short model of $W$, a classical theta simple-zero input,
an actual lattice-preserving multiplier outside $\Z$, a nonzero
rational differential factor $r$, and a nonzero real member $\xi$
of the full scaled lattice $rL$. It proves
\[
 \operatorname{Irrational}(\xi),\qquad \mu(\xi)=2,
\]
and the actual all-denominator eventual lower bound for every
$\nu>2$. Short-model nonsingularity, nonrealness of the CM multiplier,
integer coefficient normalization and the compatible rational scale
are derived, not extra certificates supplied to the endpoint.
Section~\ref{sec:classical-boundary} identifies the classical
existence and period-identification inputs that remain explicit.

For the integral normalized model, the strict all-denominator bound
is \lean{QuadraticPi.EllipticCMExponentConclusion.strict_eventual_period_bound}.
The actual contradiction preceding it is
\lean{QuadraticPi.EllipticCMNormContradiction.no_unbounded_approximations}.
Neither theorem assumes a nonzero minor, an analytic upper estimate,
a rank statement, or an irrationality conclusion.

\subsection{The power endpoint and its constant}
The complete retained declaration is
\lean{QuadraticPi.LemniscatePowerExponentConclusion.lemniscate_power_main}.
It proves positivity of a single fixed constant and, for every $n>0$,
the strict all-denominator bound, irrationality of $\Omega^n$, and
\[
 2\leq\mu(\Omega^n)\leq
 Cn\lceil\log(2n)\rceil,\qquad C=6\cdot10^{10}G^2.
\]
Here the actual fixed integer $G$ is
\lean{QuadraticPi.LemniscatePowerFixedParameters.growthIndex}, and
$C$ is \lean{QuadraticPi.LemniscatePowerExponentConclusion.powerBoundConstant}.
The endpoint uses the actual integral-defined
\lean{QuadraticPi.Lemniscate.Omega}; a free period-normalization or
EPA hypothesis is not an input to this public power declaration.
It remains separate from the general-CM origin chart.

\subsection{Verified stages of the CM argument}
The table names actual application-level conclusions; the exact types,
including every hypothesis, are included in the checkpoint audit.
\begin{longtable}{@{}>{\raggedright\arraybackslash}p{.28\textwidth}>{\raggedright\arraybackslash}p{.64\textwidth}@{}}
\toprule
Stage & Checked declaration\\
\midrule
\endhead
Quartic persistence and matched comparison &
\lean{QuadraticPi.EllipticOriginQuarticPersistence.exists_bounded_persistent_comparison}.\\[4pt]
Actual curve inequality &
\lean{QuadraticPi.EllipticOriginQuarticCurveInequality.weighted_curve_inequality}.\\[4pt]
Actual compactification and bounded EPA &
\lean{QuadraticPi.EllipticOriginQuarticEPA.epa}.\\[4pt]
Exact full/truncated determinant &
\lean{QuadraticPi.EllipticOriginQuarticLogDeterminant.det_fullSelectedMinor_eq_truncated}.\\[4pt]
Actual full-minor arithmetic lower &
\lean{QuadraticPi.EllipticOriginQuarticLogDeterminant.fullSelectedMinor_arithmetic_lower}.\\[4pt]
Same-selected-column analytic transport &
\lean{QuadraticPi.EllipticOriginMinorTransport.norm_analyticMinor_eq}.\\[4pt]
Finite analytic bound &
\lean{QuadraticPi.EllipticOriginPaperUpper.same_minor_paper_upper}.\\[4pt]
Actual cofinal nonzero full minors &
\lean{QuadraticPi.EllipticCMCofinalMinors.cofinal_full_minors}.\\[4pt]
Norm contradiction &
\lean{QuadraticPi.EllipticCMNormContradiction.no_unbounded_approximations}.\\[4pt]
Rational coefficient endpoint &
\lean{QuadraticPi.EllipticCMRationalPeriod.rational_period_main}.\\[4pt]
Curve/differential/full-lattice endpoint &
\lean{QuadraticPi.EllipticCMPaperNormalization.paper_period_main_of_noninteger_CM}.\\
\bottomrule
\end{longtable}
The finite analytic theorem uses the same selected source columns and
complete strict packets as the arithmetic determinant. The full-lattice
shifted-theta clearing in Lean is an alternative to the sigma presentation
in the written argument; its required identities and estimates are proved
directly. No identification of the two clearing functions is assumed.
All counting limits are taken after the finite data are fixed.

\section{Classical inputs and the specified differential}
\label{sec:classical-boundary}

The new argument is formally verified \emph{relative to explicitly
displayed classical inputs}. This qualification has a precise meaning.
The checked endpoint does not internally construct a complex period
lattice by integrating an abstract algebraic differential, or prove the
classical theta simple-zero theorem. Those are standard mathematical
inputs, not proposed new results or hidden rank/analytic assumptions.

\paragraph{Uniformization and CM lifting.}
Specified-differential complex uniformization supplies the full lattice
for the short model and the identities $g_2=-4A$, $g_3=-4B$; an oriented
basis can be chosen. A CM endomorphism outside integer multiplications
lifts to an actual lattice-preserving scalar outside $\Z$.
The cited statements are Milne, Chapter III, Proposition 3.3,
Theorem 3.10 and Remark 3.11 \cite{milne}. Milne uses $y=\wp'$ and
$dx/y$; in the short model here $y=\wp'/2$, so the corresponding
differential is $dx/(2y)$. Rational differential rescaling identifies
its full period group with $r\Lambda$.

\paragraph{Theta simplicity and notation.}
Mathlib's two-variable function \(\texttt{jacobiTheta}_{2}\) corresponds to
conventional $\theta_3(\pi z\mid\tau)$, not to conventional
$\theta_2(z\mid\tau)$. The shifted multiplier used in the checked
clearing is
\[
 S_L(z)=\theta_3\!\left(
 \pi\left(\frac{z}{\omega_1}-\frac{1+\tau}{2}\right)\Bigm|\tau\right),
 \qquad \tau=\frac{\omega_2}{\omega_1},\quad\Im\tau>0.
\]
The classical half-period translation and derivative product formulas
\cite[(20.2.11), (20.4.2)]{dlmf} imply $S_L'(0)\ne0$:
the half-period zero corresponds to the simple zero of $\theta_1$
at zero, with nonzero exponential and $\pi/\omega_1$ factors.
This classical implication remains an explicit input, not a new
internally proved Lean theorem. Its exact transport under lattice
scaling is checked in
\lean{QuadraticPi.EllipticCMThetaScale.simple_zero_scale_iff}.

\paragraph{Normalization is checked.}
Mathlib's existing rational short-normal-form change is reused.
The actual integral-model construction and differential coefficient
factor are
\lean{QuadraticPi.EllipticRationalIntegralModel.exists_integral_short_model}
and
\lean{QuadraticPi.EllipticRationalIntegralModel.invariant_differential_factor}.
For $D=\operatorname{den}(A)\operatorname{den}(B)>0$, denominator clearing
has $A_1=D^4A$, $B_1=D^6B$ and new lattice $D^{-1}\Lambda$.
The original canonical differential is $D$ times the new one.
The checked endpoint transports actual invariants, orientation,
CM action, theta derivative and the nonzero real period, applies the
integral theorem, then inverts the rational scale. Arbitrary algebraic
scaling is not used. Every nonzero real member of the full rationally
scaled lattice is covered; a primitive positive real period is unnecessary.

\section{Build evidence, declaration types and audit scope}
\label{sec:p1-audit}

The current pinned environment uses Lean 4.34.1 and mathlib commit
\lean{d13f23b723b8a846827a245b89c10fc7d3f11612}, with nine pinned
dependencies. The official compiler binary is checked by SHA-256
in the verification scripts. Its digest, package pins, exact commands,
logs, declaration types and transitive dependency reports accompany
the cumulative recovery checkpoint.

Checkpoint 169 checked 25 new theorem types, four definitions and
37 transitive dependency reports in 13 focused runs, with zero
warnings and all 3,172 prior guarded files unchanged. Checkpoint 168
records the actual CM norm-contradiction assembly; checkpoint 149
records the complete power endpoint. The current checkpoint 170
rechecks the named paper-reference interfaces and their exact types
and transitive dependencies. It introduces no new mathematical Lean
source theorem and does not count editorial changes as new proofs.

The permitted reported axioms are
\lean{propext}, \lean{Classical.choice} and \lean{Quot.sound}.
No custom unproved axiom or placeholder proof is added. These
standard-only dependency reports do not discharge the classical
hypotheses visibly supplied to the CM endpoint. They also do not
constitute a claim of an empty-cache or clean-room rebuild.
The full cumulative rebuild evidence at checkpoint 154 and focused
subsequent evidence are retained with their original scopes.

The October 12 emergency and reference audits are historical evidence,
not current completion claims. Their original ledgers and counts,
and the October 20 source-integrity checks, are preserved unchanged.
This version updates their former prospective status statements in
light of the checked application-level endpoints now available.

\section{Checked power stages and remaining scope}
\label{sec:new-formal-plan}

The former prospective power obligations are now represented by actual
checked results. The coefficient-clearing theorem is
\lean{QuadraticPi.LemniscatePowerCoefficientClearing.paper_prime_wise_clearing}.
The relative norm and its auxiliary-radical cancellation are checked in
\lean{QuadraticPi.LemniscatePowerNormDescent.descended_log_norm_lower}.
Actual selected-minor relative-norm and normalized lower bounds are
\lean{QuadraticPi.LemniscatePowerRelativeNorm.selectedMinor_relative_norm_lower}
and
\lean{QuadraticPi.LemniscatePowerNormalizedNormLower.normalized_lower_of_product}.
The fixed parameters include
\lean{QuadraticPi.LemniscatePowerFixedParameters.fixed_error_margin}.
The actual final obstruction is
\lean{QuadraticPi.LemniscatePowerNormContradiction.exists_weight_threshold_no_power_approximation},
leading to the power endpoint named in Appendix~\ref{sec:ledger}.
The theorem applies to all sufficiently large denominators, without
a coprimality requirement on the displayed $p/q$.

Under the stated use of established classical inputs, the new CM and
power arguments are verified at the exact boundaries just described.
A completely internal implementation of curve-to-lattice uniformization,
differential integration and theta simplicity would be additional
foundational work. Non-CM periods, arbitrary algebraic rescaling,
general CM-period power bounds and quasi-periods are not claimed by
these endpoints. No completion percentage is inferred from module counts.

\section{Version, provenance and preservation}
\label{sec:consolidation}
This version retains the October 20 consolidated paper's title,
manuscript date and mathematical theorem/proof environments. David V.
Chudnovsky has been added as coauthor for submission on October 21, 2026. The revised abstract,
formalization notes and appendices describe the checked state through
October 21. Both source versions, their hashes, an editorial change
record, the current paper-reference audit, compilation logs and restart
instructions accompany the cumulative checkpoint. No historical source
or build evidence is overwritten.

The current record distinguishes mathematical proof, kernel-checked
application-level Lean proof, and cited classical input. Internal
verification and source/layout checks do not replace independent
mathematical scrutiny. The exact public theorem hypotheses determine
the formal claim; a compiled conditional interface alone would not
prove its application hypotheses.

\section*{Acknowledgments and provenance}
The separated-weight interpolation program was inspired by OpenAI's
manuscript \cite{openai}. OpenAI ChatGPT assisted with exploration,
proof exposition, Lean development, manuscript consolidation, and
internal reviews. Gregory V. Chudnovsky and David V. Chudnovsky are
the human authors responsible for the present manuscript. The older drafts and their
verification records remain available in the research recovery set;
historical verification counts describe those recorded runs and do
not certify a stronger theorem than the exact declarations checked.

\end{document}